\documentclass[10pt]{amsart}
\usepackage{amssymb}
\usepackage{amsbsy}
\usepackage{amscd}
\usepackage[mathscr]{eucal}
\usepackage{verbatim}
\usepackage{version}
\usepackage{color}
\usepackage[matrix,arrow]{xy}
\usepackage{graphicx}

\def\cal{\mathcal}
\def\Bbb{\mathbb}

\newenvironment{NB}{
\color{red}{\bf NB}. \footnotesize 
}{}
\excludeversion{NB}
\newenvironment{NB2}{
\color{blue}{\bf NB}. \footnotesize
}{}

\newcommand{ \Supp}{\operatorname{Supp}}
\newcommand{\Ext}{\operatorname{Ext}}
\newcommand{\Hom}{\operatorname{Hom}}
\newcommand{\codim}{\operatorname{codim}}
\newcommand{\im}{\operatorname{im}}

\newcommand{\rk}{\operatorname{rk}}

\newcommand{\chr}{\operatorname{char}}

\newcommand{\NS}{\operatorname{NS}}
\newcommand{\coker}{\operatorname{coker}}
\newcommand{\Pic}{\operatorname{Pic}}
\newcommand{\ch}{\operatorname{ch}}

\newcommand{\Alb}{\operatorname{Alb}}
\newcommand{\Hilb}{\operatorname{Hilb}}

\newcommand{\Coh}{\operatorname{Coh}}

\newcommand{\WIT}{\operatorname{WIT}}

\newcommand{\Div}{\operatorname{Div}}

\newcommand{\tr}{\operatorname{tr}}

\newcommand{\Mov}{\operatorname{Mov}}
\newcommand{\Nef}{\operatorname{Nef}}

\font\b=cmr10 scaled \magstep5
\def\bigzerou{\smash{\lower1.7ex\hbox{\b 0}}}

\numberwithin{equation}{section}
\theoremstyle{plain}
 \newtheorem{thm}{Theorem}[section]
 \newtheorem{lem}[thm]{Lemma}
 \newtheorem{prop}[thm]{Proposition}
 \newtheorem{cor}[thm]{Corollary}

\theoremstyle{definition}
 \newtheorem{defn}[thm]{Definition}
 
\theoremstyle{remark}
 \newtheorem{rem}[thm]{Remark}
 \newtheorem{ex}[thm]{Example}
 
\begin{document}

\title{Wall crossing for moduli of stable sheaves on an elliptic surface}
\author{K\={o}ta Yoshioka}
\address{Department of Mathematics, Faculty of Science,
Kobe University,
Kobe, 657, Japan
}
\email{yoshioka@math.kobe-u.ac.jp}

\thanks{
The author is supported by the Grant-in-aid for 
Scientific Research (No. 18H01113), JSPS}

\begin{abstract}
We shall study the wall crossing behavior of moduli of stable sheaves on an elliptic
surface.
\end{abstract}

\maketitle

\renewcommand{\thefootnote}{\fnsymbol{footnote}}

\section{Introduction}

In \cite{F1} and \cite{F2},
Friedman systematically studied 
moduli spaces of stable sheaves of rank 2 on elliptic surfaces.
In particular he proved that the moduli spaces are birationally equivalent to the 
Hilbert scheme of points, if the relative degree of the first Chern class
is odd.
A few years later, Bridgeland \cite{Br:1} generalized Friedman's results
to higher rank cases by constructing 
relative Fourier-Mukai transforms associated to the elliptic fibration. 
This is a fundamental tool for the study of coherent sheaves on elliptic surfaces, and
many properties of moduli spaces are proved (cf. \cite{BBH}, \cite{BH}, \cite{HM}, \cite{JM}, \cite{Y:7}).  
In \cite{Y:twist2} and \cite{Y:Enriques}, we studied the Hodge numbers and the Picard groups
of the moduli spaces under the assumption that all fibers are irreducible.
In this paper, we study the cases where elliptic surfaces have
reducible fibers.

Let $\pi:X \to C$ be an elliptic surface 
with reducible fibers over an algebraically closed field $k$.
For an ample divisor $H$, let $M_H(r,\xi,a)$ be the moduli space
of stable sheaves $E$ of $(\rk E,c_1(E),\chi(E))=(r,\xi,a)$ with respect to $H$.
If $n \gg 0$, then as Friedman \cite{F1} first noticed, 
$M_{H+nf}(r,\xi,a)$ is independent of the choice of $n$.
We denote $M_{H+nf}(r,\xi,a)$ $(n \gg 0)$ by $M_{H_f}(r,\xi,a)$.
Assume that $\gcd(r,(\xi \cdot f))=1$.
Then $M_{H_f}(r,\xi,a)$ is a smooth projective variety and
Bridgeland proved it is birationally equivalent
to a moduli space of torsion free sheaves of rank 1
on an elliptic surface.  
In this note, we shall study the birational correspondences of 
 $M_{H_f}(r,\xi,a)$ under various change of stability conditions.
We shall study two types of wall crossing behaviors:
\begin{enumerate}
\item[(a)]
Wall crossing behaviors for twisted Gieseker stability of semi-stable 1-dimensional
sheaves. 
\item[(b)]
Wall crossing behaviors for stability conditions of coherent sheaves with torsions.
\end{enumerate}

For case (a), we note that
$M_{H_f}(r,\xi,a)$ is isomorphic to 
a moduli space $M^\alpha_H(0,\xi',a')$ of twisted stable 1-dimensional sheaves on $X$, if
there is a section of $\pi$ \cite{PerverseII}.
Since the twisted stability depends on the choice of an ample divisor $H$
and a twisting parameter $\alpha \in \NS(X)_{\Bbb Q}$,
it is important to study the wall-crossing behavior under the deformation of
$(H,\alpha)$.

For case (b), we shall introduce a new family of stability conditions
which corresponds to Gieseker stability
under Fourier-Mukai transforms.
We shall study the wall crossing  behaviors for this stability condition
and get a more precise
structure for the birational correspondence of Bridgeland \cite{Br:1}. 
Thus we can estimate the dimension of 
the loci where the birational map is not defined.

\begin{thm}[{Theorem \ref{thm:m=2}}]\label{thm:birat}
Assume that the multiplicity of all multiple fibers are two. 
For $(r,\xi, a) \in {\Bbb Z} \oplus \NS(X) \oplus {\Bbb Z}$
such that $r>0$ and $\gcd(r,(\xi \cdot f))=1$,
there is a (contravariant) Fourier-Mukai transform
$\Phi:{\bf D}(X) \to {\bf D}(Y)$
which induces an isomorphism
\begin{equation}
\begin{matrix}
M_{H_f}(r,\xi,a) \setminus Z & \to & M_{H_f}(1,0,a') \setminus Z'\\
E & \mapsto & \Phi(E),
\end{matrix}
\end{equation}
where 
\begin{enumerate}
\item
$\dim M_{H_f}(r,\xi,a)=\dim M_{H'_f}(1,0,a')$,
\item
$Z \subset M_{H_f}(r,\xi,a)$ is a closed subscheme of
$\dim Z \leq \dim M_{H_f}(r,\xi,a) -2$ and
\item
$Z' \subset M_{H'_f}(1,0,a')$ is a closed subscheme of $\dim  Z' \leq \dim M_{H'_f}(1,0,a')-2$.
\end{enumerate}
\end{thm}
As an application, we can compute the Picard group of $M_{H_f}(r,\xi,a)$ (Theorem \ref{thm:Pic}).

Let us explain the organization of this paper.
In section \ref{sect:1-dim}, we shall explain several properties 
of stable 1-dimensional sheaves.
In particular we shall explain the existance condition of stable sheaves 
supported on fibers.
We also explain some results on Fourier-Mukai transforms on elliptic surfaces.
In section \ref{sect:wall-crossing}, we shall treat the case (a), thus
we shall study the wall-crossing behavior
for the moduli spaces $M_H^\alpha(0,\xi',a')$ of stable 1-dimensional sheaves 
under deformations of $(H,\alpha)$, where $(\xi' \cdot f)=1$.
We first explain the wall and chamber structure in the space of pairs
$(H,\alpha)$. We then classify the walls and compare two moduli spaces separated by a
single wall.
In particular, we shall prove that the birational equivalence class is independent 
of the choice of a general parameter $(H,\alpha)$ (Theorem \ref{thm:birat-1dim}). 
This is much simpler than the wall-crossing in \cite{BM}.

In sections \ref{sect:stability} and \ref{sect:another}, we shall treat the case (b).
We first introduce a new stability condition for coherent sheaves with torsion and
study its property. In particular we shall relate the stability with Gieseker stability
via a relative Fourier-Mukai transform.
We next study the wall-crossing behaviors.
In particular we shall prove Theorem \ref{thm:birat} (see Theorem \ref{thm:m=2}).
It may be interesting to understand our stability in terms of Bridgeland stability conditions \cite{Br:3}.
In section \ref{sect:application},
we shall study line bundles and the Picard group of $M_{H_f}(r,\xi,a)$.
We also derive the Hodge numbers of $M_{H_f}(r,\xi,a)$ from
the Hodge numbers of Hilbert scheme of points which was 
computed by G\"{o}ttsche and Soergel \cite{GS}.

{\bf Notation.}

For smooth projective varieties $X, Y$ and ${\bf P} \in {\bf D}(X \times Y)$,
$\Phi_{X \to Y}^{{\bf P}}:{\bf D}(X) \to {\bf D}(Y)$ is an integral functor defined by
\begin{equation}
\Phi_{X \to Y}^{{\bf P}}(E):={\bf R}p_{Y*}({\bf P} \otimes p_X^*(E)),\; E \in {\bf D}(X),
\end{equation}
where $p_X:X \times Y \to X$ and $p_Y:X \times Y \to Y$ are projections.
If $\Phi_{X \to Y}^{{\bf P}}$ is an equivalence, then it is called a Fourier-Mukai transform.
For the equivalence $\Phi_{X \to Y}^{{\bf P}}$,
\begin{equation}
\begin{split}
\Phi_{X \to Y}^{{\bf P}} \circ \Phi_{Y \to X}^{{\bf P}^{\vee}}=& \otimes p_Y^*({\cal O}_Y(-K_Y))[-n],\\
\Phi_{Y \to X}^{{\bf P}^{\vee}} \circ \Phi_{X \to Y}^{{\bf P}}=& \otimes p_X^*({\cal O}_X(-K_X))[-n]
\end{split}
\end{equation}
where $n=\dim X=\dim Y$.

For $E \in {\bf D}(X)$, we set 
\begin{equation}
D_X(E):={\bf R}{\cal H}om_{{\cal O}_X}(E,{\cal O}_X) \in {\bf D}(X).
\end{equation}
Then $D_X$ defines a contravariant functor ${\bf D}(X) \to {\bf D}(X)$. 

Let
$\pi:X \to C$ be an elliptic surface over a smooth projective curve 
$C$ of genus $g$.
We assume that $R^1 \pi_* {\cal O}_X \not \cong {\cal O}_C$.
We also assume that all multiple fibers are tame.
Then $\chi({\cal O}_X)=e$, $q(X)=g$ and the canonical bundle formula says
$$
K_X \equiv 
(2g-2+e)f +\sum_{i=1}^s (m_i-1)f_i \mod \Pic^0(X),
$$
where $e \geq 0$ and $m_1 f_1,...,m_s f_s$ are multiple fibers.
\begin{NB}
Since $R \pi_* {\cal O}_X={\cal O}_C-R^1 \pi_* {\cal O}_X$,
$-\deg R^1 \pi_* {\cal O}_X=e$. 
\end{NB}
\begin{NB}
$p_g=e+g-1$.
\end{NB}
If there is a section $\sigma$, then
$(\sigma^2)=-e$.
\begin{rem}\label{rem:fiber-cond}
If $R^1 \pi_* {\cal O}_X \cong {\cal O}_C$, then all singular fibers are elliptic curves
with multiplicities.
In this case, the Picard groups are treated in \cite[Appendix]{Y:Enriques}
for example. 
\end{rem}

For a purely 1-dimensional sheaf $E$ on $X$, 
we take a locally free resolution
$$
0 \to V_1 \overset{f}{\to} V_0 \to E \to 0.
$$
We denote an effective divisor $\det f$ by $\Div(E)$. 
Then $E$ is an ${\cal O}_{\Div(E)}$-module and 
the algebraic equivalence class of $\Div(E)$ is $c_1(E)$.
For $\alpha \in \NS(X)_{\Bbb Q}$,
$\chi_\alpha(E):=\chi(E)-(c_1(E) \cdot \alpha)$ denotes the $\alpha$-twisted
Euler-characteristic of $E$.

\begin{defn}\label{defn:reflection}
\begin{enumerate}
\item[(1)]
An object $E_1 \in {\bf D}(X)$ is spherical if 
\begin{equation}
\begin{split}
\Hom(E_1,E_1) \cong & k,\\
\Ext^1(E_1,E_1) \cong & 0,\\
\Ext^2(E_1,E_1) \cong & k.
\end{split}
\end{equation}
\item[(2)]
For a spherical object $E_1$,
let $R_{E_1}: {\bf D}(X)  \to  {\bf D}(X)$ be an equivalence defined by
$$
R_{E_1}(E):=\mathrm{Cone}({\bf R}\Hom(E_1,E) \otimes E_1 \to E).
$$
\end{enumerate}
\end{defn}

Let $K(X)$ be the Grothendieck group of $X$ and  
\begin{equation}
\begin{matrix}
\tau:& K(X) & \to & {\Bbb Z}\oplus \NS(X) \oplus {\Bbb Z}\\
  & E & \mapsto & (\rk E,c_1(E),\chi(E))
\end{matrix}
\end{equation}
a surjective homomorphism such that
$\tau(E)$ represents the topological equivalence class of $E$.
For ${\bf e}_1,{\bf e}_2 \in  {\Bbb Z}\oplus \NS(X) \oplus {\Bbb Z}$,
we set
\begin{equation}
\chi({\bf e}_1,{\bf e}_2):=\chi(E_1,E_2),\;\; {\bf e}_i=\tau(E_i),\; (i=1,2).
\end{equation}


\section{Moduli spaces of stable 1-dimensional sheaves}\label{sect:1-dim}

A twisted semi-stability was introduced by Matsuki and Wentworth 
\cite{M-W} for torsion
free sheaves on surfaces, and generalized to purely 1-dimensional 
sheaves in \cite{Y:twist2}. Let us recall the definition.
\begin{defn}
For a pair $(H,\alpha)$ of an ample divisor $H$ and a ${\Bbb Q}$-divisor $\alpha$ 
on $X$, a purely 1-dimensional sheaf $E$ is $\alpha$-twisted semi-stable
(resp. stable)
if 
\begin{equation}
\frac{\chi(E_1)-(\alpha\cdot c_1(E_1))}{(c_1(E_1) \cdot H)} \underset{(<)}{\leq} 
\frac{\chi(E)-(\alpha \cdot c_1(E))}{(c_1(E) \cdot H)}
\end{equation}
for a proper subsheaf $0 \ne E_1$ of $E$.
For $G \in K(X)$ with $\rk G>0$, we define $G$-twisted stability as
$\frac{c_1(G)}{\rk G}$-twisted stability.
\end{defn}

\begin{rem}\label{rem:twisted}
Obviously $\alpha$-twisted stability is equivalent to $\alpha+\lambda H$-twisted stability.
Hence we may assume that $\alpha \in H^\perp$.
If $X \to C$ be an elliptic surface and we are interested in fiber sheaves, 
then we may assume that $(\alpha \cdot f)=0$ (cf. Definition \ref{defn:n-stable}). 
\end{rem}

For ${\bf e}=(r,\xi,a) \in {\Bbb Z}\oplus \NS(X) \oplus {\Bbb Z}$,
$M_H^\alpha({\bf e})$ denotes the moduli space of $\alpha$-twisted stable sheaves $E$ on $X$
with $\tau(E)={\bf e}$ and $\overline{M}_H^\alpha({\bf e})$ the projective
compactification by adding 
$S$-equivalence classes of $\alpha$-twisted semi-stable sheaves
(see \cite[Thm. 4.7]{Y:twist2} for $r=0$).

For a pair $(H,\alpha)$ of an ample divisor $H$ and a ${\Bbb Q}$-divisor $\alpha$,
let ${\cal M}_H^\alpha({\bf e})^{ss}$ be the moduli stack of $\alpha$-twisted semi-stable sheaves
$E$ with $\tau(E)={\bf e}$ and ${\cal M}_H^\alpha({\bf e})^s$ the  substack of $\alpha$-twisted stable sheaves.

Let $e({\cal M}_H^\alpha({\bf e})^{ss})$ be the virtual Hodge polynomial of
${\cal M}_H^\alpha({\bf e})^{ss}$ in \cite[1.1]{Y:twist2}.

\subsection{Relative Fourier-Mukai transforms.}\label{subsect:relFM}

Let $Y:=M_H^\alpha(0,r' f,d')$ be a fine moduli space and
${\cal P}$ a universal family, where $\alpha \in f^\perp$.
Then we have an elliptic fibration 
$\pi':Y \to C$ \cite[Lem. 3.1.9]{PerverseII}. 
Bridgeland proved that
$\Phi_{X \to Y}^{{\cal P}^{\vee}}:{\bf D}(X) \to {\bf D}(Y)$ is an equivalence.
In the notation of \cite[sect. 3.2]{PerverseII},
we have the following relation
$$
\Phi_{X \to Y}^{{\cal P}^{\vee}}(E)=\Psi(E(-K_X))^{\vee}[-2], \; E \in {\bf D}(X).
$$
We take a locally free sheaf $G$ such that
$\frac{c_1(G)}{\rk G}=\alpha+\frac{d_1 H}{r_1(H \cdot f)}$.
Then ${\cal P}_{|X \times \{ y\}}$ $(y \in Y)$ is a $G$-twisted stable sheaf with
\begin{equation}\label{eq:GP}
\chi(G,{\cal P}_{|X \times \{ y\}})=0.
\end{equation}
We set $\tau({\cal P}^{\vee}_{|\{ x \} \times Y}[1])=(0,r_1' f, d_1')$.
We assume that there is an integral curve $C \in |H|$, and we set $L:={\cal O}_C$.
Then $G':=\Phi_{X \to Y}^{{\cal P}^{\vee}}(L)[1]$ is a locally free sheaf on $Y$.
We set $L':=\Phi_{X \to Y}^{{\cal P}^{\vee}}(G)[2]$.
Then $L'$ is a purely 1-dimensional sheaf on $Y$.
By \cite[Lem. 3.2.1]{PerverseII}, $H':=c_1(L')$ is $\pi'$-ample.
We set ${\cal Q}:={\cal P}^{\vee}[1]$.
By \cite[Thm. 3.2.8]{PerverseII},
${\cal Q}_{|\{ x \} \times Y}={\cal P}^{\vee}_{|\{ x \} \times Y}[1]$ $(x \in X)$ 
is $G'$-twisted stable with respect to $H'+nf$
$(n \gg 0)$ and we have an identification
\begin{equation}
\begin{matrix}
X & \to & M_{H'}^{\alpha'}(0,r_1' f,-d_1')\\
x & \mapsto & {\cal Q}_{|\{ x \} \times Y},
\end{matrix}
\end{equation}
 where
$\alpha'=\frac{c_1(G')}{\rk G'}+\frac{d_1' H'}{r_1' (f \cdot H')} \in f^\perp$.
%
We have 
\begin{equation}\label{eq:PQ}
\Phi_{X \to Y}^{{\cal Q}}=\Phi_{X \to Y}^{{\cal P}^{\vee}[1]},\;
\Phi_{Y \to X}^{{\cal P}}=\Phi_{Y \to X}^{{\cal Q}^{\vee}[1]}.
\end{equation}

For a 1-dimensional sheaf $A$ on $X$, we get
\begin{equation}\label{eq:FM(A)}
\begin{split}
\chi(G',\Phi_{X \to Y}^{{\cal P}^{\vee}}(A))=&(c_1(L) \cdot c_1(A)),\\
(c_1(L') \cdot \Phi_{X \to Y}^{{\cal P}^\vee}(A))=& -\chi(G,A).
\end{split}
\end{equation}

\begin{lem}\label{lem:WIT-A}
Let $A$ be a 1-dimensional sheaf on $X$.
\begin{enumerate}
\item[(1)]
If $\Hom({\cal P}_{|X \times \{ y\}},A)=0$ for all $y \in Y$, then
$\Phi_{X \to Y}^{{\cal P}^{\vee}}(A)[2] \in \Coh(Y)$.
\item[(2)]
If $\Hom(A,{\cal P}_{|X \times \{ y\}})=0$ for all $y \in Y$, then
$\Phi_{X \to Y}^{{\cal P}^{\vee}}(A)[1] \in \Coh(Y)$.
\end{enumerate}
\end{lem}

\begin{proof}
(1)
By $\Hom({\cal P}_{|X \times \{ y\}},A)=0$ ($y \in Y$),
$\Phi_{X \to Y}^{{\cal P}^{\vee}}(A)$ is represented by a two term complex
of locally free sheaves $V_1 \to V_2$.
Since $A$ is 1-dimensional, we get $\Ext^1( {\cal P}_{|X \times \{ y\}},A)=0$
for a general $y \in Y$. Hence $\Phi_{X \to Y}^{{\cal P}^{\vee}}(A)[2] \in \Coh(Y)$.

(2)
Since ${\cal P}_{|X \times \{ y\}}(K_X)={\cal P}_{|X \times \{ y\}}$, we have
$\Ext^2({\cal P}_{|X \times \{ y\}},A)=
\Hom(A,{\cal P}_{|X \times \{ y\}})^{\vee}=0$.
Hence $\Phi_{X \to Y}^{{\cal P}^{\vee}}(A)$ is represented by a two term complex
of locally free sheaves $V_0 \to V_1$.
Since $\Hom({\cal P}_{|X \times \{ y\}},A)=0$ for a general $y \in Y$,
$\Phi_{X \to Y}^{{\cal P}^{\vee}}(A)[1] \in \Coh(Y)$.
\end{proof}

\begin{rem}
Under the condition of (1), $A$ is purely 1-dimensional.
\end{rem}

In the following proposition, we give a relation of twisted stability for
fiber sheaves. 

\begin{lem}\label{lem:G-stability}
Let $A$ be a $G$-twisted stable sheaf supported on a fiber.
\begin{enumerate}
\item[(1)]
Assume that 
$\chi(G,A)<0$. Then $\Phi_{X \to Y}^{{\cal P}^{\vee}}(A)[2]$ is a $G'$-twisted stable sheaf. 
\item[(2)]
Assume that $\chi(G,A)>0$. Then $\Phi_{X \to Y}^{{\cal P}^{\vee}}(A)[1]$ is a $G'$-twisted stable sheaf. 
\end{enumerate}
\end{lem}

\begin{proof}
(1)
We have $A':=\Phi_{X \to Y}^{{\cal P}^{\vee}}(A)[2] \in \Coh(Y)$ by \eqref{eq:GP} and Lemma \ref{lem:WIT-A}.
Assume that $A'$ is not $G'$-twisted stable, and we have 
an exact sequence
\begin{equation}
0 \to B_1 \to A' \to B_2 \to 0
\end{equation} 
such that $B_1$ is a stable sheaf with 
$$
\frac{\chi(G',B_1)}{(c_1(B_1) \cdot H')} \geq 
\frac{\chi(G',A')}{(c_1(A') \cdot H')}. 
$$
Since $\chi(G',{\cal Q}_{|\{x \} \times Y})=0$ and
$\chi(G',A')=(c_1(L) \cdot  c_1(A))>0$, $\Hom(B_1,{\cal Q}_{|\{x \} \times Y})=0$ for all $x \in X$.
Applying Lemma \ref{lem:WIT-A} (2) to $\Phi_{Y \to X}^{{\cal Q}^{\vee}}:{\bf D}(Y) \to {\bf D}(X)$, 
we have an exact sequence
\begin{equation}
0 \to B_1' \to A(-K_X) \to B_2' \to 0. 
\end{equation}
where $B_1':=\Phi_{Y \to X}^{{\cal P}}(B_1) \in \Coh(X)$ and
$B_2':=\Phi_{Y \to X}^{{\cal P}}(B_2) \in \Coh(X)$.
Since $A(-K_X) \cong A$, we have
$$
-\frac{(c_1(B_1') \cdot H)}{\chi(G,B_1')} \geq 
-\frac{(c_1(G) \cdot H)}{\chi(G,A)}. 
$$ 
Therefore $A'$ is $G'$-twisted stable.

(2)
We have $A':=\Phi_{X \to Y}^{{\cal P}^{\vee}}(A)[1] \in \Coh(Y)$ by
Lemma \ref{lem:WIT-A}.
Assume that $A'$ is not $G'$-twisted stable, and we have 
an exact sequence
\begin{equation}
0 \to B_1 \to A' \to B_2 \to 0
\end{equation} 
such that $B_2$ is a stable sheaf with 
$$
\frac{\chi(G',B_2)}{(c_1(B_2) \cdot H')} \leq 
\frac{\chi(G',A')}{(c_1(A') \cdot H')}. 
$$
Since $\chi(G',{\cal Q}_{|\{x \} \times Y})=0$ and
$\chi(G',A')<0$, $\Hom({\cal Q}_{|\{x \} \times Y},B_2)=0$ for all $x \in X$.
By Lemma \ref{lem:WIT-A},
we have an exact sequence
\begin{equation}
0 \to B_1' \to A(-K_X) \to B_2' \to 0. 
\end{equation}
where $B_1':=\Phi_{Y \to X}^{{\cal P}}(B_1)[1] \in \Coh(X)$ and
$B_2':=\Phi_{Y \to X}^{{\cal P}}(B_2)[1] \in \Coh(X)$.
Since $A(-K_X) \cong A$, we have
$$
-\frac{(c_1(B_2') \cdot H)}{\chi(G,B_2')} \leq 
-\frac{(c_1(G) \cdot H)}{\chi(G,A)}. 
$$ 
Therefore $A'$ is $G'$-twisted stable.
\end{proof}

\begin{rem}\label{rem:G-stability}
\begin{enumerate}
\item[(1)]
By \eqref{eq:PQ}, Lemma \ref{lem:G-stability} implies that
$A \in {\bf D}(X)$ is a $G$-twisted stable sheaf supported on a fiber
such that $\chi(G,A)<0$
if and only if $A':=\Phi_{X \to Y}^{{\cal P}^{\vee}}(A)[2]$
is a $G'$-twisted stable sheaf supported on a fiber such that $\chi(G',A')>0$.
\item[(2)]
$G'$-twisted stability is also defined for 0-dimensional sheaves and 
$G$-twisted stable sheaf $A$ supported on a fiber with $\chi(G,A)=0$
corresponds to a structure sheaf of a point via $\Phi_{X \to Y}^{{\cal P}^{\vee}[2]}$.
\end{enumerate}
\end{rem}

\subsection{Properties of purely 1-dimensional sheaves}\label{subsect:1-dim}

We shall study some properties of $\alpha$-twisted stable sheaves
on a fiber.

\begin{lem}\label{lem:1-dim}
Let $E$ be a purely 1-dimensional sheaf with $(c_1(E)\cdot f)=1$.
For purely 1-dimensional sheaves $E_1$ and $E_2$ fitting in an exact 
sequence
$$
0 \to E_1 \to E \to E_2 \to 0,
$$
we have $(c_1(E_1)\cdot f)=0$ or $(c_1(E_2)\cdot f)=0$.
Hence $E_1$ or $E_2$ is supported on fibers.
\end{lem}

\begin{proof}
$\Div(E_1)$ and $\Div(E_2)$ are effective divisors such that
$$
(\Div(E_1)\cdot f)+(\Div(E_2) \cdot f)=(c_1(E) \cdot f)=1.
$$
Since $f$ is nef, 
we get $(\Div(E_1) \cdot f)=0$ or $(\Div(E_2) \cdot f)=0$.
Therefore $E_1$ is supported on fibers or 
$E_2$ is supported on fibers.
\end{proof}

\begin{lem}\label{lem:g=1}
Let $E$ be an $\alpha$-twisted stable sheaf of dimension 1 
whose support is contained in 
a fiber $mD$, where $m$ is the multiplicity.
Then $E$ is an ${\cal O}_D$-module. 
\end{lem}

\begin{proof}
By our assumption, $kD-\Div E$ is an effective divisor 
for a large positive integer $k$.
Then $E$ is an ${\cal O}_{kD}$-module.
We may assume that 
the multiplication map
$$
\phi:E \overset{(k-1)D}{\longrightarrow} E((k-1)D)
$$
 is non-zero.
Since $\tau(E)=\tau(E((k-1)D))$,
by the $\alpha$-twisted stability of $E$ and $E((k-1)D)$,
$\phi$ is an isomorphism.  
Since $\phi$ factors through $E_{|D}$, $E$ is an ${\cal O}_D$-module.
\end{proof}

\begin{lem}\label{lem:u}
Let $E$ be an $\alpha$-twisted stable 1-dimensional sheaf such that $(c_1(E) \cdot f)=0$.
\begin{enumerate}
\item[(1)]
$E$ is an ${\cal O}_D$-module, where 
$mD$ is a fiber of $\pi$ with multiplicity $m$.
\item[(2)]
If $(c_1(E)^2)<0$, then $(c_1(E)^2)=-2$ and $E$ is a spherical sheaf. In particular $E(K_X) \cong E$.
\item[(3)]
Assume that there is a section.
Then $E(K_X) \cong E$.
If $(c_1(E)^2)=0$, then $\tau(E)=(0,rf,a)$, $\gcd(r,a)=1$ and
$E$ is a stable sheaf of rank $r$ on $f$. 
\end{enumerate}
\end{lem}  

\begin{proof}
(1) 
By $(c_1(E) \cdot f)=0$, $\Div(E)$ is supported on fibers.
Since $E$ is $\alpha$-twisted stable, $\pi(\Div E)$ is a point $c \in C$.
We set $\pi^{-1}(c)=mD$, where $D$ is the multiplicity.
By Lemma \ref{lem:g=1}, $E$ is an ${\cal O}_D$-module.
We also have $(c_1(E)^2) \leq 0$.

(2)
We note that 
$$
\Ext^2(E,E) \cong \Hom(E,E(K_X))^{\vee}
$$
by the Serre duality.
\begin{NB}
$$
\Ext^2(E,E) \cong \Hom(E,E(K_X))^{\vee} \cong \Hom(E,E)^{\vee} \cong {\Bbb C}.
$$
\end{NB}
If $\Hom(E,E(K_X)) \ne 0$, then by 
$\tau(E)=\tau(E(K_X))$ and the $\alpha$-twisted stability of $E, E(K_X)$,
we see that $E \cong E(K_X)$ and $\Ext^2(E,E) \cong {\Bbb C}$.
Hence $-2 \leq -\chi(E,E)=(c_1(E)^2)$. Since $(c_1(E) \cdot K_X)=0$,
we get $(c_1(E)^2)=0,-2$.
Moreover if $(c_1(E)^2)=-2$, then $\Ext^1(E,E)=0$ and $E$ is a spherical sheaf.

(3) Assume that there is a section $\sigma$.
If $(c_1(E)^2)=0$, then $\tau(E)=(0,rf,a)$ for some $r,a \in {\Bbb Z}$.
Assume that $d:=\gcd(r,a)>1$. We set $r':=r/d$ and $a':=a/d$.
We consider $Y:=M_H^{\alpha'}(0,r' f,a')$, where $\alpha'$ is general
and sufficiently close to $\alpha$.
By the existense of $\sigma$, $Y$ is fine and 
isomorphic to $X$. For a universal family ${\cal P}$ on $X \times Y$,
we consider $\Phi_{X \to Y}^{{\cal P}^{\vee}[1]}:{\bf D}(X) \to {\bf D}(Y)$.
Since $(c_1(E) \cdot H)=d(r'f \cdot H)>(r'f \cdot H)$ and $E$ is $\alpha'$-twisted stable, we get
$$
\Hom(E,{\cal P}_{|X \times \{ y \}})=\Hom({\cal P}_{|X \times \{ y \}},E)=0
$$
for all $y \in Y$. 
\begin{NB}
${\cal P}_{|X \times \{ y \}}$ and $E$ are $\alpha'$-stable sheaves of the same slope.
\end{NB}
Then we get $\Ext^1({\cal P}_{|X \times \{ y \}},E)=0$ for all $y \in Y$
by $\chi({\cal P}_{|X \times \{ y \}},E)=0$. Hence
$\Phi_{X \to Y}^{{\cal P}^{\vee}[1]}(E)=0$, which is a 
contradiction.  Therefore $d=1$.
\begin{NB}
We set $\Div (E)=rD$, where $D$ is a fiber of $\pi$.
Assume that $E(-kD) \to E$ is not zero and $E(-(k+1)D) \to E(-kD) \to E$ is zero.
Since $E$ is purely 1-dimensional, $F:=\im(E(-kD) \to E)$ is a 1-dimensional sheaf.
We have a surjective homomorphism $E \to E_{|D} \to F$.
By the stability of $E$, $E \cong F$. Therefore $E$ is an ${\cal O}_D$-module.  
\end{NB}
\end{proof}

We have the following corollary from the proof of Lemma \ref{lem:u} (3).
\begin{cor}\label{cor:u}
For a fiber $D$, 
we set
$$
{\cal M}_H^\alpha(0,rf,a,D)^{ss}:=\{E \in {\cal M}_H^\alpha(0,rf,a)^{ss} \mid \Supp E=D\},
$$
where $\gcd(r,a)=1$.
If there is a section, then $\dim {\cal M}_H^\alpha(0,rf,a,D)^{ss}=0$.
\end{cor}

\begin{rem}\label{rem:u}
Assume that there is a section of $\pi$.
Let $C_0,C_1,...,C_N$ be the irreducible components of a singular fiber $D=\sum_{i=0}^N a_i C_i$.
We may assume that $a_0=(C_0 \cdot \sigma)=1$ for a section $\sigma$.
We note that
$V^*:={\Bbb Q}\sigma+\sum_{i=1}^{N} {\Bbb Q}C_i$ is isomorphic to
the dual of $V:={\Bbb Q}f+\sum_{i=1}^{N} {\Bbb Q}C_i$ by the intersection pairing.

Let $E$ be a 1-dimensional sheaf on $D$.
Then for any linear form $\varphi:V \to {\Bbb Q}$,
there is $\alpha \in V^*$ such that
$\varphi(c_1(E))=(c_1(E) \cdot \alpha)$. 
Moreover if $\varphi(C_i)>0$ for all $i$, then $\alpha$ is relatively ample.
For a pair of linear forms $(\varphi,\psi)$ such that $\varphi(C_i)>0$ for all $i$,
we define $(\varphi,\psi)$-semi-stability of an ${\cal O}_D$-module $E$ by
using the slope function 
$$\mu(E):=\frac{\chi(E)-\psi(c_1(E))}{\varphi(c_1(E))},
$$
 where
$c_1(E) \in \sum_{i=0}^{N} {\Bbb Z}C_i$ is the 1-cycle associated to $E$. 
Then Lemma \ref{lem:u} (3) is regarded as a claim for the $(\varphi,\psi)$-stability.
In particular, by taking an elliptic surface $X$ with a section such that $D$ is contained as a singular fiber,
we get that
similar claims to Lemma \ref{lem:u} (3) hold for the $(\varphi,\psi)$-stability.
\begin{NB}
In order to find an ample divisor $H$ which represents $\varphi$,
we take relatively ample divisors $H_i$ for each singular fibers $f_i$. Then replacing $\sum_i H_i$ by 
$\sum_i H_i+pf$, we get an ample
divisor.   
\end{NB}
\end{rem}

\begin{NB}
Let $E$ be an $\alpha$-twisted stable sheaf on a fiber $\pi^{-1}(t)$.
We set $\pi^{-1}(t)=mD$, where $m$ is the multiplicity of $\pi^{-1}(t)$.
Then we have a non-zero homomorphism $E(-kD) \to E$ such that
$E(-(k+1)D) \to E(-kD) \to E$ is zero.
Then we have a non-zero homomorphism $E_{|D}(-kD) \to E$.
Since $\chi(E(-kD))=\chi(E)$, $\alpha$-twisted stability of $E$ and $E(-kD)$ imply
that $E(-kD) \to E$ is isomorphic. Therefore $E \cong E_{|D}(-kD)$.

Hence any surjective homomorphism $E' \to E$ factors through
$E' \to E'_{|D} \to E$. 
 
\end{NB}

\begin{prop}\label{prop:smooth}
Let $\xi$ be an effective divisor with $(\xi \cdot f)=1$. Then
${\cal M}_H^\alpha(0,\xi,a)^{ss}$ is smooth and
\begin{equation}\label{eq:moduli-dim}
\dim {\cal M}_H^\alpha(0,\xi,a)^{ss}=(\xi^2)+g+e-1.
\end{equation}
\end{prop}

\begin{proof}
In \cite[Prop. 3.18]{Y:twist2}, we proved the claim 
for elliptic surfaces with irreducible fibers.
The same proof also works in our situation.
Thus we see that
$$
H^0(X,{\cal O}_X(K_X)) \to \Hom(E,E(K_X))
$$
is an isomorphism for all $E \in {\cal M}_H^\alpha(0,\xi,a)^{ss}$.
Then \eqref{eq:moduli-dim} follows by $\chi(E,E)=-(\xi^2)$ and
$p_g=e+g-1$.
\begin{NB}
Let $E$ be a simple purely 1-dimensional sheaf with 
$c_1(E)=\xi$.
It is sufficient to prove that
$$
H^0(X,{\cal O}_X(K_X)) \to \Hom(E,E(K_X))
$$
is an isomorphism.
We set 
$\Div(E)=\tau+D$, where $\tau$ is a section of $\pi$ and $D$ is supported on fibers.
We have an exact sequence
\begin{equation}
0 \to E_1 \to E \to E_2 \to 0
\end{equation}
where $E_2:=E_{|\tau}/(\text{torsion})$ and
$E_1:=\ker(E \to E_2)$.
Since $\tau \cong C$ is smooth, $E_2$ is a line bundle on $\tau$. 
For $\phi:E \to E(K_X)$,
$\phi(E_1) \subset E_1(K_X) \cong E_1$.
Hence $\phi$ induces a homomorphism $E_2 \to E_2(K_X)$.
Thus we get a homomorphism 
$$
g:\Hom(E.E(K_X)) \to \Hom(E_2,E_2(K_X)).
$$
By the simpleness of $E$ and $E_1(K_X) \cong E_1$,
$\Hom(E,E_1(K_X))=\Hom(E,E_1)=0$.
Since $\Hom(E_2,E_1(K_X))=0$, we get $\Hom(E,E_1(K_X))=0$. 
Thus $g$ is injective.
Hence
$$
H^0(C,\pi_*({\cal O}_X(K_X))) =H^0(X,{\cal O}_X(K_X)) \to \Hom(E,E(K_X)) \to 
\Hom(E_2,E_2(K_X))
$$
is isomorphic. Therefore $H^0(X,{\cal O}_X(K_X)) \to \Hom(E,E(K_X))$ is
isomorphic.
\end{NB}
\end{proof}

\subsection{Existence of spherical stable sheaves}\label{subsect:spherical}


\begin{lem}\label{lem:e-poly}
Let $D$ be an effective divisor such that $(D^2)=-2$ and $(D\cdot f)=0$.
Then $e({\cal M}_H^\alpha(0,D,a)^{ss})$ is independent of the choice of a general $(H,\alpha)$. 
\end{lem}

\begin{proof}
We fix an ample divisor $L$.
Applying \cite[Prop. 2.5]{Y:twist2} to 1-dimensional sheaves on fibers, we get \cite[(2.29)]{Y:twist2}.
By induction on $(D \cdot L)$, we can  prove the claim.
\end{proof}

\begin{lem}\label{lem:-2}
Assume that there is a section $\sigma$.
Let $D$ be a divisor such that ${\cal O}_f(D) \cong {\cal O}_f$ for a general fiber $f$
and $(D^2)=-2$.
Then there is a reducible singular fiber $\pi^{-1}(c):=\sum_{i=0}^m a_i C_i$ such that
$C_i$ are smooth rational curves, $a_0=(C_0 \cdot \sigma)=1$, and
$D \equiv \pm (\sum_{i>0}b_i C_i) \mod {\Bbb Z}f$ in $\NS(X)$, where $0 \leq b_i \leq a_i$. 
\end{lem}

\begin{proof}
We take two fibers $f_1$ and $f_2$ such that 
${\cal O}_{f_i}(D) \cong {\cal O}_{f_i}$ ($i=1,2$).
We have an exact sequence
\begin{equation}
0 \to {\cal O}_X(D-f_1-f_2) \to {\cal O}_X(D) \to
\oplus_{i=1}^2 {\cal O}_{f_i}(D) \to 0.
\end{equation}
We note that 
\begin{equation}\label{eq:|D|}
\chi({\cal O}_X(D-f_1-f_2))=\chi({\cal O}_X(D))=e-1 \geq -1.
\end{equation}
If $h^1(X,{\cal O}_X(D-f_1-f_2)) \leq 1$, then
$h^0(X,\oplus_{i=1}^2 {\cal O}_{f_i}(D))=2$ implies
$h^0(X,{\cal O}_X(D))\geq 1$.
If $h^1(X,{\cal O}_X(D-f_1-f_2)) \geq 2$, then 
\eqref{eq:|D|} implies $h^0(X,{\cal O}_X(D-f_1-f_2)) \geq 1$ or
$h^2(X,{\cal O}_X(D-f_1-f_2)) \geq 1$.
Hence $D =\pm D'+k f$, where $D'$ is effective.
Let $L$ be the sublattice of $f^\perp \subset \NS(X)$ 
generated by irreducible components of singular fibers.
Then $D \in L$. Since $(D^2)=-2$, there is a singular fiber $\pi^{-1}(c)=\sum_{i=0}^m a_i C_i$
and $D \equiv \pm(\sum_{i>0} b_i C_i) \mod {\Bbb Z}f$ ($0 \leq b_i \leq a_i$).  
\end{proof}

Since $\sum_{i=0}^m a_i C_i-\sum_{i>0} b_i C_i$ is effective,
we also get the following.

\begin{cor}\label{cor:-2}
Assume that there is a section $\sigma$.
Let $D$ be a divisor such that ${\cal O}_f(D) \cong {\cal O}_f$ for a general fiber $f$
and $(D^2)=-2$. Then $D$ or $-D$ is algebraically equivalent to an effective divisor.
\end{cor}

\begin{NB}
If $n \geq 0$, then $\sum_i b_i C_i+nf \geq 0$.
If $n<0$, then $-(\sum_i b_i C_i+nf)=
-(\sum_i b_i C_i+|n|f)=-((\sum_i b_i C_i+f)+(|n|-1)f)$.  
\end{NB}

\begin{NB}
$D \equiv \sum_{i=1}^m b_i C_i \mod {\Bbb Z}f$.
$f \equiv \sum_{i=0}^m a_i C_i$, where $(C_0 \cdot \sigma)=1$.
$0 \leq b_i \leq a_i$.
If $(D \cdot H)=0$, then
$((\sum_{i=1}^m b_i C_i)\cdot H) \in {\Bbb Z}(f \cdot H)$.
Since $0 \leq (( \sum_{i=1}^m b_i C_i) \cdot H) <(f \cdot H)$,
we get a contradiction.
\end{NB}

\begin{lem}\label{lem:spherical}
Let $D$ be an effective divisor such that 
$(D \cdot f)=0$ and $(D^2)=-2$. Then
$M_H^\alpha(0,D,a) \ne \emptyset$ for a general $(H,\alpha)$.  
\end{lem}

\begin{proof}
By Lemma \ref{lem:e-poly}, it is sufficient to find a pair $(H,\alpha)$ such that
$M_H^\alpha(0,D,a) \ne \emptyset$.
We note that there is an effective divisor $D'$ such that $D+D'=rf_0$,
where $mf_0$ is a fiber of $\pi$ with multiplicity $m$.
\begin{NB}
Since $D$ is effective and $(D,f)=0$, $D$ is supported on fibers.
$D=\sum_\lambda D_\lambda$, $\pi(D_\lambda)=\lambda \in C$.
By $-2=(D^2)=\sum_\lambda (D_\lambda^2)$,
$D_\lambda=n_\lambda f$ or $(D_\lambda^2)=-2$.
Since $\pi^{-1}(\lambda)$ is of affine type,
$D_\lambda \mod f \in \oplus_j {\Bbb Z}D_{\lambda j}$, where $D_{\lambda j}$ is an irreducible
component in $\sigma^\perp$.
\end{NB}
We shall prove that there is a coherent ${\cal O}_{f_0}$-module $E$ which is 
stable with respect to $(H,\alpha)$ (cf. Remark \ref{rem:u}). 
Replacing $X$ by an elliptic surface with a section, we may assume that
$X$ has a section and hence $m=1$.
For $(0,D,a)$ and $(0,D',a')$ such that $\gcd(r,a+a')=1$,
we take a general $\alpha \in \NS(X)_{\Bbb Q}$ such that
$$
\frac{(D \cdot \alpha)-a}{(D \cdot H)}=\frac{(rf \cdot \alpha)-(a+a')}{(rf \cdot H)}.
$$
We set $d:=a+a'$ and set $Y':=\overline{M}_H^\alpha(0,rf,d)$.
By \cite[Cor. 3.1.7]{PerverseII},
$Y'$ is singular if and only if
there are $\alpha$-twisted stable sheaves $E$ and $E'$ such that 
$\tau(E)=(0,D,a)$, $\tau(E')=(0,D',a')$ and $E \oplus E'$ is the $S$-equivalence class. 
We have an elliptic fibration $\varpi':Y' \to C$ by \cite[Lem. 3.1.9]{PerverseII}.
For a general $\alpha'$ which is sufficiently close to 
$\alpha$, we set $X':=\overline{M}_H^{\alpha'}(0,rf,d)$.
Then we have a morphism $X' \to Y'$ which is the minimal resolution 
of $Y'$ (\cite[Cor. 3.1.7]{PerverseII}).
We also note that $X' \cong X$ by the existence of $\sigma$.
By \cite[Lem. 3.2.1]{PerverseII}, we have a divisor 
$\widehat{H}$ on $X'$ which is the pull-back of a $\varpi'$-relative ample divisor on $Y'$.
Let $\Phi_{X \to X'}^{{\cal P}^{\vee}}:{\bf D}(X) \to {\bf D}(X')$ be the equivalence
where ${\cal P}$ is the universal family.
Then 
$\Phi_{X \to X'}^{{\cal P}^{\vee}}((0,D,a))=(0,\widehat{D},b)$, where
$\widehat{D}$ is a divisor such that $(\widehat{D}^2)=(D^2)=-2$ and
${\cal O}_f(\widehat{D}) \cong {\cal O}_f$ for a general fiber $f$.
By Corollary \ref{cor:-2}, $\widehat{D}$ or $-\widehat{D}$ is algebraically equivalent to an effective
divisor.
\begin{NB}
By Lemma \ref{lem:-2},
there is a singular fiber $\sum_{i=0}^m a_i C_i$ and
$\widehat{D} \equiv \pm(\sum_{i>0}b_i C_i) \mod {\Bbb Z}f$
$(0 \leq b_i \leq a_i)$.
Assume that $\widehat{H}$ is relatively ample.
Then $(C_i \cdot \widehat{H})>0$ for all $i$.
Hence $0 \leq (((\sum_{i>0}b_i C_i)\cdot \widehat{H})<(f \cdot \widehat{H})$.
\end{NB}
Since $(\widehat{H} \cdot \widehat{D})=0$, we get $\widehat{H}$ is not relatively ample
with respect to the elliptic fibration $X' \to C$.
Hence $Y'$ is singular and we get 
$M_H^\alpha(0,D,a) \ne \emptyset$ for a general $(\alpha,H)$.  
\end{proof}

\subsection{Stable sheaves on a multiple fiber.}\label{subsect:multiple}

Let $\pi:X \to C$ be an elliptic surface such that $R^1 \pi_* {\cal O}_X 
\not \cong {\cal O}_C$.

\begin{lem}\label{lem:multiple1}
Let $mf_0$ be a fiber of $\pi$, where $m \geq 1$ is the multiplicity.
We take $E \in {\cal M}_H^\alpha (0,lrf_0,ld)^s$, where $\gcd(r,d)=1$.
\begin{enumerate}
\item[(1)]
Assume that 
$\Div E=lrf_0$.
Then $l=1$ and $E$ is an ${\cal O}_{f_0}$-module.
\item[(2)]
Assume that $\Supp E \ne f_0$ and $\Div E$ is algebraically equivalent to $lrf_0$.
Then $m \mid lr$ and $\gcd(\frac{lr}{m},ld)=1$.  
\end{enumerate}
\end{lem}

\begin{proof}
(1) 
\begin{NB}
Since $E$ is an ${\cal O}_{lrf_0}$-module,
there is a non-negative integer $k<lr$ and 
we have a non-zero homomorphism $E(-kf_0) \to E$ such that
$E(-(k+1)f_0) \to E(-kf_0) \to E$ is zero.
Then we have a non-zero homomorphism $E_{|f_0}(-kf_0) \to E$.
Since $\chi(E(-kf_0))=\chi(E)$, $\alpha$-twisted stability of $E$ and $E(-kf_0)$ imply
that $E(-kf_0) \to E$ is isomorphic. Therefore $E \cong E_{|f_0}(-kf_0)$.
\end{NB}
By Lemma \ref{lem:g=1}, $E$ is an ${\cal O}_{f_0}$-module.
We take an elliptic surface $X' \to C'$ such that there is a section and a fiber
$f_0$. Then $E$ is regarded as an 1-dimensional sheaf on $X'$. 
By Remark \ref{rem:u} and Lemma \ref{lem:u} (3), we get $l=1$.

(2)
Since $E$ is $\alpha$-twisted stable, $\Supp(E)$ is connected.
Hence there is a fiber $m_1 f_1$ such that 
$\Div E=k f_1$, where $m_1$ is the multiplicity of $f_1$.
By our assumption and $\Pic^0(X)=\Pic^0(C)$,
${\cal O}_X(lr f_0-kf_1) \in \Pic^0(C)$. 
We take integers $a,b$ such that $lr=am+b$, $0 \leq a$ and $0 \leq b<m$.
Then 
$$ 
{\cal O}_{f_0}(bf_0) \cong 
{\cal O}_X(lr f_0-kf_1)_{|f_0} \cong
{\cal O}_{f_0}.
$$
Hence $b=0$ and $m \mid lr$.
Since $(0,lrf_0,ld)=(0,\frac{lr}{m}m_1 f_1,ld)$, (1) implies $\gcd(\frac{lr}{m}m_1,ld)=1$.
 \end{proof}

\begin{NB}
\begin{lem}
Let $D \ne af_0$ be a small deformation of $af_0$ 
as effective Cartier divisors on $X$.
Then $D=D'+b f_0$, where $m \mid (a-b)$ and
$D'$ is algebraically equivalent to $\frac{a-b}{m}f$.
In particular $a \geq m$. 
\end{lem}
\begin{proof}
Let ${\cal D}_t$ $(t \in T)$ be a relative Cartoer divisor over an irreducible curve $T$
such that ${\cal D}_{t_0}=a f_0$ and ${\cal D}_{t_1}=D$.
We may assume that there is an integer
$0<b \leq a$ such that ${\cal D}_t-b f_0$ is not effective for $t \ne t_0$.
We may also assume that ${\cal D}_t-b f_0$ does not intersect other multiple fibers.
Then ${\cal D}_t-b f_0= kf$ for $t \ne t_0$.
Since $(({\cal D}_{t_0}-{\cal D}_t) \cdot H)=
(a-km-b)(f_0 \cdot H)=0$, $a=mk+b$. In particular $a \geq m$.
\end{proof}
\end{NB}

\begin{NB}
Let ${\cal E}$ be a family of semi-stable sheaves over a curve $T$
such that $\Div ({\cal E}_{t_0})=af_0$ 
and $\Div({\cal E}_t)=bf_0+D_t$ ($D_t \cap f_0=\emptyset$) 
for $t \ne t_0$, where $0 \leq b<a$. 
Thus $\Div({\cal E}_t)-bf_0$ is effective for all $t$ but
$\Div({\cal E}_t)-(b+1)f_0$ is not effective for $t \ne t_0$.
Assume that ${\cal E}_t$ is $S$-equivalent to $\oplus_{i,j} E_{ij}$,
where $E_{ij}$ are stable sheaves such that 
$\Supp(E_{0j})=f_0$ and $\Supp(E_{1j}) \ne f_0$.
We may assume that $\Supp(E_{1j})$ are smooth for all $j$.
Then $\tau(E_{1j})=(0,r' f,d')$ and $\tau(E_{0j})=(0,rf_0,d)$, where
$(r' m,d') \in {\Bbb Z}(r,d)$. 
Hence
$\frac{a}{r}(r f_0,d)=\frac{b}{r}(r f_0,d)+l(r' m f_0,d')$, where $a \equiv b \equiv 0 \mod r$.
\end{NB}

\begin{lem}\label{lem:multiple2}
Let $mf_0$ be a fiber with multiplicity $m$. We set $\tau:=(0,rf_0,d)$, where
$\gcd(r,d)=1$ and $r>0$.
If $m \nmid r$, then 
$\dim M_H^\alpha(0,rf_0,d)=1$ for a general $\alpha$.
\end{lem}

\begin{proof}
For $E \in M_H^\alpha(0,rf_0,d)$,
Lemma \ref{lem:multiple1} implies
$E$ is an ${\cal O}_{f_0}$-module.
By Remark \ref{rem:u} and Corollary \ref{cor:u},
we get our claim.
\end{proof}

\begin{rem}
We take $E \in {\cal M}_H^\alpha (0,lrf_0,ld)^s$.
\begin{enumerate}
\item[(1)]
Assume that $m \nmid lr$ and $E$ is a locally free ${\cal O}_{f_0}$-module.
Then $l=1$ and $\Hom(E,E(K_X))=0$. In particular 
${\cal M}_H^\alpha (0,lrf_0,ld)^s$ is smooth of dimension 0 at $E$.
\item[(2)]
Assume that $m \mid lr$. Then  
${\cal M}_H^\alpha (0,lrf_0,ld)^s$ is smooth of dimension 1 at $E$.
Moreover $\Supp E \ne f_0$ for a general $E$ if $l>1$.
\end{enumerate}
\end{rem}

\begin{NB}
 $h^0({\cal O}_{f_0}(kf_0))=0$ for $0<k<m$.
Hence $h^0({\cal O}_X((k-1)f_0)=h^0({\cal O}_X(kf_0))$.

We note that $(K_X)_{|f_0}$ is of order $m$ as a line bundle on $f_0$.
If $m \nmid r$, then $\Hom(E,E(K_X))=0$. 
\end{NB}

\begin{defn}\label{defn:multiple}
Let $mf_0$ be a multiple fiber. We set $\tau:=(0,rf_0,d)$, where $\gcd(r,d)=1$.
We set
$$
{\cal M}_H^\alpha(l\tau,lrf_0)^{ss}:=\{E \in {\cal M}_H^\alpha(l\tau)^{ss}
\mid \Div E=ld f_0 \}.
$$
\end{defn}

\begin{prop}\label{prop:multiple}
Assume that $(H,\alpha)$ is general. Then 
$\dim {\cal M}_H^\alpha(l \tau,lr f_0)^{ss} \leq 0$. In particular
$\dim {\cal M}_H^\alpha(l \tau)^{ss} \leq 0$ if $l \gcd(r,m)<m$.
\end{prop}

For the proof of this claim, we start with the following definition.

\begin{defn}\label{defn:J}
For $E_0 \in {\cal M}_H^\alpha (l_0 \tau)^s$,
we set
\begin{equation}
{\cal J}(l, E_0):=\{E \in {\cal M}_H(l \tau)^{ss} \mid 
\text{ $E$ is generated by $E_0(p K_X)$, $p \in {\Bbb Z}$ } \},
\end{equation}
where $l_0 \mid l$.
\end{defn}

\begin{lem}\label{lem:fiber-dim}
$\dim {\cal J}(l,E_0) \leq -1$.
\end{lem}

\begin{proof}
For $F \in \{E_0(p K_X) \mid p \in {\Bbb Z} \}$ and $n \geq 0$, we set
\begin{equation}
{\cal J}(l,E_0, F^{\oplus n}):=\{E \in {\cal J}(l,E_0) \mid
\dim \Hom(F,E)=n \}.
\end{equation}
For $E \in {\cal J}(l,E_0,F^{\oplus n})$,
we have an exact sequence
\begin{equation}\label{eq:ext1}
0 \to \Hom(F,E) \otimes F \to E \to E' \to 0
\end{equation}
and $E' \in {\cal J}(l-nl_0,E_0,F(-K_X)^{\oplus n'})$ $(n' \geq 0)$.
Since
${\cal J}(l,E_0,F^{\oplus n})$ is an open substack of the stack of extensions
\eqref{eq:ext1}, \cite[Lem. 5.2]{K-Y} implies 
\begin{equation}
\dim {\cal J}(l,E_0,F^{\oplus n}) \leq 
\dim {\cal J}(l-n,E_0,F(-K_X)^{\oplus n'})
+nn'-n^2.
\end{equation} 
\begin{NB}
$\dim \Ext^1(E',F)-\dim \Hom(E',F)=\dim \Hom(F(K_X),E')=n'$.
$\dim \Hom(F,E)=n$ is an open condition.
\end{NB}
Then the same proof of \cite[(3.8)]{K-Y} works. 
\end{proof}
 
{\it Proof of Proposition \ref{prop:multiple}.}

We note that $E \in {\cal M}_H^\alpha(l \tau,lr f_0)^{ss}$ is generated by
members in ${\cal M}_H^\alpha(\tau,r f_0)^{ss}$ (Lemma \ref{lem:multiple1} (1)).
There are $F_i \in {\cal M}_H^\alpha(\tau)^s$ and 
$E_i \in {\cal J}(l_i, F_i)$ such that $E \cong \oplus_i E_i$,
where $\sum_i l_i=l$.
By Lemma \ref{lem:fiber-dim} and Lemma \ref{lem:multiple2}, 
we get our claim. For more details, see \cite[sect. 5.3]{K-Y}.
\qed

\begin{NB}
Let $X_1 \to C_1$ be an elliptic surface such that there is a section $\sigma$ and
a singular fiber $D$ of type $I_N$.
Let $C_0,C_1,...,C_{N-1}$ be the irreducible components of $D$.
We may assume that $(C_0 \cdot \sigma)=1$.
We note that
${\Bbb Q}\sigma+\sum_{i=1}^{N-1} {\Bbb Q}C_i$ is isomorphic to
the dual of ${\Bbb Q}f+\sum_{i=1}^{N-1} {\Bbb Q}C_i$ by the intersection pairing.
Let $E$ be a 1-dimensional sheaf on $D$.
Then for any linear form $\varphi:\NS(X)_{\Bbb Q} \to {\Bbb Q}$,
there is $\alpha \in {\Bbb Q}\sigma+\sum_{i=1}^{N-1} {\Bbb Q}C_i$ such that
$\varphi(c_1(E))=(c_1(E) \cdot \alpha)$. 

\end{NB}

\begin{NB}
${\bf f}=(0,rf_0,d)$.
$m=m_1 m'$ and $r=r' m_1$ with $\gcd(m', r')=1$.
For $l=m'$, $(0,lrf_0,ld)=(0,m' r' m_1 f_0,m' d)=(0,r' f,m' d)$, which is primitive.
If $m' \nmid l$, then $\gcd(m',r')=1$ implies $lr=lr' m_1 \not \in m{\Bbb Z}$.
Hence $lrf_0$ is not movable.
\end{NB}

Let $m_1 f_1,m_2 f_2,...,m_s f_s$ be multiple fibers of $\pi$.
For a class $(0,r' f,d')$ with $\gcd(r',d')=1$,
let us consider $(0,r_i f_i,d_i)$ such that $\gcd(r_i,d_i)=1$ and
${\Bbb Q}(0,r_i f_i,d_i)={\Bbb Q}(0,r' f,d')$.
We set $p_i:=\gcd(r_i,m_i)$.
Then $r' m_i d_i=r_i d'$,
$r_i=p_i r'$ and $m_i=p_i \frac{d'}{d_i}$, where $d_i \mid d'$.
We also have
$(0,r' f,d')=\frac{d'}{d_i}(0,r_i f_i,d_i)$.
\begin{NB}
Indeed $d_i \mid d'$ implies $(0,r' f,d')=\frac{d'}{d_i}(0,r_i f_i,d_i)$.
Then $r' m_i=\frac{d'}{d_i} r_i$.
Since $\gcd(r',\frac{d'}{d_i})=1$ and $\gcd(\frac{m_i}{p_i},\frac{r_i}{p_i})=1$, $r'=\frac{r_i}{p_i}$.
\end{NB}
\begin{NB}
$(0,r' f,d')=(0,r' m_i f_i,d')=\frac{d'}{d_i}(0,r' p_i f_i,d_i)=\frac{d'}{d_i}(0,r_i f_i,d_i)$.
\end{NB}
We set
\begin{equation}\label{eq:isotropic}
{\bf f}:=\sum_i l_i(0,r_i f_i,d_i)+l(0,r' f,d')
\end{equation}
where $l_i,l \in {\Bbb Z}$ and
$0 \leq l_i <\frac{d'}{d_i}$.
Then $(l_1,...,l_s,l)$ is uniquely determined by ${\bf f}$.
\begin{NB}
Assume that
$$
\sum_i l_i r_i f_i+lr' f=\sum_i l_i' r_i f_i+l' r' f
$$
in $\NS(X)$,
where $0 \leq l_i,l_i'<\frac{d'}{d_i}$.
Then ${\cal O}_{f_i}((l_i-l_i') r_i f_i) \cong {\cal O}_{f_i}$ for all $i$.
Hence $\frac{d'}{d_i} \mid (l_i-l_i')$.
Since $|l_i-l_i'|<\frac{d'}{d_i}$, $l_i-l_i'=0$.
Then $lr' f=l' r' f$ implies $l=l'$.
\end{NB}
By using Lemma \ref{lem:fiber-dim}, we also see that the following holds.
\begin{lem}\label{lem:isotropic-dim}
$\dim {\cal M}_H^\alpha({\bf f})^{ss}=l$.
\end{lem}

\begin{NB}
$m_i \mid l_i p_i r'$ if and only if $\frac{d'}{d_i} \mid l_i$ by $\gcd(r',d')=1$.
\end{NB}

\begin{NB}
For ${\bf f}$, 
$\dim {\cal M}_H^\alpha({\bf f})^{ss}=l$, where $l$ is the maximal 
in the decomposition 
\begin{equation}\label{eq:isotropic}
{\bf f}=\sum_i l_i(0,r_i f_i,d_i)+l(0,r' f,d')
\end{equation}
where $0 \leq l_i <\frac{d'}{d_i}$.

\end{NB}

\begin{lem}\label{lem:spherical2}
Let $D$ be an effective divisor such that $(D^2)=-2$ and $\pi(D)$ is a point.
Then ${\cal M}_H^\alpha(0,D,a)^s$ consists of a spherical object for 
a general $(H,\alpha)$. 
\end{lem}

\begin{proof}
The existance is a consequence of Lemma \ref{lem:spherical}.
For $E_1,E_2 \in {\cal M}_H^\alpha(0,D,a)^s$,
we have $\chi(E_1,E_2)=2$. Then $\Hom(E_1,E_2) \ne 0$ implies $E_1 \cong E_2$.
\end{proof}

Let $Y:=M_H^\alpha(0,r_1 f,d_1)$ be a fine moduli space
and ${\cal P}$ a universal family on $X \times Y$,
where $(H,\alpha)$ is general.
Then $Y$ is an elliptic surface, that is, $Y$ has an elliptic fibration $\pi':Y \to C$.
The following lemma shows that the multiple fibers of $\pi'$
are explicitly determined by those for $\pi$.

\begin{lem}\label{lem:Y-multiple}
For a mutiple fiber $\pi^{-1}(c)=mf_0$,
we set ${\pi'}^{-1}(c)=m' f_0'$, where $m'$ is the multiplicity.
Then $m'=m$.
\end{lem}

\begin{proof}
Let $C$ be an irreducible component of $f_0$ with multiplicity 1 and 
$C'$ be an irreducible component of $f_0'$ with multiplicity 1. 
Since $\rk {\cal P}_{|X \times \{ y \}}=r' m$ on $C$ for all $y \in {\pi'}^{-1}(c)$,
$\rk {\cal P}_{|C \times C'}=r' m$.
We also have $\rk {\cal P}_{|C \times C'}=r' m'$.
\begin{NB}
Old:
Since $\rk {\cal P}_{|X \times \{ y \}}=r' m$ on $f_0$ for all $y \in {\pi'}^{-1}(c)$,
$\rk {\cal P}_{|f_0 \times f_0'}=r' m$.
We also have $\rk {\cal P}_{|f_0 \times f_0'}=r' m'$.
\end{NB}
Therefore $m=m'$.
\end{proof}

\section{Wall crossing for the moduli spaces of stable 1-dimensional sheaves}\label{sect:wall-crossing}

\subsection{Classification of walls.}\label{subsect:class-walls}

We set ${\bf e}:=(0,\xi,a)$, where $\xi$ is an effective divisor with $(\xi \cdot f)=1$
and $a \in {\Bbb Z}$. 
Let ${\cal F}({\bf e}_1,{\bf e}_2,\dots,{\bf e}_s)$ (cf. \cite[Prop. 2.4]{Y:twist2})
be the stack of 
Harder-Narasimhan filtrations
$$
0 \subset F_1 \subset F_2 \subset \cdots \subset F_s=E
$$
of $E \in {\cal M}_H^\alpha({\bf e})^{ss}$ 
with respect to $(H_\pm, \alpha_\pm)$, where $\tau(F_i/F_{i-1})={\bf e}_i$.
By Lemma \ref{lem:1-dim},
$$
\Ext^2(F_i/F_{i-1},F_j/F_{j-1})=\Hom(F_j/F_{j-1},F_i/F_{i-1} (K_X))^{\vee}=
\Hom(F_j/F_{j-1},F_i/F_{i-1})^{\vee}=0
$$
for $i>j$.
Hence by \cite[Prop. 2.5]{Y:twist2}, we get
\begin{equation}\label{eq:HNF-dim}
\dim {\cal F}({\bf e}_1,{\bf e}_2,\dots,{\bf e}_s)=\sum_{i<j} -\chi({\bf e}_j,{\bf e}_i)+
\sum_i \dim {\cal M}_{H_\pm}^{\alpha_\pm}({\bf e}_i)^{ss}.
\end{equation}

\begin{prop}\label{prop:wall}
Let $W$ be a wall for ${\bf e}=(0,\xi,a)$ where $\xi$ is an effective divisor with $(\xi,f)=1$.
\begin{enumerate}
\item[(1)]
$W$ is defined by one of the following ${\bf u}$. 
\begin{enumerate}
\item
$-\chi({\bf e},{\bf u}) \leq \frac{(\xi^2)+e-2}{2}$ 
for ${\bf u}=(0,D,b)$ such that $D$ is effective, 
$(D^2)=-2$ and $(D\cdot f)=0$.
\item
$0 < -\chi({\bf e},{\bf u})=r \leq \frac{(\xi^2)+e}{2}$ for ${\bf u}=(0,rf,d)$ with $\gcd(r,d)=1$.
\end{enumerate}
\item[(2)]
${\cal M}_H^\alpha({\bf u})^{ss}={\cal M}_H^\alpha({\bf u})^s$ if $(H,\alpha) \in W$ is general. 
\end{enumerate}
\end{prop}

\begin{proof}
Let $W$ be a wall for ${\bf e}$ and we take a general $(H,\alpha)$.
Let ${\cal C}_\pm$ be two chambers separated by $W$ and
$(H,\alpha) \in \overline{{\cal C}_\pm}$.
We take $(H_\pm, \alpha_\pm) \in {\cal C}_\pm$ in a 
neighborhood of $(H,\alpha)$.
For the Harder-Narasimhan filtrations
$$
0 \subset F_1 \subset F_2 \subset \cdots \subset F_s=E
$$
of $E \in {\cal M}_H^\alpha({\bf e})^{ss}$ 
with respect to $(H_\pm, \alpha_\pm)$,
we set ${\bf e}_i:=\tau(F_i/F_{i-1})=(0,\xi_i,a_i)$.
We note that
${\bf e}$ and ${\bf e}_i$ span a 2-plane in
$H^*(X,{\Bbb Q})$.
Since $(\xi \cdot f)=1$, there is $i_0$ such that $(\xi_i \cdot f)=0$ for $i \ne i_0$ and
$(\xi_{i_0}\cdot f)=1$.
If $s>2$, then $L:=\oplus_{i \ne i_0} {\Bbb Q}{\bf e}_i$ satisfies
$\dim L \geq 2$ and hence $\dim (L+{\Bbb Q}{\bf e}_{i_0}) \geq 3$.
Therefore $s=2$.
Hence ${\bf e}={\bf e}_1+{\bf e}_2$.
Assume that ${\bf e}_1=l_1 (0,D_1,b_1)$, where ${\bf u}:=(0,D_1,b_1)$ is primitive and $(D_1 \cdot f)=0$.
Then $(D_1^2)=0,-2$. If $(D_1^2)=-2$, then ${\bf u}=(0,D,b)$, where $D$ is an effective,
$(D^2)=-2$ and $(D \cdot f)=0$. By $((\xi-D)^2) \geq -e$, we get $2(\xi \cdot D) \leq (\xi^2)+e-2$. 
If $(D_1^2)=0$, then ${\bf u}=(0,rf,d)$ with $\gcd(r,d)=1$.
In this case, $((\xi-rf)^2) \geq -e$ implies $2r \leq (\xi^2)+e$.

(2)
For $E \in {\cal M}_H^\alpha({\bf u})^{ss}$,
if $E$ is properly semi-stable,
then we take a stable factor $E_1$ of $E$.
Then $E_1$ also defines a wall for ${\bf e}$.
Since $(H,\alpha) \in W$ is general,
$E_1$ also define $W$, and hence $\tau(E_1) \in {\Bbb Q}{\bf e}+{\Bbb Q}{\bf u}$.
Since $(c_1(E_1)\cdot f)=0$ and $\tau(E_1) \in {\Bbb Q}{\bf e}+{\Bbb Q}{\bf u}$, 
$\tau(E_1) \in {\Bbb Q}{\bf u}$. Therefore $E$ is stable.
\end{proof}

\begin{NB}
\begin{equation}
\frac{a-(\xi,\alpha)}{(\xi,H)}=\frac{a_i-(\xi_i,\alpha)}{(\xi_i,H)}
\Longleftrightarrow 
(a\xi_i-a_i \xi,H)=((\xi_i,H)\xi-(\xi,H)\xi_i,H).
\end{equation}
If $a \xi_i-a_i \xi \ne 0$, then
we assume that $(a \xi_i-a_i \xi ,H) \ne 0$.
Then as an equation of $\alpha$,
$(\xi_i,a_i)$ and $(\xi_j,a_j)$ defines the same wall
if and only if 
$$
((\xi_i,H)\xi-(\xi,H)\xi_i: a(\xi_i,H)-a_i(\xi,H))= ((\xi_j,H)\xi-(\xi,H)\xi_j: a(\xi_j,H)-a_j(\xi,H)).
$$
Then we have $(\xi_j,a_j) \in {\Bbb Q}(\xi,a)+{\Bbb Q}(\xi_i,a_i)$.
\end{NB}

\begin{NB}
\begin{rem}
$M_{H_\pm}^{\alpha_\pm}({\bf e}) \to \overline{M}_H^\alpha({\bf e})$ is injective if and only if
$|\chi({\bf e},{\bf u})| > \frac{(\xi^2)+e-2}{2}$. 
\end{rem}
\end{NB}

\begin{NB}
$((\xi-kD)^2)=(\xi^2)-2k((\xi,D)+k) \geq -e$.
\end{NB}

\begin{NB}
Assume that $W$ is defined by ${\bf u}$. Then 
\begin{equation}
\dim ({\cal M}_{H_\pm}^{\alpha_\pm}({\bf e})^{ss} \setminus {\cal M}_H^\alpha({\bf e})^s)
\leq 
-\chi({\bf e}_1,{\bf e}_2)+\dim {\cal M}_{H_\mp}^{\alpha_\mp}({\bf e}_1)^{ss}+
\dim {\cal M}_{H_\mp}^{\alpha_\mp}({\bf e}_2)^{ss}
\end{equation}

\begin{equation}
\dim ({\cal M}_{H_\pm}^{\alpha_\pm}({\bf e})^{ss} \setminus {\cal M}_H^\alpha({\bf e})^s)
\leq 
-\chi({\bf e}-l{\bf u},{\bf u})+\dim {\cal M}_{H_\mp}^{\alpha_\mp}({\bf e}-l{\bf u})^{ss}+
\dim {\cal M}_{H_\mp}^{\alpha_\mp}(l{\bf u})^{ss}
\end{equation}
\end{NB}

\begin{prop}\label{prop:wall-dim}
Let $W$ be a wall defined by ${\bf u}$.  
\begin{enumerate}
\item[(1)]
Assume that ${\bf u}=(0,D,b)$ with $(D^2)=-2$.
If $(\xi \cdot D) \geq 0$, then
\begin{equation}
\dim ({\cal M}_{H_\pm}^{\alpha_\pm}({\bf e})^{ss} \setminus {\cal M}_H^\alpha({\bf e})^s)
=\dim {\cal M}_{H_\pm}^{\alpha_\pm}({\bf e})^{ss}
-((\xi \cdot D)+1).
\end{equation}
If $(\xi \cdot D)<0$, then ${\cal M}_H^\alpha({\bf e})^s =\emptyset$.
\item[(2)]
Assume that ${\bf u}=(0,rf,b)$ with $\gcd(r,b)=1$.
Then 
\begin{equation}
\dim ({\cal M}_{H_\pm}^{\alpha_\pm}({\bf e})^{ss} \setminus {\cal M}_H^\alpha({\bf e})^s)
=\dim {\cal M}_{H_\pm}^{\alpha_\pm}({\bf e})^{ss}
-(r-1).
\end{equation}
\end{enumerate}
\end{prop}

\begin{proof}
(1)
Let $E_0$ be an $\alpha$-twisted stable sheaf $E_0$ of $\tau(E_0)={\bf u}$
with respect to $H$.
Since $E_0$ is spherical, as in \cite[sect. 2.2.2]{Reflection},
we get the claim. So let us briefly explain the computation.
For simplicity, we assume that 
$$
\frac{a-(\xi \cdot \alpha_+)}{(\xi \cdot H_+)}<\frac{b-(D \cdot \alpha_+)}{(D \cdot H_+)}.
$$
We set 
$$
{\cal M}_{H_+}^{\alpha_+}({\bf e})_l^{ss}:=\{E \in {\cal M}_{H_+}^{\alpha_+}({\bf e})^{ss} \mid
\dim \Hom(E,E_0)=l \}.
$$
For any $E \in M_{H_+}^{\alpha_+}({\bf e})_l^{ss}$, we have an exact sequence 
\begin{equation}\label{eq:ext}
0 \to E_1 \to E \to E_0^{\oplus l} \to 0
\end{equation}
where $E_1$ is an $\alpha$-twisted stable sheaf of $\tau(E_1)={\bf e}-l {\bf u}$
with respect to $H$.
Conversely for $E_1 \in {\cal M}_H^\alpha({\bf e}-l{\bf u})^s$,
$\dim \Ext^1(E_0,E_1)=(\xi \cdot D)+2l$ and
any $l$-dimensional subspace $V$ of $\Ext^1(E_0,E_1)$ gives an extension
\eqref{eq:ext} such that $E \in {\cal M}_{H_+}^{\alpha_+}({\bf e})_l^{ss}$.
Therefore ${\cal M}_{H_+}^{\alpha_+}({\bf e})_l^{ss}$ is a Grassmanian bundle
($Gr((\xi \cdot D)+2l,l)$-bundle) over ${\cal M}_H^\alpha({\bf e}-l{\bf u})^s$.
Hence we see that
$$
\dim {\cal M}_{H_+}^{\alpha_+}({\bf e})_l^{ss}=\dim {\cal M}_{H_+}^{\alpha_+}({\bf e})^{ss}-
l((\xi \cdot D)+l).
$$
We set 
$$
{\cal M}_{H_-}^{\alpha_-}({\bf e})_l^{ss}:=\{E \in {\cal M}_{H_-}^{\alpha_-}({\bf e})^{ss} \mid
\dim \Hom(E_0,E)=l \}.
$$
Then we also see that
${\cal M}_{H_-}^{\alpha_-}({\bf e})_l^{ss}$ is a Grassmanian bundle
($Gr((\xi \cdot D)+2l,l)$-bundle) over ${\cal M}_H^\alpha({\bf e}-l{\bf u})^s$
and
$$
\dim {\cal M}_{H_-}^{\alpha_-}({\bf e})_l^{ss}=\dim {\cal M}_{H_-}^{\alpha_-}({\bf e})^{ss}-
l((\xi \cdot D)+l).
$$
Since
$$
{\cal M}_{H_\pm}^{\alpha_\pm}({\bf e})^{ss} \setminus {\cal M}_H^\alpha({\bf e})^s=\cup_l
{\cal M}_{H_\pm}^{\alpha_\pm}({\bf e})^{ss}_l,
$$
we get (1).

(2)
We note that $\dim {\cal M}_{H_\pm}^{\alpha_\pm}(l {\bf u})^{ss}=l$
(cf. \cite[Prop. 1.9]{Y:Enriques}). 
Hence the claims follow from \eqref{eq:HNF-dim} and Proposition \ref{prop:wall}.  
For $\{{\bf e}_1,{\bf e}_2 \}=\{ {\bf e}-l{\bf u},l{\bf u}\}$ such that
$$
\frac{a_1-(\xi_1 \cdot \alpha_\pm)}{(\xi_1 \cdot H_\pm)}<
\frac{a_2-(\xi_2 \cdot \alpha_\pm)}{(\xi_2 \cdot H_\pm)},
$$
we have 
$$
\dim {\cal F}({\bf e}_1,{\bf e}_2)=\dim {\cal M}_{H_\pm}^{\alpha_\pm}({\bf e})-l((\xi \cdot rf)-1) 
$$
where ${\bf e}_i=(0,\xi_i,a_i)$.
It is easy to see that $E \in {\cal M}_{H_\pm}^{\alpha_\pm}({\bf e})^{ss}$ for
a general $0 \subset F_1 \subset E \in {\cal F}({\bf e}_1,{\bf e}_2)$.
Hence we get the claim.
\end{proof}

\begin{cor}\label{cor:wall}
$M_{H_\pm}^{\alpha_\pm}({\bf e}) 
\to \overline{M}_H^\alpha({\bf e})$ is a divisorial contraction
if and only if there is ${\bf u}$ satisfying 
\begin{enumerate}
\item[(1)]
$-\chi({\bf e},{\bf u})=0$ for ${\bf u}=(0,D,a)$ such that $D$ is effective,
$(D^2)=-2$ and $(D \cdot f)=0$.
\item[(2)]
$-\chi({\bf e},{\bf u})=1,2$ for ${\bf u}=(0,rf,d)$.
\end{enumerate}
\end{cor}

\begin{NB}
Assume that $(D_1^2)=-2$.
If $(\xi,D_1) \geq 0$, then $(\xi_2,D_1) \geq -l_1 (D_1^2)=-2l_1$. Hence
$(\xi_2,\xi_1)-l_1^2=l_1((\xi_2,D_1)-l_1) \geq l_1^2$.
If $(\xi_2,\xi_1)-l_1^2=1$, then $l_1=1$ and $(\xi,D_1)=0$.
\end{NB}

\subsection{Birational correspondences 
between $M_{H_+}^{\alpha_+}({\bf e})$ and  $M_{H_-}^{\alpha_-}({\bf e})$.}
\label{subsect:birat}

For the walls in Proposition \ref{prop:wall-dim},
we shall construct birational correspondences 
between $M_{H_+}^{\alpha_+}({\bf e})$ and  $M_{H_-}^{\alpha_-}({\bf e})$.

(1)
Let $W$ be a wall defined by ${\bf u}=(0,D,b)$ where $D$ is effective, $(D^2)=-2$ and $(D \cdot f)=0$.
Replacing $(H_+,\alpha_+)$ by $(H_-,\alpha_-)$ if necessary,
we may assume that
$$
\frac{a-(\xi \cdot \alpha_+)}{(\xi \cdot H_+)}<\frac{b-(D \cdot \alpha_+)}{(D \cdot H_+)}.
$$

\begin{prop}\label{prop:R-isom}
Let $E_0$ be the $\alpha$-twisted stable sheaf with $\tau(E_0)={\bf u}$. Then
$R_{E_0}$ induces an isomorphism
$$
{\cal M}_{H_-}^{\alpha_-}({\bf e})^{ss} \to 
{\cal M}_{H_+}^{\alpha_+}(R_{E_0}({\bf e}))^{ss}.
$$
\end{prop}

\begin{proof}
For $E \in {\cal M}_{H_-}^{\alpha_-}({\bf e})^{ss}$,
we have an exact sequence
$$
0 \to \Hom(E_0,E) \otimes E_0 \to E \overset{\varphi}{\to}
R_{E_0}(E) \to \Ext^1(E_0,E) \otimes E_0 \to 0
$$
such that  
$\Hom(E_0,\im \varphi)=0$ and $\im \varphi$
is $\alpha$-twisted stable with respect to $H$.
Since $\Hom(E_0,R_{E_0}(E))=0$, we get $E':=R_{E_0}(E) \in 
{\cal M}_{H_+}^{\alpha_+}(R_{E_0}({\bf e}))^{ss}$.
Conversely for $E' \in {\cal M}_{H_+}^{\alpha_+}(R_{E_0}({\bf e}))^{ss}$,
we see that $E:=R_{E_0}^{-1}(E')$ is $\alpha_-$-twisted semi-stable
with respect to $H_-$. Hence our claim holds.
\end{proof}

Assume that $(D \cdot \xi) > 0$. 
We set $Z^\pm({\bf e}):={\cal M}_{H_\pm}^{\alpha_\pm}({\bf e})^{ss} 
\setminus {\cal M}_{H_\mp}^{\alpha_\mp}({\bf e})^{ss}$.
Then $\dim Z^\pm({\bf e}) \leq \dim {\cal M}_{H_\pm}^{\alpha_\pm}({\bf e})^{ss}-2$
(Proposition \ref{prop:wall-dim}) 
and we have an identification 
$$
{\cal M}_{H_-}^{\alpha_-}({\bf e})^{ss} \setminus Z^-({\bf e})
={\cal M}_{H_+}^{\alpha_+}({\bf e})^{ss} \setminus Z^+({\bf e}).
$$

Assume that $(D \cdot \xi) \leq 0$.
For $E_0 \in {\cal M}_H^\alpha({\bf u})^{ss}$,  
we set ${\bf e'}:=R_{E_0}({\bf e})$.
If $(D \cdot \xi)=0$, then ${\bf e}'={\bf e}$ and we have an isomorphism
\begin{equation}\label{eq:codim=1}
R_{E_0}: {\cal M}_{H_-}^{\alpha_-}({\bf e})^{ss}  \to  {\cal M}_{H_+}^{\alpha_+}({\bf e})^{ss}
\end{equation}
by Proposition \ref{prop:R-isom}.
If $(D \cdot \xi)<0$, then 
we also have isomorphisms
\begin{equation}
\begin{matrix}
R_{E_0}:& {\cal M}_{H_-}^{\alpha_-}({\bf e})^{ss} & \to & {\cal M}_{H_+}^{\alpha_+}({\bf e}')^{ss}\\
R_{E_0}:&{\cal M}_{H_-}^{\alpha_-}({\bf e}')^{ss} & \to & {\cal M}_{H_+}^{\alpha_+}({\bf e})^{ss}
\end{matrix}
\end{equation}
by Proposition \ref{prop:R-isom}.
We note that
$\dim Z^\pm({\bf e}') \leq \dim {\cal M}_{H_\pm}^{\alpha_\pm}({\bf e})^{ss}-2$
(Proposition \ref{prop:wall-dim}).
Combining the identifications 
$$
{\cal M}_{H_-}^{\alpha_-}({\bf e}')^{ss} \setminus Z^-({\bf e}')=
{\cal M}_{H_+}^{\alpha_+}({\bf e}')^{ss} \setminus Z^+({\bf e}'),
$$
$R_{E_0} \circ R_{E_0}$ induces a birational map
\begin{equation}\label{eq:birat}
M_{H_-}^{\alpha_-}({\bf e}) \to M_{H_+}^{\alpha_+}({\bf e}') \cdots \to M_{H_-}^{\alpha_-}({\bf e}')
\to M_{H_+}^{\alpha_+}({\bf e}).
\end{equation}

(2)
Let $W$ be a wall defined by ${\bf u}=(0,rf,d)$ with $\gcd(r,d)=1$.
We may assume that
$$
\frac{a-(\xi \cdot \alpha_+)}{(\xi \cdot H_+)}<\frac{d-(rf \cdot \alpha_+)}{(rf \cdot H_+)}.
$$
We note that
$M_H^\alpha(0,rf,d) \cong X$ and there is a universal family ${\cal P}$ 
on $X \times X$.
For the equivalence $\Phi_{X \to X}^{{\cal P}^{\vee}[1]}:{\bf D}(X) \to {\bf D}(X)$,
we have isomorphims
$$
\Phi_{X \to X}^{{\cal P}^{\vee}[1]}:{\cal M}_{H_-}^{\alpha_-}({\bf e})^{ss} \to {\cal M}_L(r,\xi',a')^{ss}
$$ 
and
$$
D_X \circ \Phi_{X \to X}^{{\cal P}^{\vee}[1]}:{\cal M}_{H_+}^{\alpha_+}({\bf e})^{ss} \to {\cal M}_L(r,-\xi',a')^{ss}. 
$$

\begin{lem}\label{lem:isotropic-wall}
\begin{enumerate}
\item[(1)]
Assume that $r=1$. Then ${\cal M}_L(1,\xi',a')^{ss} \cong {\cal M}_L(1,-\xi',a')^{ss}$.
\item[(2)]
Assume that $r=2$. Then 
we have an isomorphism 
$$
{\cal M}_L(2,\xi',a')^{ss} \to {\cal M}_L(2,-\xi',a')^{ss}
$$
 by sending $E$ to $E \otimes \det E^{\vee}$.
\item[(3)]
Assume that $r \geq 3$. We have a birational map 
$$
{\cal M}_L(r,\xi',a')^{ss} \cdots \to {\cal M}_L(r,-\xi',a')^{ss}
$$
 by sending $E$ to $E^{\vee}$.
This map is regular up to codimension $r-1$.
\end{enumerate}
\end{lem}

By \eqref{eq:codim=1} and Lemma \ref{lem:isotropic-wall} (1), (2),
we have the following proposition.

\begin{prop}\label{prop:contraction}
For the divisorial contractions in Corollary \ref{cor:wall},
we have isomorphisms
$$
{\cal M}_{H_-}^{\alpha_-}({\bf e})^{ss} \to {\cal M}_{H_+}^{\alpha_+}({\bf e})^{ss}
$$
by contravariant Fourier-Mukai transforms.
\end{prop}

By using \eqref{eq:birat} and Lemma \ref{lem:isotropic-wall},
we get the following result.
\begin{thm}\label{thm:birat-1dim}
Let $\xi$ be an effective divisor such that $(\xi \cdot f)=1$. Then
for general $(H,\alpha)$ and $(H',\alpha')$,
there is a (contravariant) Fourier-Mukai transform
$\Phi:{\bf D}(X) \to {\bf D}(X)$
which induces an isomorphism
\begin{equation}
\begin{matrix}
{\cal M}_H^\alpha(0,\xi,a)^{ss} \setminus {\cal Z} & 
\to & {\cal M}_{H'}^{\alpha'}(0,\xi,a)^{ss} \setminus {\cal Z}'\\
E & \mapsto & \Phi(E),
\end{matrix}
\end{equation}
where 
\begin{enumerate}
\item
${\cal Z} \subset {\cal M}_H^\alpha(0,\xi,a)^{ss}$ is a closed substack of
$\dim {\cal Z} \leq (\xi^2)+e+g-3$ and
\item
${\cal Z}' \subset {\cal M}_{H'}^{\alpha'}(0,\xi,a)^{ss}$ is a closed substack of 
$\dim  {\cal Z}' \leq (\xi^2)+e+g-3$.
\end{enumerate}
\end{thm}

\begin{NB}
\begin{prop}
Assume that $(c_1 \cdot f)=1$. Then
$e({\cal M}_H^\alpha({\bf e})^{ss})$ is independent of the general choice of $(H,\alpha)$.
\end{prop}

\begin{proof}
By the classification of walls, the proof is the same as in \cite[Prop. 3.15]{Y:twist2}.
\end{proof}
\end{NB}

For ${\bf e}=(r,\xi,a)$ with $\gcd(r,(\xi \cdot f))=1$,
we have an isomorphism
${\cal M}_{H_f}(r,\xi,a)^{ss} \cong {\cal M}_{H'}^{\alpha'}(0,\xi',a')$
by \cite[Prop. 3.4.5]{PerverseII}, where
$(\xi' \cdot f)=1$ and $(H',\alpha')$ depends on the choice of $r$ and $\xi$.
Combining Theorem \ref{thm:birat-1dim}, we get a desired isomorphism in Theorem \ref{thm:birat}
if there is a section of $\pi$.

\begin{NB}
Assume that all fibers of $\pi$ are irreducible.
We set ${\bf e}:=(0,\tau+nf,a)$ and assume that $a>\frac{d}{r}(n+(\sigma \cdot \tau))$.
We set $\alpha=t \frac{d}{r}\sigma$ $(0 \leq t \leq 1)$ and 
we shall study walls.
For $\tau(F)=(0,r_1 f,d_1)$,

\begin{equation}
\begin{split}
\frac{\chi_\alpha(E)}{(c_1(E) \cdot (\sigma+kf))}-\frac{\chi_\alpha(F)}{(c_1(F) \cdot (\sigma+kf))}=&
\frac{a-\frac{d}{r}(n+(\sigma \cdot \tau)t}{n+k+(\sigma \cdot \tau)}-
\frac{d_1-\frac{d}{r}r_1 t}{r_1}\\
=& \frac{r_1 a+\frac{d}{r}r_1 t k-(n+k+(\sigma \cdot \tau))d_1}{
r_1(n+k+(\sigma \cdot \tau))}
\end{split}
\end{equation}
Since $k \gg 0$, $\tau(F)$ defines a  
wall in $0<t<1$ iff $dr_1-rd_1 \geq 0$ and $0<r_1 \leq n$.

To be more precise, if $dr_1=rd_1$, then
 $r_1 a+\frac{d}{r}r_1 t k-(n+k+(\sigma \cdot \tau))d_1=r_1(a-(n+(\sigma \cdot \tau))\frac{d}{r}+(t-1)k \frac{d}{r})>0$ if $t=1$
and $<0$ if $t<1$ by $k \gg 0$. 

Approximation of walls:

$\alpha=\frac{d}{r}t \sigma=(\frac{d_1}{r_1}+\frac{(n+(\sigma \cdot \tau))}{k}\frac{d_1}{r_1}-\frac{a}{k})\sigma
\sim \frac{d_1}{r_1} \sigma$.

\end{NB}

\begin{NB}
\subsection{Birational correspondences by Fourier-Mukai transforms}

\begin{thm}\label{thm:birat0}
Let $\xi$ be an effective divisor such that $(\xi \cdot f)=1$. Then
for a general $(H,\alpha)$,
there is a (contravariant) Fourier-Mukai transform
$\Phi:{\bf D}(X) \to {\bf D}(X)$
which induces an isomorphism
\begin{equation}
\begin{matrix}
M_H^\alpha(0,\xi,a) \setminus Z & \to &  \Hilb_X^{((\xi^2)+e)/2} \times \Pic^0(X) \setminus Z'\\
E & \mapsto & \Phi(E),
\end{matrix}
\end{equation}
where 
\begin{enumerate}
\item
$Z \subset M_H^\alpha(0,\xi,a)$ is a closed subscheme of
$\dim Z \leq (\xi^2)+e+g-2$ and
\item
$Z' \subset \Hilb_X^{((\xi^2)+e)/2} \times \Pic^0(X)$ is a closed subscheme of 
$\dim  Z' \leq (\xi^2)+e+g-2$.
\end{enumerate}
In particular
${\cal M}_H^\alpha(0,\xi,a)^{ss} \ne \emptyset$ for a general $(H,\alpha)$ if and only if
$(\xi^2) \geq -e$. 
\end{thm}

\begin{proof}
We note that $M_H(0,f,1) \cong X$. 
Let ${\cal P}$ be the universal family on $X \times X$ such that
${\cal P}_{|X \times \{x \}} \in M_H(0,f,1)$. 
We consider an equivalence
\begin{equation}
\Phi_{X \to X}^{{\cal P}[1]}:{\bf D}(X) \to {\bf D}(X),
\end{equation}
where $\Phi_{X \to X}^{{\cal P}[1]}({\Bbb C}_x)={\cal P}_{|X \times \{x \}}[1]$.
Since $(\xi \cdot f)=1$,
$\Phi_{X \to X}^{{\cal P}[1]}((0,\xi,a))=(1,\xi',a')$.
Replacing $M_H(0,f,1)$ by $M_H^{-\xi'}(0,f,1-(\xi' \cdot f))$ and
${\cal P}$ by ${\cal P} \otimes p_1^*(L)$ ($L \in \Pic(X)$, $c_1(L)=-\xi'$),
we may assume that $\xi'=0$. 
\begin{NB2}
$\chi({\cal O}_X)-2 \deg Z=\chi(I_Z,I_Z)=
\chi(\Phi_{X \to X}^{{\cal P}[1]}(E),\Phi_{X \to X}^{{\cal P}[1]}(E))=\chi(E,E)=-(\xi^2)$.
$(\xi^2)+e=2\deg z$.
\end{NB2}
By our assumption, $2a' =e-(\xi^2) \leq 2e$.
Hence ${\cal M}_{H'}(1,0,a')^{ss} \ne \emptyset$.
By \cite[Prop. 3.4.5]{PerverseII},
we have an isomorphism
\begin{equation}
\begin{matrix}
{\cal M}_{H'}(1,0,a')^{ss} & \to & {\cal M}_H^\alpha(0,\xi,a)^{ss}\\
E & \mapsto & \Phi_{X \to X}^{{\cal P}[1]}(E^{\vee}),
\end{matrix}
\end{equation}
where $(H,\alpha)$ depends on the choice of $H'$.
By Theorem \ref{thm:birat-1dim},
we get the claim.
\end{proof}
\end{NB}

\section{Stability associated to Fourier-Mukai transforms}\label{sect:stability}

In this section, we shall introduce a stability condition for coherent sheaves
with torsion associated to Fourier-Mukai transforms, and study some properties.

\subsection{$\lambda$-stability}\label{subsect:n-stability}
In this subsection, we set 
\begin{equation}\label{eq:e}
{\bf e}:=(r,\xi,a) \in {\Bbb Z} \oplus \NS(X) \oplus {\Bbb Z}.
\end{equation}
We introduce a notion of stability which is related to a relative Fourier-Mukai
transform.

\begin{defn}\label{defn:n-stable}
Let $H$ be an ample divisor on $X$ and $\alpha$ a ${\Bbb Q}$-divisor on $X$.
Let $\lambda$ be a real number.
A coherent sheaf $E$ is $\lambda$-stable (resp. $\lambda$-semi-stable) if
the following three conditions hold. 
\begin{enumerate}
\item[(a)]
The restriction $E_\eta:=E_{|X_\eta}$ of $E$ to the generic fiber $X_\eta$ of $\pi$
is a semi-stable vector bundle.
\item[(b)]
$\Hom(E,A)=0$ for all $\alpha$-twisted stable sheaf $A$
on a fiber with $\frac{\chi_\alpha(A)}{(c_1(A) \cdot H)} \leq \lambda$
(resp. $\frac{\chi_\alpha(A)}{(c_1(A) \cdot H)} < \lambda$).
\item[(c)]
$\Hom(A,E)=0$ for all $\alpha$-twisted stable sheaf $A$
on a fiber with
 $\frac{\chi_\alpha(A)}{(c_1(A) \cdot H)}>\lambda$.
\end{enumerate}
\end{defn}

\begin{rem}\label{rem:n-stable}
\begin{enumerate}
\item[(1)]
We note that $\Hom(E,E_{|f}) \ne 0$ for a coherent sheaf $E$ with $\rk E>0$.
Hence if $\lambda \geq \frac{(\xi \cdot f)-r(\alpha \cdot f)}{r(f \cdot H)}$,
then there is no $\lambda$-stable sheaf $E$ with $\tau(E)={\bf e}$.
Hence we may assume that $\lambda<\frac{(\xi \cdot f)-r(\alpha \cdot f)}{r(f \cdot H)}$.
\item[(2)]
By the condition (c), $E$ does not contain a non-trivial 0-dimensional subsheaf.
In particular $E$ has a locally free resolution of length 1.
\end{enumerate}
\end{rem}

\begin{lem}\label{lem:trace-property}
Assume that $\gcd((\xi \cdot f),r)=1$ and $\gcd(r,\chr k)=1$.
For a $\lambda$-stable sheaf $E$ with $\tau(E)={\bf e}$, 
$$ 
H^0(X,{\cal O}_X(D)) \cong \Hom(E,E(D)),
$$ 
where $D$ is a divisor such that
$D \equiv \lambda f$ $(\lambda \in {\Bbb Q})$.
\end{lem}

\begin{proof}
For a $\lambda$-stable sheaf $E$ with $\tau(E)={\bf e}$,
we have an exact sequence
$$
0 \to T \to E \to E' \to 0
$$
such that $T$ is the torsion subsheaf and $E'$ is the torsion free quotient.
Then $T$ is generated by $\alpha$-twisted stable sheaves
$A$ with $\frac{\chi_\alpha(A)}{(c_1(A) \cdot H)} \leq \lambda$.
Hence $\Hom(E,T(D))=0$.
We have an exact sequence
$$
\Hom(E,T(D)) \to \Hom(E,E(D)) \to \Hom(E,E'(D)).
$$
Since $\Hom(E,T(D))=0$, $\Hom(E,E(D)) \to \Hom(E',E'(D))$ is injective.
Hence we have a homomorphism
$$
\psi:H^0(X,{\cal O}_X(D)) \overset{\varphi}{\to} \Hom(E,E(D)) \to \Hom({E'}^{\vee \vee},{E'}^{\vee \vee}(D)) 
\overset{\tr}{\to} H^0(X,{\cal O}_X(D)).
$$
Since $\psi$ is multiplication by $r$, it is an isomorphism.
By the stability of $E_\eta \cong E'_\eta$, we get $\ker \tr=0$.
Therefore $\varphi$ is an isomorphism.
\end{proof}

Let $Y:=M_H^\alpha(0,r_1 f,d_1)$ be a fine moduli space
and ${\cal P}$ a universal family on $X \times Y$,
where $(H,\alpha)$ is general.
Let $\Phi_{X \to Y}^{{\cal P}^{\vee}}:{\bf D}(X) \to {\bf D}(Y)$ be 
the Fourier-Mukai transform whose kernel is ${\cal P}^{\vee}$.
\begin{lem}\label{lem:FM-property}
Let $E$ be a coherent sheaf on $X$.
\begin{enumerate}
\item[(1)]
 $H^0(\Phi_{X \to Y}^{{\cal P}^{\vee}}(E))$ is torsion free.
\item[(2)]
If $\Hom({\cal P}_{|X \times \{y \}},E)=0$ except for finitely many $y \in Y$, then 
$H^1(\Phi_{X \to Y}^{{\cal P}^{\vee}}(E))$ is torsion free.
\end{enumerate}
\end{lem}

\begin{proof}
We take a locally free resolution
\begin{equation}
0 \to V_{-2} \to V_{-1} \to V_0 \to {\cal P} \to 0 
\end{equation}
such that $R^j p_{Y*} (V_i^{\vee} \otimes E)=0$ for $i=0,-1$ and $j>0$.
\begin{NB}
For a sufficiently large $n$,
we set $V_0:=p_Y^*( p_{Y*}({\cal P} \otimes {\cal O}_X(nH))) \otimes {\cal O}_X(-nH)$.
Then $V_0 \to {\cal P}$ is surjective and ${cal P}^1:=\ker(V_0 \to {\cal P})$
is flat over $Y$. Since ${\cal P}_{|X \times \{ y \}}$ is purely 1-dimensional,
${\cal P}^1_{|X \times \{ y \}}$ is locally free.
Hence ${\cal P}^1$ is locally free. We shall apply the same procedure to
${\cal P}^1$ to get $V_{-1} \to {\cal P}^1$ and denote the kernel by $V_{-2}$.  
\end{NB}
Then $\Phi_{X \to Y}^{{\cal P}^{\vee}}(E)$ is represented by a complex
of locally free sheaves on $Y$:
\begin{equation}
0 \to p_{Y*}(V_0^{\vee} \otimes E) \overset{\varphi_0}{\to} p_{Y*}(V_{-1}^{\vee} \otimes E) \overset{\varphi_1}{\to}
p_{Y*}(V_{-2}^{\vee} \otimes E) \to 0.
\end{equation}
Hence $H^0(\Phi_{X \to Y}^{{\cal P}^{\vee}}(E))=\ker \varphi_0$ is torsion free.
If $\Hom({\cal P}_{|X \times \{y \}},E)=0$ except for finitely many points
$y \in Y$, then $\varphi_0$ is injective and
$\coker \varphi_0$ is torsion free. 
Therefore $H^1(\Phi_{X \to Y}^{{\cal P}^{\vee}}(E))$ is also torsion free.
\end{proof}

%

\begin{lem}\label{lem:A}
Let $A$ be an $\alpha$-twisted stable sheaf on a fiber.
\begin{enumerate}
\item[(1)]
Assume that 
$\frac{\chi_\alpha(A)}{(c_1(A) \cdot H)}> 
\frac{\chi_\alpha({\cal P}_{|X \times \{y \}})}{(r_1 f \cdot H)}$.
Then $\Phi_{X \to Y}^{{\cal P}^{\vee}}(A)[1] \in \Coh(Y)$.
\item[(2)]
Assume that 
$\frac{\chi_\alpha(A)}{(c_1(A) \cdot H)}
\leq \frac{\chi_\alpha({\cal P}_{|X \times \{y \}})}{(r_1 f \cdot H)}$.
Then $\Phi_{X \to Y}^{{\cal P}^{\vee}}(A)[2] \in \Coh(Y)$.
\end{enumerate}
\end{lem}

\begin{proof}
We note that $H^i(\Phi_{X \to Y}^{{\cal P}^{\vee}}(A))$
are fiber sheaves.
By Lemma \ref{lem:FM-property} (1), we get
$H^0(\Phi_{X \to Y}^{{\cal P}^{\vee}}(A))=0$.

If $\frac{\chi_\alpha(A)}{(c_1(A) \cdot H)}> 
\frac{\chi_\alpha({\cal P}_{|X \times \{y \}})}{(r_1 f \cdot H)}$, then
$$
\Ext^2({\cal P}_{|X \times \{y \}},A) \cong \Hom(A,{\cal P}_{|X \times \{y \}}(K_X))^{\vee}=0
$$
for all $y \in Y$.
Hence $H^2(\Phi_{X \to Y}^{{\cal P}^{\vee}}(A))=0$ and
$\Phi_{X \to Y}^{{\cal P}^{\vee}}(A))[1] \in \Coh(Y)$.

If $\frac{\chi_\alpha(A)}{(c_1(A) \cdot H)}
\leq \frac{\chi_\alpha({\cal P}_{|X \times \{y \}})}{(r_1 f \cdot H)}$,
then $\Hom({\cal P}_{|X \times \{y \}},A)=0$ except for finitely many points
$y \in Y$. By Lemma \ref{lem:FM-property} (2),
we get $H^1(\Phi_{X \to Y}^{{\cal P}^{\vee}}(A))=0$.
Therefore $\Phi_{X \to Y}^{{\cal P}^{\vee}}(A))[2] \in \Coh(Y)$.
\end{proof}

\begin{lem}\label{lem:n-stability1}
Let $F$ be a torsion free sheaf on $Y$ such that $F_\eta$ is semi-stable and
$r_1 (c_1(F) \cdot f)+\rk F d_1'<0$,
where $\tau({\cal P}_{|\{x \} \times Y})=(0,r_1 f,d_1')$.
Then $E:=\Phi_{Y \to X}^{{\cal P}}(F)[1] \in \Coh(X)$ and  
$\lambda$-stable, where 
$\lambda=\frac{\chi_\alpha({\cal P}_{|X \times \{y \}})}{(r_1 f \cdot H)}$.
\end{lem}

\begin{proof}
Since ${\cal P}$ is flat over $Y$, 
${\cal P} \otimes p_Y^*(F) \in \Coh(X \times Y)$.
Hence $H^i(\Phi_{Y \to X}^{{\cal P}}(F))=0$ for $i \ne 0,1$.
Since ${\cal P}$ is flat over $X$ (cf. \cite[Lem. 5.1]{Br:1}), 
$H^0(\Phi_{Y \to X}^{{\cal P}}(F^{\vee \vee}))$ is torsion free.
Hence $H^0(\Phi_{Y \to X}^{{\cal P}}(F))$ is also torsion free.
By our assumption,
$H^0(\Phi_{Y \to X}^{{\cal P}}(F))$ is supported on fibers.
Hence we get $H^0(\Phi_{Y \to X}^{{\cal P}}(F))=0$.
Thus $E:=\Phi_{Y \to X}^{{\cal P}}(F)[1] \in \Coh(X)$.
For an $\alpha$-twisted stable sheaf $A$ on a fiber and
a homomorphism $E \to A$, 
if $\frac{\chi_\alpha(A)}{(c_1(A) \cdot H)} \leq  \lambda$, then
$\Phi_{X \to Y}^{{\cal P}^{\vee}}(A)[2] \in \Coh(Y)$.
Hence
$$
\Hom(E,A)=\Hom(F(-K_X), \Phi_{X \to Y}^{{\cal P}^{\vee}}(A)[1])=0.
$$

Since $F$ is torsion free, we also see that
$\Hom(A,E)=0$ for an $\alpha$-twisted stable sheaf $A$ on a fiber
such that $\frac{\chi_\alpha(A)}{(c_1(A) \cdot H)}> \lambda$.
\end{proof}

\begin{lem}\label{lem:n-stability2}
Assume that $r_1(\xi \cdot f)-r d_1>0$.
We set 
$\lambda:=
\frac{\chi_\alpha({\cal P}_{|X \times \{ y\}})}{(c_1({\cal P}_{|X \times \{ y\}}) \cdot H)}
=\frac{d_1-r_1(f \cdot \alpha)}{r_1(f \cdot H)}$.
For $\lambda' \geq \lambda$ and
a $\lambda'$-stable sheaf $E$, $\Phi_{X \to Y}^{{\cal P}^{\vee}}(E)[1] \in \Coh(Y)$.
Moreover $\Phi_{X \to Y}^{{\cal P}^{\vee}}(E)[1]$ is torsion free if 
$\lambda'=\lambda$. 
\end{lem}

\begin{proof}
By $\lambda' \geq \lambda$ and the definition of $\lambda'$-stability, we have 
 $$
\Ext^2({\cal P}_{|X \times \{ y\}},E) \cong \Hom(E,{\cal P}_{|X \times \{ y\}})^{\vee}=0
$$
for all $y \in Y$.
Hence $H^2(\Phi_{X \to Y}^{{\cal P}^{\vee}}(E))=0$.
Since $E_\eta$ is locally free, $H^0(\Phi_{X \to Y}^{{\cal P}^{\vee}}(E))$ is supported on fibers.
Then Lemma \ref{lem:FM-property} implies $H^0(\Phi_{X \to Y}^{{\cal P}^{\vee}}(E))=0$.
Therefore
$\Phi_{X \to Y}^{{\cal P}^{\vee}}(E)[1] \in \Coh(Y)$.

Assume that $\lambda=\lambda'$.
For a non-trivial homomorphism $\phi:{\cal P}_{|X \times \{ y\}} \to E$,
$\lambda$-stability of $E$ implies $\phi$ is injective and $\coker \phi$ is also
$\lambda$-stable. Since $\im \phi$ is contained in the torsion submodule of $E$,
$\Hom({\cal P}_{|X \times \{ y\}},E)=0$ except for finitely many points $y \in Y$.
Then $\Phi_{X \to Y}^{{\cal P}^{\vee}}(E)[1]$ is torsion free by Lemma \ref{lem:FM-property} (2).
\end{proof}

\begin{rem}
If $\lambda'<\lambda$, then
$(\Phi_{X \to Y}^{{\cal P}^{\vee}[1]}(E))^{\vee} \in \Coh(Y)$ for a $\lambda'$-stable sheaf $E$. 
\end{rem}

By Lemma \ref{lem:n-stability1} and Lemma \ref{lem:n-stability2},
we get the following result.

\begin{prop}\label{prop:n-stability}
Assume that $\gcd((\xi \cdot f),r)=1$ and
$r_1(\xi \cdot f)-r d_1>0$.
We set 
$\lambda:=
\frac{\chi_\alpha({\cal P}_{|X \times \{ y\}})}{(c_1({\cal P}_{|X \times \{ y\}}) \cdot H)}$.
Let $E$ be an object of ${\bf D}(X)$.
Then $E \in \Coh(X)$ and $E$ is $\lambda$-stable if and only if
$E':=\Phi_{X \to Y}^{{\cal P}^{\vee}[1]}(E) \in \Coh(Y)$
and $E'$ is stable with respect to $H_f'$, where $H'$ is an ample divisor on $Y$.
\end{prop}

We shall give a generalization of this correspondence in subsection \ref{subsect:n-stable2}. 

\begin{NB}
For $(r_1,d_1)$, $(dr1',d_1'):=(r_1 N \prod_i m_i, d_1 N \prod_i m_i+1)$
satisfies $|\frac{d_1}{r_1}-\frac{d_1'}{r_1'}|$ is sufficiently small if $N \gg 0$.
So for a general $\lambda$, we have $M_H^\alpha (0,r' f, d')$ such that 
 we can apply Lemma \ref{lem:n-stability1}.
\end{NB}

\subsection{Stacks of $\lambda$-stable sheaves.}\label{subsect:stack}
In this subsection, we shall consider moduli stacks of $\lambda$-stable sheaves.
We quote the following Bogomolov inequality.
\begin{lem}[{\cite[Lem. 3.3]{Y:Enriques}}]\label{lem:Bogomolov}
For a torsion free sheaf $E$ on $X$ such that $E_\eta$ is semi-stable,
\begin{equation}\label{eq:Bogomolov}
c_2(E)-\frac{\rk E-1}{2\rk E}(c_1(E)^2)=
\rk E \chi({\cal O}_X)-\frac{1}{2}(c_1(E) \cdot K_X)+\frac{1}{2\rk E}(c_1(E)^2)-\chi(E) \geq 0.
\end{equation}
\end{lem}

\begin{NB}
\begin{prop}\label{prop:finite}
Assume that $r'(\xi \cdot f)-rd'>0$. 
Let $B$ be a compact subset of $\NS(X)_{\Bbb R}$ and take $\alpha \in B$.
Let ${\cal T}_B$ be the set of torsion submodules $T$ of 
$\lambda$-stable sheaves $E$
with $\tau(E)={\bf e}$, where $\lambda 
\leq \frac{d'-r'(f \cdot \alpha)}{r'(f \cdot H)}$.
Then $\{ \tau(T) \mid T \in {\cal T}_B \}$ is a finite set.
\end{prop}

\begin{proof}
Let $m_1 f_1,...,m_k f_k$ be the reducible fibers of $\pi$
and let $f_i=\sum_{j=0}^{n_i}a_{ij} C_{ij}$
be the decomposition of $f_i$, where $m_i$ are the 
multiplicities and $C_{ij}$ are the irreducible components of 
$f_i$.
We may assume that $a_{i0}=1$. 
Then 
$$
\sum_{j=0}^{n_i} {\Bbb Z}C_{ij}=
\sum_{j=1}^{n_i} {\Bbb Z}C_{ij} + {\Bbb Z}f_i.
$$ 
We set $L:=\sum_i \sum_{j=1}^{n_i} {\Bbb Z}C_{ij}$.

Let $T$ be the torsion submodule of $E$ and set
$\tau(T)=(0,g+D,l d'+b)$, where
$D \in L$, $l \in {\Bbb Q}_{>0}$, $g=lr' f \in \NS(X)_{\Bbb Q}$ and $b \in {\Bbb Q}$.
Then $\tau(E/T)=(r,\xi-g-D,a-ld'-b)$.
Let $A$ be an $\alpha$-twisted stable subsheaf of $T$ with the maximal slope.
Then the $\lambda$-stability of $E$ implies 
$\frac{\chi_\alpha(A)}{(c_1(A) \cdot H)} \leq \lambda$.
Hence we get
\begin{equation}
\frac{\chi_\alpha(T)}{(c_1(T) \cdot H)} \leq 
\frac{\chi_\alpha(A)}{(c_1(A) \cdot H)} \leq \lambda 
\leq \frac{d'-r'(f \cdot \alpha)}{r'(f \cdot H)}.
\end{equation}
\begin{NB2}
Since $L$ is negative definite, we have a orthogonal decomposition 
$L_{\Bbb Q}=\sum_i {\Bbb Q}e_i$.
We set $D=\sum_i x_i e_i$. Then
\begin{equation}
\begin{split}
((\xi-D)^2)=& (\xi^2)-2\sum_i (\xi \cdot e_i)x_i+\sum_i (e_i^2)x_i^2 \\
=& \sum_i (e_i^2) \left(x_i-\frac{(\xi \cdot e_i)}{(e_i^2)} \right)^2-
\sum_i \frac{(\xi \cdot e_i)^2}{(e_i^2)}+(\xi^2).
\end{split}
\end{equation}
Hence $((\xi-D)^2)$ is bounded above.
\end{NB2}
Then we have
\begin{equation}
\frac{b-(D \cdot \alpha)}{(D \cdot H)} \leq 
\frac{d'-r'(f \cdot \alpha)}{r'(f \cdot H)}.
\end{equation}
\begin{NB2}
\begin{equation}
\begin{split}
&((b-(D \cdot \alpha))+l(d'-(r' f \cdot \alpha)))(r' f \cdot H)-
(d'-(r' f \cdot \alpha))((D \cdot H)+l(r' f \cdot H))\\
=&(b-(D \cdot \alpha))(r' f \cdot H)-(d'-(r' f \cdot \alpha))(D \cdot H)
\end{split}
\end{equation}
\end{NB2}
Since $L$ is negative definite, 
\begin{equation}
|(D \cdot H)| \leq \sqrt{-(D^2)} \sqrt{(H_1^2)},\;
|(D \cdot \alpha)| \leq \sqrt{-(D^2)}\sqrt{-(\alpha_1^2)},
\end{equation} 
where $H_1, \alpha_1 \in L_{\Bbb R}$ are the orthgonal projections of $H,\alpha$ respectively.
Hence there is a positive number $\beta$ depending on $B$ such that  
\begin{equation}
b \leq \beta \sqrt{-(D^2)}.
\end{equation}
We note that 
\begin{equation}
\frac{d'-r' (f \cdot \alpha)}{(r' f \cdot H)}
<\frac{(\xi \cdot f)-(rf \cdot \alpha)}{(rf \cdot H)}
\Longleftrightarrow 
r' (\xi \cdot f)-rd'>0.
\end{equation}
We have 
$$
|(D \cdot \xi)| \leq \sqrt{-(D^2)}\sqrt{-(\xi_1^2)}
$$
where $\xi_1 \in L$ is the orthognal projection of $\xi$.
Then we see that
\begin{equation}
\begin{split}
& ((\xi-D-lr' f)^2)-2r(a-b-ld')\\
=& ((\xi-D)^2)-2l(r'(\xi \cdot f)-rd')-2ra+2rb\\
\leq & -2l(r'(\xi \cdot f)-rd')-
\left(\sqrt{-(D^2)}-\left(\sqrt{-(\xi_1^2)}+r\beta \right)\right)^2+
\left(\sqrt{-(\xi_1^2)}+r\beta \right)^2+(\xi^2)-2ra.
\end{split}
\end{equation}
Applying Bogomolov inequality (Lemma \ref{lem:Bogomolov}) for
$E/T$,  we see that the choice of $l$ and $D$ are finite.
Then the choice of $b$ is also finite. Therefore the choice of $\tau(T)$ is finite.  
\end{proof}

\end{NB}

\begin{defn}\label{defn:E-A-T}
Assume that $r'(\xi \cdot f)-rd'>0$. 
Let $B$ be a compact subset of $\NS(X)_{\Bbb R}$.
We assume that $\alpha  \in B$ is  
a ${\Bbb Q}$-divisor.
\begin{enumerate}
\item[(1)]
Let ${\cal E}_B$ be the set of coherent sheaves
$E$ such that 
\begin{enumerate}
\item
$\tau(E)={\bf e}$,
\item
$E_\eta$ is semi-stable and
\item
$\frac{\chi_\alpha(A')}{(c_1(A') \cdot H)} \leq \frac{d'-r'(f \cdot \alpha)}{r'(f \cdot H)}$
for any $\alpha$-twisted stable subsheaf $A'$ of $E$.
\end{enumerate}
\item[(2)]
Let ${\cal A}_B$ be the set of $\alpha$-twisted stable sheaves $A$ 
which are quotients of $E \in {\cal E}_B$ such that
\begin{equation}\label{eq:A-cond}
\chi_\alpha(A) \leq \frac{d'-r'(f \cdot \alpha)}{r'(f \cdot H)}(c_1(A) \cdot H).
\end{equation}
\item[(3)]
Let ${\cal T}_B$ be the set of torsion submodules $T$ of 
$E \in {\cal E}_B$.
\end{enumerate}
\end{defn}

\begin{prop}\label{prop:finite}
Assume that $r'(\xi \cdot f)-rd'>0$. 
Let $B$ be a compact subset of $\NS(X)_{\Bbb R}$ and take $\alpha \in B$.
\begin{enumerate}
\item[(1)]
$\{\tau(A) \mid A \in {\cal A}_B \}$ is a finite set.
\item[(2)]
$\{ \tau(T) \mid T \in {\cal T}_B \}$ is a finite set.
\end{enumerate}
\end{prop}

\begin{proof}
Let $m_1 f_1,...,m_k f_k$ be the reducible fibers of $\pi$
and let $f_i=\sum_{j=0}^{n_i}a_{ij} C_{ij}$
be the decomposition of $f_i$, where $m_i$ are the 
multiplicities and $C_{ij}$ are the irreducible components of 
$f_i$.
We may assume that $a_{i0}=1$. 
Then 
$$
\sum_{j=0}^{n_i} {\Bbb Z}C_{ij}=
\sum_{j=1}^{n_i} {\Bbb Z}C_{ij} + {\Bbb Z}f_i.
$$ 
We set $L:=\sum_i \sum_{j=1}^{n_i} {\Bbb Z}C_{ij}$.
 
For $E \in {\cal E}_B$ and a quotient $E \to A$ $(A \in {\cal A}_B)$ with \eqref{eq:A-cond},
we set $E':=\ker(E \to A)$ and set
$\tau(A)=(0,l_1 r' f+D_1,l_1 d'+b_1)$,
where
$D_1 \in L$, $l_1 \in {\Bbb Q}_{ \geq 0}$ and $b_1 \in {\Bbb Q}$.
Let $T$ be the torsion submodule of $E'$ and set
$\tau(T)=(0,l_2 r' f+D_2,l_2 d'+b_2)$, where
$D_2 \in L$, $l_2 \in {\Bbb Q}_{ \geq 0}$ and $b_2 \in {\Bbb Q}$.
we also set $l:=l_1+l_2,D:=D_1+D_2, b:=b_1+b_2$.
Then $\tau(E'/T)=(r,\xi-lr'f-D,a-ld'-b)$.
Let $A'$ be an $\alpha$-twisted stable subsheaf of $T$ with the maximal slope.
Since $E'$ is a subsheaf of $E \in {\cal E}_B$, 
we get
\begin{equation}\label{eq:T}
\frac{\chi_\alpha(T)}{(c_1(T) \cdot H)} \leq 
\frac{\chi_\alpha(A')}{(c_1(A') \cdot H)} 
\leq \frac{d'-r'(f \cdot \alpha)}{r'(f \cdot H)}.
\end{equation}
\begin{NB}
Since $L$ is negative definite, we have a orthogonal decomposition 
$L_{\Bbb Q}=\sum_i {\Bbb Q}e_i$.
We set $D=\sum_i x_i e_i$. Then
\begin{equation}
\begin{split}
((\xi-D)^2)=& (\xi^2)-2\sum_i (\xi \cdot e_i)x_i+\sum_i (e_i^2)x_i^2 \\
=& \sum_i (e_i^2) \left(x_i-\frac{(\xi \cdot e_i)}{(e_i^2)} \right)^2-
\sum_i \frac{(\xi \cdot e_i)^2}{(e_i^2)}+(\xi^2).
\end{split}
\end{equation}
Hence $((\xi-D)^2)$ is bounded above.
\end{NB}
By \eqref{eq:A-cond} and \eqref{eq:T},
$$
\chi_\alpha(A)+\chi_\alpha(T) \leq  \frac{d'-r'(f \cdot \alpha)}{r'(f \cdot H)}((c_1(A)+c_1(T))\cdot H).
$$
Since $\tau(A)+\tau(T)=(0,l r' f+D,ld'+b)$, we have
\begin{equation}
\frac{b-(D \cdot \alpha)}{(D \cdot H)} \leq 
\frac{d'-r'(f \cdot \alpha)}{r'(f \cdot H)}.
\end{equation}
\begin{NB}
\begin{equation}
\begin{split}
&((b-(D \cdot \alpha))+l(d'-(r' f \cdot \alpha)))(r' f \cdot H)-
(d'-(r' f \cdot \alpha))((D \cdot H)+l(r' f \cdot H))\\
=&(b-(D \cdot \alpha))(r' f \cdot H)-(d'-(r' f \cdot \alpha))(D \cdot H)
\end{split}
\end{equation}
\end{NB}
Since $L$ is negative definite, 
\begin{equation}
|(D \cdot H)| \leq \sqrt{-(D^2)} \sqrt{-(H_L^2)},\;
|(D \cdot \alpha)| \leq \sqrt{-(D^2)}\sqrt{-(\alpha_L^2)},
\end{equation} 
where $H_L, \alpha_L \in L_{\Bbb R}$ are the orthgonal projections of $H,\alpha$ respectively.
Hence there is a positive number $\beta$ depending on $B$ such that  
\begin{equation}
b \leq \beta \sqrt{-(D^2)}.
\end{equation}
We note that 
\begin{equation}
\frac{d'-r' (f \cdot \alpha)}{(r' f \cdot H)}
<\frac{(\xi \cdot f)-(rf \cdot \alpha)}{(rf \cdot H)}
\Longleftrightarrow 
r' (\xi \cdot f)-rd'>0.
\end{equation}
We have 
$$
|(D \cdot \xi)| \leq \sqrt{-(D^2)}\sqrt{-(\xi_L^2)}
$$
where $\xi_L \in L$ is the orthognal projection of $\xi$.
Then we see that
\begin{equation}
\begin{split}
& ((\xi-D-lr' f)^2)-2r(a-b-ld')\\
=& ((\xi-D)^2)-2l(r'(\xi \cdot f)-rd')-2ra+2rb\\
\leq & -2l(r'(\xi \cdot f)-rd')-
\left(\sqrt{-(D^2)}-\left(\sqrt{-(\xi_L^2)}+r\beta \right)\right)^2+
\left(\sqrt{-(\xi_L^2)}+r\beta \right)^2+(\xi^2)-2ra.
\end{split}
\end{equation}
Applying Bogomolov inequality (Lemma \ref{lem:Bogomolov}) for
$E'/T$,  we see that the choice of $l$ and $D$ are finite.
Then the choice of $b$ is also finite. 
Since $l_1,l_2 \geq 0$, the choice of $l_1$ and $l_2$ are finite.
Since $l_1 r'f+D_1$ and $l_2 r' f+D_2$ are effective,
the choice of $D_1$ and $D_2$ are also finite.
Since $b_1$ and $b_2$ are bounded above and $b=b_1+b_2$, the choice of
$b_1$ and $b_2$ are finite.  
Therefore the choice of $\tau(A)$ and $\tau(T)$ are finite, which implies
the claim (1).
The proof of (2) is similar. 
\end{proof}

\begin{NB}
\begin{lem}
$\{[D] \in \NS(X) \mid (D \cdot H)=d, D>0 \}$ is a finite set
\end{lem}
\begin{proof}
Let $\NS(X)_{\Bbb Q}={\Bbb Q}H \oplus {\Bbb Q}D_1 \oplus \cdots
{\Bbb Q}D_n$ be an orthogonal decomposition.
We set $D=d/(H^2)H+\sum_i x_i D_i$.
We take ${\Bbb Q}$-ample divisors $H \pm \epsilon D_i$.
Since $(H \pm \epsilon D_i,D)>0$,
$d+x_i \epsilon_i (D_i^2)>0$ and $d-x_i \epsilon_i (D_i^2)>0$. 
Hence
$|x_i|| \epsilon_i (D_i^2)|<d$. Thus the choice of $x_i$ are bounded. 
\end{proof}
\end{NB}

\begin{lem}\label{lem:open}
\begin{enumerate}
\item[(1)]
$\lambda$-semi-stability is an open condition.
\item[(2)]
If $\gcd((\xi \cdot f),r)=1$, 
then the set of $\lambda$-semi-stable sheaves $E$ of $\tau(E)={\bf e}$ is bounded.
\end{enumerate}
\end{lem}

\begin{proof}
(1)
We shall check that three conditions of Definition \ref{defn:n-stable} are open conditions.
For a coherent sheaf $E$,
$E_\eta$ is semi-stable if and only if there is a smooth fiber $f$ such that
$E_{|f}$ is a semi-stable vector bundle. Hence it is an open condition.

\begin{NB}
For a $\lambda$-semi-stable sheaf $E$, let $T$ be the torsion submodule of $E$.
We consider the set
$$
{\cal A}:=\{\tau (A) \mid \text{ $A$ is $\alpha$-twisted stable and
 $A$ is a quotient of $E$ with $\chi_\alpha(A) \leq \lambda (c_1(A) \cdot H)$ } \}.
$$
We set $T':=\im(T \to A)$ and $A':=A/T'$.
By Proposition \ref{prop:finite}, the set of $\tau(T)$ is finite.
By the $\lambda$-stability of $E$, the choice of Harder-Narasimhan filtration of
$T$ is finite, and hence the choice of $\tau(T')$ is also finite.
 Since $A'$ is a quotient of $E/T$, by a similar argument in the proof
of Proposition \ref{prop:finite},
the choice of $\tau(A')$ is also finite. 
Hence ${\cal A}$ is a finite set.
Therefore 
$\Hom(E,A)=0$ in Definition \ref{defn:n-stable} is an open condition.
\end{NB}
Let ${\cal E}$ be a $S$-flat family of coherent sheaves on $X$ 
with $\tau({\cal E}_{|X \times \{s \}})={\bf e}$ $(s \in S)$.
Since the set of torsion submodules $T$ of ${\cal E}_{|X \times \{s \}}$ is bounded,
the set of Harder-Narasimhan filtrations of $T$ is also bounded. Hence
$\Hom(A,E)=0$ in Definition \ref{defn:n-stable} is an open condition. 
By Proposition \ref{prop:finite},
$$
{\cal A}:=\{\tau (A) \mid \text{ $A$ is $\alpha$-twisted stable and
 $A$ is a quotient of $E$ with $\chi_\alpha(A) \leq \lambda (c_1(A) \cdot H)$ } \}
$$
is a finite set. Hence
$\Hom(E,A)=0$ in Definition \ref{defn:n-stable} is an open condition.

(2)
For a $\lambda$-semi-stable sheaf $E$ with $\tau(E)={\bf e}$, let $T$ be the torsion 
submodule of $E$. Then
the choice of $\tau(T)$ is finite. 
Hence the choice of $\tau(E/T)$ is finite,
which implies the set of $E/T$ is bounded.
Hence our claim holds.
\end{proof}

\begin{prop}\label{prop:JHF}
For a $\lambda$-semi-stable sheaf $E$, there is a filtration
$$
0 \subset F_1 \subset F_2 \subset F_3=E 
$$
such that 
\begin{enumerate}
\item[(i)]
$F_1$ and $F_3/F_2$ are $\alpha$-twisted semi-stable 1-dimensional sheaves with 
$\frac{\chi_\alpha(F_1)}{(c_1(F_1) \cdot H)}=\frac{\chi_\alpha(F_3/F_2)}{(c_1(F_3/F_2) \cdot H)}=\lambda$. 
\item[(ii)]
$F_2/F_1$ is a $\lambda$-stable sheaf 
such that
$\Hom(A,F_2/F_1)=0$ for any $\alpha$-twisted stable sheaf $A$ with
 $\frac{\chi_\alpha(A)}{(c_1(A) \cdot H)}=\lambda$.
\end{enumerate}
\end{prop}

\begin{proof}
For a $\lambda$-semi-stable sheaf $E$ and 
a 1-dimensional subsheaf $A$ of $E$ with $\frac{\chi_\alpha(A)}{(c_1(A) \cdot H)}=\lambda$,
it is easy to see that $A$ is $\alpha$-twisted semi-stable and $E/A$ is $\lambda$-semi-stable. 
By Proposition \ref{prop:finite}, we can find an $\alpha$-twisted semi-stable subsheaf $F_1$ of $E$ 
such that 
$\frac{\chi_\alpha(F_1)}{(c_1(F_1) \cdot H)}=\lambda$ 
and $\Hom(A,E/F_1)=0$ for any $\alpha$-twisted stable sheaf $A$ with
 $\frac{\chi_\alpha(A)}{(c_1(A) \cdot H)}=\lambda$.
For a 1-dimensional sheaf $A$ with $\frac{\chi_\alpha(A)}{(c_1(A) \cdot H)}=\lambda$ 
and a quotient $f:E/F_1 \to A$, we see that $A$ is $\alpha$-twisted semi-stable
and $\ker f$ is $\lambda$-semi-stable. 
By Proposition \ref{prop:finite}, we get a subsheaf $F_2$ of $E$
such that $F_1 \subset F_2$ and $F_2$ satisfies the required property.  
\end{proof}

\begin{defn}\label{defn:n-stable-stack}
Let ${\cal M}^{\lambda}({\bf e})^{ss}$ (resp. ${\cal M}^\lambda({\bf e})^{s}$) 
be the stack of $\lambda$-semi-stable
sheaves (resp. $\lambda$-stable sheaves) $E$ with $\tau(E)={\bf e}$. 
\end{defn}

${\cal M}^\lambda({\bf e})^{ss}$ is an open substack of the stack of 
coherent sheaves.

\begin{prop}\label{prop:lambda-smooth}
We set ${\bf e}=(r,\xi,a)$. Assume that $\gcd((\xi \cdot f),r)=1$. Then
${\cal M}^\lambda({\bf e})^{s}$ is smooth of 
$$
\dim {\cal M}^\lambda({\bf e})^{s}
=(\xi^2)-2ra+(r^2+1)\chi({\cal O}_X)-r(\xi \cdot K_X)+q-1.
$$
\end{prop}

\begin{proof}
Applying Lemma \ref{lem:trace-property}, we get that the trace map
$$
\Ext^2(E,E) \to H^2(X,{\cal O}_X)
$$
is an isomorphism.
Hence we get the claim.
\end{proof}
By Proposition \ref{prop:finite}, we get the following claim.

\begin{prop}\label{prop:torsion-free}
If $\lambda$ is sufficiently small, then
${\cal M}^\lambda({\bf e})^{ss}$ consists of torsion free sheaves, i.e.,
${\cal M}^\lambda({\bf e})^{ss}={\cal M}_{H_f}({\bf e})^{ss}$.
\end{prop}

\section{Wall crossing behaviors for the positive rank cases.}\label{sect:another}

\subsection{Structure of walls}\label{subsect:wall}
In this subsection, we continue to use ${\bf e}=(r,\xi,a)$ with $\gcd(r,(\xi \cdot f))=1$.

\begin{defn}\label{defn:wall}
\begin{enumerate}
\item[(1)]
$W:=\{\lambda \} (\subset {\Bbb R})$ is a wall for ${\bf e}$ if
there is a $\lambda$-stable sheaf $E$ with $\tau(E)={\bf e}$
such that there is a subsheaf $A$ of $E$ supported on a fiber and 
$\frac{\chi_\alpha(A)}{(c_1(A) \cdot H)}=\lambda$. 
\item[(2)]
A chamber is a connected component of the compliment of the set of walls.
\end{enumerate}
\end{defn}

\begin{rem}\label{rem:wall}
$A$ in Definition \ref{defn:wall} is $\alpha$-twisted semi-stable.
Let $A' \subset A$ be a stable factor. Then 
$A'$ also defines the wall $W$. In particular $(c_1(A')^2)=0,-2$.
Hence we may assume that $(c_1(A)^2)=0,-2$ for the defining subsheaf $A$.   
\end{rem}

\begin{NB}
Usually a wall is defined by the condition that
there is a $\lambda$-semi-stable sheaf $E$ with
$\Hom(A,E) \ne 0$ or $\Hom(E,A) \ne 0$.
\end{NB}

\begin{NB}
${\cal M}^{\lambda_+}({\bf e})^s \cap {\cal M}^{\lambda_-}({\bf e})^s 
\subset {\cal M}^\lambda({\bf e})^s$.
${\cal M}^\lambda({\bf e})={\cal M}^{\lambda_+}({\bf e})$.
\end{NB}

\begin{prop}
Let $B$ a compact subset of $\NS(X)_{\Bbb R}$ and
$\alpha \in B$ a ${\Bbb Q}$-divisor.
For a pair of integers $(r_0',d_0')$ such that $r_0'>0$ and
$\frac{d_0'}{r_0'}<\frac{(\xi \cdot f)}{r}$,
we set 
$$
\lambda_0:=\frac{d_0'}{r_0'(f \cdot H)}-\frac{(f \cdot \alpha)}{(f \cdot H)}.
$$
Then the following claims hold.
\begin{enumerate}
\item[(1)]
In the region ${\Bbb R}_{\lambda_0}:=
\{\lambda \in {\Bbb R} \mid \lambda<\lambda_0\}$,
the set of walls is finite.
\item[(2)]
If $\lambda$ is in a chamber, then ${\cal M}^\lambda({\bf e})^{ss}={\cal M}^\lambda({\bf e})^s$.
\end{enumerate}
\end{prop}

\begin{proof}
We use the notation in Proposition \ref{prop:finite}.
Since $\{\tau(T) \mid T \in {\cal T}_B \}$ is finite,
the choice of $c_1(A)$ of the first filter of the Harder-Narasimhan filtration of 
$T \in {\cal T}_B$ is  
finite. Since $\chi_\alpha(A)/(c_1(A) \cdot H)$ is bounded, 
the choice of $\tau(A)$ is finite.
Therefore (1) holds. (2) is obvious.
\end{proof}

\begin{rem}\label{rem:general}
$\lambda_0$ depends on the choice of $\alpha$.
Let ${\cal A}$ be the set of $\tau(A)$ of $\alpha$-twisted semi-stable sheaves $A$
with $(c_1(A)^2)=0,-2$ which define walls (cf Remark \ref{rem:wall}). 
Under a small purturbation of $\alpha$, all $\tau(A) \in {\cal A}$ define walls in 
${\Bbb R}_{\lambda_0}$. 
Then we may assume that $\alpha$ is general with respect to
all $\tau(A)$.
Thus $A_1$ and $A_2$ define the same wall if and only if
$\tau(A_1) \in {\Bbb Q}\tau(A_2)$.  
\end{rem}

\begin{thm}\label{thm:moduli}
Let ${\cal C}$ be a chamber.
For $\lambda_0:=\frac{d_0}{r_0(f \cdot H)}-\frac{(f \cdot \alpha)}{(f \cdot H)} \in {\cal C}$, 
we have a projective coarse moduli space
$M^{\lambda_0}(r,\xi,a)$ of $\lambda_0$-stable sheaves
$E$ with $\tau(E)=(r,\xi,a)$.
\end{thm}

\begin{proof}
For a chamber ${\cal C}$, we take 
$\lambda_0:=\frac{d_0}{r_0(f \cdot H)}-\frac{(f \cdot \alpha)}{(f \cdot H)} \in {\cal C}$, 
where $r_0, d_0 \in {\Bbb Z}$ and $\gcd(r_0,d_0)=1$.
Replacing $(r_0,d_0)$ by $(r_0',d_0')=(r_0 (f \cdot H)k,d_0(f \cdot H)k+1)$ $(k \gg 0)$,
we may assume that $\gcd(r_0(f \cdot H),d_0)=1$ and
$\frac{d_0}{r_0(f \cdot H)}-\frac{(f \cdot \alpha)}{(f \cdot H)} \in {\cal C}$.
By purturbing $\alpha$ again,
we may assume that $Y:=M_H^\alpha(0,r_0 f,d_0)$ consists of
$\alpha$-twisted stable sheaves, and hence,
$Y$ is a fine moduli space.
Let ${\cal P}$ be a universal family on $X \times Y$. 
By Proposition \ref{prop:n-stability},
$\Phi_{X \to Y}^{{\cal P}^{\vee}[1]}$ induces an isomorphism 
${\cal M}^{\lambda_0}(r,\xi,a)^{ss} \to {\cal M}_{H_f}(r',\xi',a')^{ss}$.
Hence we have a projective coarse moduli space $M^{\lambda_0}(r,\xi,a)$. 
\end{proof}

\begin{rem}\label{rem:moduli}
Since $\lambda_0$ is a point of an open set ${\cal C}$,
${\cal M}_{H_f}(r',\xi',a')^{ss}$ consists of $\mu$-stable vector bundles.
\begin{NB}
$\Ext^1({\Bbb C}_y,\Phi_{X \to Y}^{{\cal P}^{\vee}[1]}(E))=
\Hom({\cal P}_{|X \times \{ y\}},E)=0$. 
Hence $\Phi_{X \to Y}^{{\cal P}^{\vee}[1]}(E)$ is locally free.
\end{NB}
\end{rem}

\begin{NB}
For the torsion submodule $T$ of a $\lambda$-semi-stable sheaf 
$E$ with $\tau(E)={\bf e}$,
let 
$$
0 \subset T_1 \subset T_2 \subset \cdots \subset T_p=T
$$
be the Harder-Narasimhan filtration with respect to
$(H,\alpha)$-twisted semi-stability.
Then the choice of $\tau(T_i/T_{i-1})$ are finite by Proposition \ref{prop:finite}.
By perturbing $\alpha$, we may assume that 
$T_i/T_{i-1}$ are general with respect to
$(H,\alpha)$.
\end{NB}

We take a general $\alpha$ (see Remark \ref{rem:general}).
Let $W$ be a wall defined by an $\alpha$-twisted stable sheaf $A$
with $\tau(A)=\tau=(0,\tau_1,\tau_2)$.
We set $\lambda:=\frac{\chi_\alpha(A)}{(c_1(A) \cdot H)}$.
We take $\lambda_\pm$ ($\lambda_-<\lambda<\lambda_+$) 
from two adjacent chambers.
The Harder-Narasimhan filtration with respect to $\lambda_+$-semi-stability
is an exact sequence
\begin{equation}\label{eq:HNF}
0 \to E_1 \to E \to E_2 \to 0
\end{equation}
where 
$E_1 \in {\cal M}^{\lambda_+}({\bf e}-l \tau)^{ss}$ and
$E_2 \in {\cal M}^\alpha_H(l \tau)^{ss}$.
\begin{NB}
Old:
$E_2 \in {\cal M}^{\lambda_+}(l \tau)^{ss}$.
\end{NB}
The Harder-Narasimhan filtration with respect to 
$\lambda_-$-semi-stability is an exact sequence
\begin{equation}\label{eq:HNF2}
0 \to E_1 \to E \to E_2 \to 0
\end{equation}
where $l \in {\Bbb Q}_{>0}$, $E_1 \in {\cal M}^\alpha_H(l \tau)^{ss}$ and
$E_2 \in {\cal M}^{\lambda_-}({\bf e}-l \tau)^{ss}$.

Let ${\cal F}^+({\bf e}_1,{\bf e}_2)$ (resp. ${\cal F}^-({\bf e}_1,{\bf e}_2)$)
be the stack of filtrations
parameterizing \eqref{eq:HNF} (resp, \eqref{eq:HNF2}), where
$({\bf e}_1,{\bf e}_2)=({\bf e}-l\tau,l\tau)$ (resp. $=(l\tau,{\bf e}-l\tau)$).
Then we get
\begin{equation}
\dim {\cal F}^\pm ({\bf e}_1,{\bf e}_2)=
\dim {\cal M}^{\lambda_\pm}({\bf e}-l\tau)^{ss}
-\chi({\bf e}_1,{\bf e}_2)+\dim {\cal M}^{\lambda_\pm}(l \tau)^{ss}
\end{equation}
by \eqref{eq:HNF-dim}, where $l \in {\Bbb Q}_{>0}$,
$\chi({\bf e}_1,{\bf e}_2)=l(r \tau_2+l (\tau_1^2)-(\xi \cdot \tau_1))$.

\begin{prop}\label{prop:codim}
Let $W$ be a wall defined by $\tau$.
\begin{enumerate}
\item[(1)]
Assume that $\tau=(0,D,b)$ with $(D^2)=-2$. If $(D \cdot \xi)-rb \geq 0$, then
$$
\dim({\cal M}^{\lambda_\pm}({\bf e})^{ss} 
\setminus {\cal M}^\lambda({\bf e})^s)= 
\dim {\cal M}^{\lambda_\pm}({\bf e})^{ss}- ((D \cdot \xi)-rb+1).
$$
If  $(D \cdot \xi)-rb < 0$, then ${\cal M}^\lambda({\bf e})^s=\emptyset$.
\item[(2)]
Assume that $\tau=(0,r'f,d')$ with $\gcd(r',d')=1$ and $(r' f \cdot \xi)-rd'>0$.
Then
$$
\dim({\cal M}^{\lambda_\pm}({\bf e})^{ss} \setminus 
{\cal M}^\lambda({\bf e})^s)= 
\dim {\cal M}^{\lambda_\pm}({\bf e})^{ss}- 
\min_{(l_1,...,l_s,l)} \left(\sum_i l_i (r_i (f_i \cdot \xi)-rd_i)+l(r' (f \cdot \xi)-rd'-1) \right)
$$
where $r_i,d_i,l_i,l \in {\Bbb Z}$ are defined in \eqref{eq:isotropic}.
\end{enumerate}
\end{prop}

\begin{proof}
(1) The proof is similar to that of Proposition \ref{prop:wall-dim}. 
(2) For $\tau=(0,r' f, d')$, we consider ${\bf f}$ in \eqref{eq:isotropic}.
Then we have
\begin{equation}
\begin{split}
\codim {\cal F}^\pm({\bf e}_1,{\bf e}_2)=&
\left(\sum_i l_i \frac{d_i}{d'}+l \right)(r' (f \cdot \xi)-rd')-l \\
=& \sum_i l_i (r_i (f_i \cdot \xi)-rd_i)+l(r' (f \cdot \xi)-rd'-1).
\end{split}
\end{equation}
\end{proof}

\subsection{Birational correspondences}\label{subsect:barat-corresp}

As in subsection \ref{subsect:birat}, we have birational correspondences
induced by some Fourier-Mukai transforms.

\begin{prop}\label{prop:Psi}
\begin{enumerate}
\item[(1)]
Let $D$ be an effective divisor on a fiber with $(D^2)=-2$.
\begin{enumerate}
\item
Assume that $(\xi \cdot D)-rb<0$. Then
$R_A \circ R_A$ induces a birational map
\begin{equation}\label{eq:RR}
{\cal M}^{\lambda_+}({\bf e})^{ss} \to {\cal M}^{\lambda_-}({\bf e}')^{ss} 
\cdots \to
{\cal M}^{\lambda_+}({\bf e}')^{ss} \to {\cal M}^{\lambda_-}({\bf e})^{ss},
\end{equation}
where ${\bf e}'={\bf e}+((\xi \cdot D)-rb)\tau(A)$.
\begin{NB}
For an $\alpha$-twisted stable sheaf $A$ with
$\tau(A)=(0,D,b)$, 
we have an exact sequence
\begin{equation}
0 \to A^{\oplus k} \to E \to E' \to A^{\oplus k} \to 0.
\end{equation}
\end{NB}

\item
Assume that $(\xi \cdot D)-rb=0$. Then $R_A$ induces an isomorphism
\begin{equation}\label{eq:R}
{\cal M}^{\lambda_+}({\bf e})^{ss} \to {\cal M}^{\lambda_-}({\bf e})^{ss}.
\end{equation}
\end{enumerate}

\item[(2)]
\begin{enumerate}
\item
If $r' (\xi \cdot f)-r d'=1$, then as in \ref{subsect:birat},
we have a contravariant Fourier-Mukai transform
$\Psi:{\bf D}(X) \to {\bf D}(X)$ which induces an isomorphism
\begin{equation}\label{eq:Psi-isom1}
{\cal M}^{\lambda_+}({\bf e})^{ss} \to {\cal M}^{\lambda_-}({\bf e})^{ss}.
\end{equation}
\item
If $r' (\xi \cdot f)-rd'=2$ and $M_H^\alpha(0,r' f,d')$ is projective
(e.g. all multiple fibers have odd multiplicities),
then we also have a contravariant Fourier-Mukai transform
$\Psi:{\bf D}(X) \to {\bf D}(X)$ which induces an isomorphism
\begin{equation}\label{eq:Psi-isom2}
{\cal M}^{\lambda_+}({\bf e})^{ss} \to {\cal M}^{\lambda_-}({\bf e})^{ss}.
\end{equation} 
\end{enumerate}
\end{enumerate}
\end{prop}

\begin{proof}
We only give a remark on the proof  for (2) (b).
 In this case, $X':=M_H^\alpha(0,r' f,d')$ is not fine in general.
So we need to use a universal family ${\cal P}$ as a twisted sheaf.
Thus we have a twisted Fourier-Mukai transform
$\Phi_{X \to X'}^{{\cal P}^{\vee}[1]}:{\bf D}(X) \to {\bf D}^\beta(X')$,
where $\beta$ is a 2-cocycle of ${\cal O}_{X'}^{\times}$.
For $E \in {\cal M}^{\lambda_+}({\bf e})^{ss}$,
$F:=\Phi_{X \to X'}^{{\cal P}^{\vee}[1]}(E)$ is a stable
$\beta$-twisted sheaf of rank 2.
We take a twisted vector bundle $F_0 \in {\bf D}^\beta(X')$ of rank 2
such that $(\det F_0)^{\vee} \otimes \det F \in \Pic^0(X')$.
Then we shall replace the isomorphism
in Lemma \ref{lem:isotropic-wall} (2) by the correspondence
${\bf D}^\beta(X') \to {\bf D}^{\beta^{-1}}(X')$ 
($F \mapsto F \otimes (\det F_0)^{\vee}$).

For the projectivity of $M_H^\alpha(0,r' f,d')$,
we shall study properly semi-stable sheaves $P$ with
$\tau(P)=(0,r' f,d')$. Obviously $P$ is a sheaf on a multiple fiber $f_i$. 
In the notation of \eqref{eq:isotropic},
$(r'f,d')=\frac{d'}{d_i}(r_i f_i,d_i)$ and $m_i=p_i \frac{d'}{d_i}$.
Since $2=r' (\xi \cdot f)-rd'=\frac{d'}{d_i}(r_i(\xi \cdot f_i)-rd_i)$,
if $m_i$ is odd, then we get $\frac{d'}{d_i}=1$.
Hence there is no properly semi-stable sheaf $P$.
\end{proof}

\begin{rem}
Assume that $r'(\xi \cdot f)-rd'=2$.
If $d'$ is odd, then $(r'f,d')=\frac{d'}{d_i}(r_i f_i,d_i)$ implies 
$\frac{d'}{d_i}=1$.
If $d'$ is even, then $r'$ is odd and $(\xi \cdot f)$ is even.
Thus if $d'$ is odd or $(\xi \cdot f)$ is odd, then
$M_H^\alpha(0,r' f,d')$ is projective.
\begin{NB}
$m_i \mid (\xi \cdot f)$.
\end{NB}
\end{rem}

Therefore we get the following result.

\begin{prop}\label{prop:birat}
For general $\lambda_1, \lambda_2$, 
there is a (contravariant) Fourier-Mukai transform
$\Psi:{\bf D}(X) \to {\bf D}(X)$ and proper closed substacks 
${\cal Z}_i \subset {\cal M}^{\lambda_i}({\bf e})^{ss}$ $(i=1,2)$
such that
$\Psi$ induces an isomorphism 
$$
\Psi:{\cal M}^{\lambda_1}({\bf e})^{ss} \setminus {\cal Z}_1 
\to {\cal M}^{\lambda_2}({\bf e})^{ss} \setminus {\cal Z}_2.
$$
Moreover $\codim_{{\cal M}^{\lambda_i}({\bf e})^{ss}} {\cal Z}_i \geq 2$ if
there is no multiple fibers.
\end{prop}

\begin{proof}
For the wall in Proposition \ref{prop:codim} (1), 
\eqref{eq:RR} or \eqref{eq:R} is a desired morphism. 

We shall treat the wall in Proposition \ref{prop:codim} (2).
In the notation of the proposition, 
we set
$$
d({\bf f}):=\sum_i l_i (r_i (f_i \cdot \xi)-rd_i)+l(r' (f \cdot \xi)-rd'-1).
$$
Then $d({\bf f})=0$ if and only if
$r'(f \cdot \xi)-rd'=1$ and $l_i=0$ for all $i$.
In this case \eqref{eq:Psi-isom1} is a desired isomorphism.
We also see that $d({\bf f})=1$ if and only if
(I) $r'(f \cdot \xi)-rd'=2$ and $l_i=0$ for all $i$ or (II)
there is $i$ such that $l_i=1$, $r_i(f_i \cdot \xi)-rd_i=1$ and
$l_j=0$ for $j \ne i$.
For case (I), \eqref{eq:Psi-isom2} is a desired isomorphism.
Hence ${\cal Z}_1$ and ${\cal Z}_2$ are proper substack and
$\codim {\cal Z}_1, \codim {\cal Z}_2 \geq 2$ if (II) does not occur.
\end{proof}

\begin{lem}\label{lem:m_i}
Assume that $r'_0 d-rd'_0=1$ with $0 \leq r_0'<r$.
Let $\tau=(0,r' f,d')$ be a topological invariant in Proposition \ref{prop:codim}.
\begin{enumerate}
\item[(1)]
If $r>m_i r'_0$, then for $r_i,d_i$ in \eqref{eq:isotropic},
$r_i (f_i \cdot \xi)-r d_i=1$ implies 
$$
\frac{d'_0}{r'_0(f \cdot H)} \leq \frac{d_i}{r_i(f_i \cdot H)}.
$$
\item[(2)]
If $r>2r_0'$, then for $(r',d')$ satisfying $r'(\xi \cdot f)-rd'=2$, we have  
$$
\frac{d'_0}{r'_0(f \cdot H)} \leq \frac{d'}{r'(f \cdot H)}.
$$
\end{enumerate}
\end{lem}

\begin{proof}
(1)
If $r_i (f_i \cdot \xi)-r d_i=1$, then
$r'_0 m_i (f_i \cdot \xi)-rd'_0=1$ implies
$(r_i,d_i)=(r'_0 m_i+kr,d'_0+k(f_i \cdot \xi))$ $(k \in {\Bbb Z})$.
Then 
$\frac{d'_0}{r'_0(f \cdot H)}
\leq \frac{d'_0+k(f_i \cdot \xi)}{(r'_0 m_i+kr)(f_i \cdot H)}$ for $k \geq 0$.
Since $r>m_i r'_0$, $r'_0 f+kr f_i$ is not effective for $k<0$.
Hence $(0,r'_0 f+kr f_i,d'_0+k(f_i \cdot \xi))$ $(k<0)$ does not define a wall.

(2)
By our assumption, we see that $(r',d')=(2r_0'+kr,2d_0'+k(\xi \cdot f))$
$(k \in {\Bbb Z})$.
By a similar argument to (1), we get the claim. 
\end{proof}

\begin{thm}\label{thm:birat2}
We set ${\bf e}:=(r,\xi,a)$, where $\gcd(r,(\xi \cdot f))=1$.
For a fine moduli space $Y:=M_H^\alpha(0,r_1 f,d_1)$ and a universal family
${\cal P}$ on $X \times Y$, 
let $\Phi_{X \to Y}^{{\cal P}^{\vee}}:{\bf D}(X) \to {\bf D}(Y)$ be the associated
 Fourier-Mukai transform.
Assume that $r_1 (\xi \cdot f)-rd_1>0$.
Then there is a (contravariant) autoequivalence
$\Psi:{\bf D}(X) \to {\bf D}(X)$ such that
\begin{enumerate}
\item
$\tau(\Psi(E))=\tau(E)$ and
\item
$\Phi_{X \to Y}^{{\cal P}^{\vee}[1]} \circ \Psi$ induces a birational map 
${\cal M}_{H_f}(r,\xi,a)^{ss} \cdots \to {\cal M}_{H'_f}(r',\xi',a')^{ss}$,
\end{enumerate}
where $E \in  {\cal M}_{H_f}(r,\xi,a)$, 
$\tau(\Phi_{X \to Y}^{{\cal P}^{\vee}}(E)[1])=(r',\xi',a')$ and 
$H'$ is a polarization of $Y$. 
\end{thm}

\begin{proof}
We set $\lambda_1:=\frac{d_1-(r_1 f \cdot \alpha)}{(r_1 f \cdot H)}$.
By Proposition \ref{prop:n-stability},
 $\Phi_{X \to Y}^{{\cal P}^{\vee}[1]}$ induces an isomorphism
$$
{\cal M}^{\lambda_1}(r,\xi,a)^{ss} \to  {\cal M}_{H'_f}(r',\xi',a')^{ss}.
$$
For a sufficiently small $\lambda$, Proposition \ref{prop:torsion-free}
implies
${\cal M}^{\lambda}(r,\xi,a)^{ss}={\cal M}_{H_f}(r,\xi,a)^{ss}$.
Hence
by using Proposition \ref{prop:birat},
we get a desired birational map.
\end{proof}

Thanks to Proposition \ref{prop:codim},
we can estimate the codimension of the locus where the birational map
is not defined. 
In this sense, Theorem \ref{thm:birat2} is regarded as a refinement of 
\cite[Thm. 1.1]{Br:1}.

\begin{thm}\label{thm:m=2}
Let $\pi:X \to C$ be an arbitrary elliptic surface with multiple fibers
$m_1 f_1,...,m_s f_s$.
For $(r,\xi,a)$ with $\gcd((\xi \cdot f),r)=1$,
we take a pair of integers $(r',d')$ such that
$(\xi \cdot f)r'-rd'=1$ and $0 \leq r'<r$.
If $r > r' m_i$ for all $i$, then
there is a (contravariant) equivalence
$\Psi:{\bf D}(X) \to {\bf D}(X)$ such that   
$\Phi_{X \to Y}^{{\cal P}^{\vee}[1]} \circ \Psi$ induces a birational map
\begin{equation}\label{eq:birat:m=2}
M_{H_f}(r,\xi,a) \cdots \to \Hilb_Y^l \times \Pic^0(Y)
\end{equation}
which is defined up to codimension 2.
In particular if $m_i =2$ for all $i$, then the claim holds.
\end{thm}

\begin{proof}
Lemma \ref{lem:m_i}
and the proof of Theorem \ref{thm:birat2} imply the claim.

Assume that $m_i =2$ for all $i$.
Since $2 \mid (\xi \cdot f)$, we may assume that $r \geq 3$.
Then we have a birational map 
\begin{equation}
\begin{matrix}
M_{H_f}(r,\xi,a) & \cdots \to & M_{H_f}(r,-\xi,a')\\
E & \mapsto & E^{\vee},
\end{matrix}
\end{equation}
which is defined in codimension 1, where
$a'=a+(\xi \cdot K_X)$.
Since $(r-r')(-\xi \cdot f)-r((f \cdot \xi)-d')=1$,
replacing $(r,\xi,a)$ by $(r,-\xi,a')$ if necessary,
we may assume that $r>2r'$.
Hence the claim holds.
\end{proof}

\begin{NB}
\begin{cor}\label{cor:m=2}
Let $\pi:X \to C$ be an arbitrary elliptic surface with multiple fibers
$m_1 f_1,...,m_s f_s$
Assume that $(r,\xi,a)$ satisfies $\gcd((\xi \cdot f),r)=1$,
If $m_i \leq 2$ for all $i$, then the claim of Theorem \ref{thm:Pic} holds.
\end{cor}

\begin{proof}
We may assume that $r>1$.
Assume that $m_i=2$ for all $i$. Then $2 \mid (\xi \cdot f)$ and $r \geq 3$
Then we have a birational map 
\begin{equation}
\begin{matrix}
M_{H_f}(r,\xi,a) & \cdots \to & M_{H_f}(r,-\xi,a')\\
E & \mapsto & E^{\vee},
\end{matrix}
\end{equation}
which is defined in codimension 1, where
$a'=a+(\xi \cdot K_X)$.
Since $(r-r')(-\xi \cdot f)-r((f \cdot \xi)-d')=1$,
replacing $(r,\xi,a)$ by $(r,-\xi,a')$ if necessary,
we may assume that $r>2r'$.
Then the claim follows from Proposition \ref{prop:m=2}.
\end{proof}
\end{NB}

\begin{NB}
For example, if $d \equiv -1 \mod r$ and $r \geq \max_i \{m_i \}$, then
$\codim_{{\cal M}^{\lambda_i}({\bf e})^ss} {\cal Z}_i \geq 2$, 
where $\lambda_1 \ll 0$ and 
$\lambda_2:=
\frac{-r-(a f \cdot \alpha)}{(a f \cdot H)}$.
\end{NB}

\subsection{A relation with \cite{Br:1}.}

Let us explain a more precise relation with \cite{Br:1}.
Let $Y:=M_H^\alpha(0,r' f,d')$ be a fine moduli space and
${\cal P}$ a universal family in subsection \ref{subsect:relFM}.
Then
$\Phi$ and $\Psi$ in \cite{Br:1} correspond to $\Phi_{Y \to X}^{{\cal P}}$
and $\Phi_{X \to Y}^{{\cal Q}}$ respectively.
%
\begin{NB}
We note that the birational type of $M_{H'}^{\alpha'}(0,r' f,-p)$ 
is independent of $\alpha'$.
Indeed there is an equivalence $\Lambda:{\bf D}(Y) \to {\bf D}(Y)$
which induces an isomorphism
$M_{H'}^{\alpha'}(0,r' f,-p) \to M_{H'}(0,r' f,-p)$.
We set 
${\cal Q}':=\Lambda({\cal Q})$ and
${\cal P}':=({\cal Q}')^{\vee}[1]$.
Then there is $\alpha''$ such that 
${\cal P}'_{|X \times \{ y \}}$ is $\alpha''$-twisted stable
for all $y \in Y$. 
\end{NB}

For a coherent sheaf $E$, let
$d(E):=(c_1(E) \cdot f)$ be the relative degree of $E$.
By \cite[Thm. 5.3]{Br:1}, we get the following relations:
\begin{equation}
\begin{split}
\begin{pmatrix}
\rk(\Phi_{X \to Y}^{{\cal P}^{\vee}}(E))\\
d(\Phi_{X \to Y}^{{\cal P}^{\vee}}(E))
\end{pmatrix}
=&
\begin{pmatrix}
d' & -r' \\
-q & p
\end{pmatrix}
\begin{pmatrix}
\rk E\\
d(E)
\end{pmatrix},\;
d' p-r' q=1,\\
\begin{pmatrix}
\rk(\Phi_{Y \to X}^{{\cal P}}(F))\\
d(\Phi_{Y \to X}^{{\cal P}}(F))
\end{pmatrix}
=&
\begin{pmatrix}
p & r' \\
q & d'
\end{pmatrix}
\begin{pmatrix}
\rk F \\
d(F)
\end{pmatrix}.
\end{split}
\end{equation}
In section \ref{subsect:n-stability},
we introduced $\lambda$-semi-stability for $F \in \Coh(Y)$ and studied the stability of
$\Phi_{Y \to X}^{{\cal P}}(F)=\Phi_{Y \to X}^{{\cal Q}^{\vee}[1]}(F)$ under the condition
$p \rk F+r' d(F)>0$. In particular, we studied the $\Phi$-$\WIT_0$ property in \cite{Br:1}.

\begin{NB}
If $\tau:=(0,r_i f_i, d_i)$ defines a codimension 1 wall for 
${\bf e}=(1,0,\chi)$, then
$d_i=-1$ and $\frac{-1}{r_i} \leq \frac{-r}{r' m_i}$.

If $\tau=(0,r'' f,d'')$ defines a codimension 1 wall for
${\bf e}$, then $d''=-2$ and $\frac{-2}{r''} \leq \frac{-r}{r'}$.
Since $r \geq a$, $2/r' \geq 1$. Thus $r'=1,2$.
Furthermore we assume that $2r' \leq r$. Then we have $r'=2$
and $2r'=r$. Thus $r=2$ by $\gcd(r,r')=1$. 
\end{NB}

\begin{ex}\label{ex:multiple}
Assume that all fibers are irreducible. 
Then we may assume that $\alpha=0$.
We set ${\bf e}=(r,\xi,a)$ and set $d:=(\xi \cdot f)$.
Let $m_1 f_1',...,m_s f_s'$ be the multiple fibers of $\pi'$.
For $(1,0,e-l)$, $(0,r_i f_i',d_i)$ defines a wall if and only if
$-l \leq d_i$. 
We set $p:=r$ and $q:=d$.
We note that $\tau({\cal Q}_{|\{ x \} \times Y})=(0,r' f,-r)$.
Then we can easily classify walls: 
\begin{enumerate}
\item 
$(0,r_i f_i, d_i)$ $(\gcd(r_i,d_i)=1)$ defines a wall if and only if
$-l \leq d_i$ and $\frac{d_i}{r_i(f_i \cdot \xi)} \leq \frac{-r}{r'(f \cdot \xi)}$.
\item
$(0,r'' f, d'')$ ($\gcd(r'',d'')=1$) defines a wall if and only if
$-l \leq d''$ and $\frac{d''}{r''(f \cdot \xi)} \leq \frac{-r}{r'(f \cdot \xi)}$.
\end{enumerate}
In particular $(r_i,d_i)$ satisfies $l \geq -d_i>0$ and
$0<r_i \leq \frac{lr' m_i}{r}$.

We note that $r_i f_i=\frac{r_i}{m_i}f$. 
If $-l>-\frac{r}{m_i r'}$ for all $i$,
then all $F \in \Hilb_Y^l \times \Pic^0(Y)$ are $-\frac{r}{r' (f \cdot H)}$-stable.
Hence $\Phi_{Y \to X}^{{\cal P}}=\Phi_{Y \to X}^{{\cal Q}^{\vee}[1]}$ induces an isomorphism
\begin{equation}\label{eq:isom}
\Hilb_Y^l \times \Pic^0(Y) \cong M_{H_f}(r,\xi,a).
\end{equation}
In particular if $r>l r' m_i$ for $1 \leq i \leq s$, then \eqref{eq:isom} holds, where
$(r',d')$ satisfies $rd'-r' d=1$ and
 $0<r'<r$ (cf. \cite[Lem. 7.4]{Br:1}). 
In particular if $d \equiv 1 \mod r$, then we may assume that $r'=1$
and the condition is $r>l m_i$ for all $i$.  
\end{ex}

\begin{NB}
Assume that all fibers are integral curves.
Then we may assume that $\alpha=0$.
For $\Hilb_X^l \times \Pic^0(X)$,
the candidates of of walls in $\lambda \leq -2$
are defined by $(0,rf,-k)$ such that
$-l \leq -\frac{k}{r} \leq -2$ and $l \geq k$. Thus 
$lr \geq k \geq 2r$ and $l \geq k$.
In particular $2r \leq l$.
For $k=2$, we have $r=1$.
In this case, $M^{\lambda_-}(1,0,e-l) \cong M^{\lambda_+}(1,0,e-l)$,
where $\lambda_\pm$ are adjacent cahmber separated by the wall defined by $(0,f,-2)$.  
\end{NB}

\begin{ex}\label{ex:reducible}
Assume that there is a reducible fiber. 
We set $Y:=M_H^\alpha(0,f,-1)$ (i.e., $(r',d')=(1,-1)$). 
Since $Y$ is fine, we have a Fourier-Mukai
transform 
\begin{equation}\label{eq:reducible-FM}
\Phi_{X \to Y}^{{\cal P}^{\vee}}:{\bf D}(X) \to {\bf D}(Y)
\end{equation}
associated to a universal family ${\cal P}$.
By the existence of reducible fibers, the equivalence \eqref{eq:reducible-FM} depends on 
the choice of $\alpha$. 
For example, we assume that $\alpha=0$.
Let $D_1$ be a smooth rational curve in a fiber $\pi^{-1}(c)$. 
For ${\bf e}=\tau({\cal O}_X(D_1))$,
$\tau({\cal O}_{D_1}(D_1))=(0,D_1,-1)$ defines a wall.
Since $\lambda_1:=\frac{-1}{(D_1 \cdot H)}<\frac{-1}{(f \cdot H)}$,
and ${\cal O}_X(D_1)$ is not $\lambda_1$-stable,
$\Phi_{X \to Y}^{{\cal P}^{\vee}}({\cal O}_X(D_1))[1] \not \in \Coh(Y)$.
More generally for any locally free sheaf $E \in {\cal M}_{H_f}(r,\xi,a)$ with
$r'(\xi \cdot f)-rd'>0$, 
since $\tau(E(nD_1)_{|D_1})=(0,rD_1,(\xi \cdot D_1)+(1-2n)r)$,
$\Phi_{X \to Y}^{{\cal P}^{\vee}}(E(nD_1))[1] \not \in \Coh(Y)$ for $n \gg 0$,
where $Y:=M_H(0,r' f,d')$. On the other hand if $\alpha=-D_1$, then
we can easily show that 
$\Phi_{X \to Y}^{{\cal P}^{\vee}}({\cal O}_X(D_1))[1] \in \Coh(Y)$ 
(see Remark \ref{rem:relation}).
Thus the choice of $\alpha$ is important.
\end{ex}

\begin{NB}
Codimension 1 wall is defined by
\begin{equation}
0 \to E_1 \to E \to {\Bbb C}_x[-1] \to 0.
\end{equation}
\end{NB}

\begin{rem}\label{rem:relation}
Let $D (\subset Y)$ be a genus 1 curve in a fiber. 
Since $\chi({\cal O}_{D_1})=-(D_1^2)/2>0$ for any non-trivial quotient
${\cal O}_D \to {\cal O}_{D_1}$ of dimension 1, 
${\cal O}_D$ is a stable sheaf
of $\chi({\cal O}_D)=0$. If $X=M_{H'}(0,r' f,-r)$ and $r>r'$,
then we see that 
a general $I_Z \in \Hilb_Y^l$ is $\frac{-r}{r' (f \cdot H')}$-stable and
$\Phi_{Y \to X}^{{\cal P}}(I_Z)$ is a stable sheaf.
This is \cite[Lem. 7.3]{Br:1}.
\end{rem}

\begin{NB}
\begin{rem}\label{rem:relation2}
We shall explain a relation with Bridgeland paper \cite{Br:1}.
We assume that $\alpha' \in \NS(Y)_{\Bbb Q}$ is trivial.
Let $D (\subset X)$ be a genus 1 curve in a fiber. 
Since $\chi({\cal O}_{D_1})=-(D_1^2)/2>0$ for any non-trivial quotient
${\cal O}_D \to {\cal O}_{D_1}$ of dimension 1, 
${\cal O}_D$ is a stable sheaf
of $\chi({\cal O}_D)=0$. If $Y=M_H(0,r_1 f,d_1)$ and $-1>\frac{d_1}{r_1}$,
then we see that 
a general $I_Z \in \Hilb_X^l$ is $\frac{d_1}{r_1 (f \cdot H)}$-stable and
$\Phi_{X \to Y}^{{\cal P}^{\vee}}(I_Z)$ is a stable sheaf.
This is \cite[Lem. 7.3]{Br:1}.
\end{rem}
\end{NB}

\section{Application}\label{sect:application}

\subsection{Picard groups.}
For ${\bf e}=(r,\xi,a) \in {\Bbb Z} \oplus \NS(X) \oplus {\Bbb Z}$ with
$\gcd(r,(\xi \cdot f))=1$,
we set 
$$
K(X)_{\bf e}:=\{\alpha \in K(X) \mid \chi(\alpha,{\bf e})=0 \},
$$
where $\chi(\alpha,{\bf e}):=\chi(\alpha,E)$ $(\tau(E)={\bf e})$.
Let ${\cal E}$ be a universal family on $M_{H_f}({\bf e}) \times X$.
We have a homomorphism
\begin{equation}
\begin{matrix}
\theta_{\bf e}:& K(X)_{\bf e} & \to & \Pic(M_{H_f}({\bf e}))\\ 
& \alpha & \mapsto & \det p_{!}({\cal E} \otimes p_X^*(\alpha^{\vee}))
\end{matrix}
\end{equation}
%
which is independent of the choice of ${\cal E}$.
We define a homomorphism $D^*$ by 
\begin{equation}
\begin{matrix}
D^*:& K(X) & \to & K(X)\\
& \alpha & \mapsto & \alpha^{\vee}(-K_X).
\end{matrix}
\end{equation} 
Then the Serre duality implies
$\chi(\alpha,{\bf e})=\chi({\bf e},\alpha(K_X))=\chi(D_X(\alpha(K_X)),D_X({\bf e}))$, and hence 
$D^*$ induces an isomorphism
$K(X)_{{\bf e}} \to K(X)_{D_X({\bf e})}$, where
$D_X({\bf e}):=(r,-\xi,a)$. 

For the Fouruer-Mukai transform $\Phi$ in Theorem \ref{thm:birat},
we have a commutative diagram
\begin{equation}\label{eq:comm}
\begin{CD}
K(X)_{\bf e} @>{\Phi_*}>> K(Y)_{{\bf e}'}\\
@V{\theta_{\bf e}}VV @VV{\theta_{{\bf e}'}}V\\
\Pic(M_{H_f}({\bf e})) @= \Pic(M_{H'_f}({\bf e}'))
\end{CD}
\end{equation}
where $\Phi_*$ is an isomorphism and ${\bf e}'=(1,0,a')$. 
Indeed for $D_X$, Grothendieck-Serre duality implies 
\begin{equation}\label{eq:D}
\theta_{{\bf e}}(\alpha)=\det p_! ({\cal E}^{\vee} \otimes p_X^*(\alpha(K_X)))^{\vee}
=\theta_{D_X({\bf e})}(D^*(\alpha))^{\vee}.
\end{equation}
We also have 
\begin{equation}
\theta_{{\bf e}}(\alpha)=\theta_{{\bf e}'}(\Phi_{X \to Y}^{\bf P}(\alpha))
\end{equation}
for ${\bf e}'=\Phi_{X \to Y}^{\bf P}({\bf e})$.
Since $\Phi$ is a composite of these functors,
we have \eqref{eq:comm}.

\begin{NB}
Let $q:S \times X \times Y \to S$ be the projection, 
$q_X$ and $q_Y$ are projections from $S \times X \to Y$ to $X$ and $Y$ respectively. 
For ${\cal E}' \in {\bf D}(S \times Y)$,
\begin{equation}
\begin{split}
\det p_{!}(\Phi_{Y \to X}^{{\bf P}^{\vee}}({\cal E}') \otimes p_X^*(\alpha^{\vee}))=&
\det p_{!}({\bf R}q_{S \times X*}({\bf P}^{\vee} \otimes {\cal E}') \otimes p_X^*(\alpha^{\vee}))\\
=& \det q_{!}({\bf P}^{\vee} \otimes {\cal E}' \otimes q_X^*(\alpha^{\vee}))\\
=& \det q_{!}({\cal E}' \otimes q_{X \times Y}^*(({\bf P} \otimes p_X^*(\alpha))^{\vee}))\\
=& \det p'_{!}({\cal E}' \otimes {\bf R}p_{Y*}({\bf P} \otimes p_X^*(\alpha(K_X)))^{\vee})\\
=& \det p'_{!}({\cal E}' \otimes \Phi_{X \to Y}^{{\bf P}}(\alpha(K_X))^{\vee})
\end{split}
\end{equation}
\end{NB}

Now we come to our second main result.
\begin{thm}\label{thm:Pic}
Assume that the multiplicity of all multiple fibers are two. 
We set ${\bf e}:=(r,\xi,a)$ with $\gcd(r,(\xi \cdot f))=1$.
Assume that $\dim M_{H_f}({\bf e})\geq 4+g$ and $k={\Bbb C}$.
Then we have an exact sequence
$$
0 \longrightarrow \ker \tau \longrightarrow
 K(X)_{\bf e} \overset{\theta_{\bf e}}{\longrightarrow}
 \Pic(M_{H_f}({\bf e}))/\Pic(\Alb(M_{H_f}({\bf e}))) \longrightarrow 0.
$$
\end{thm}

\begin{proof}
If $r=1$, the the claim is \cite[Cor. A.4]{Y:Enriques}.
By using \eqref{eq:comm} and Theorem \ref{thm:m=2}, we get the claim.
\end{proof}

\begin{rem}\label{rem:m=2}
In the notation of Theorem \ref{thm:m=2},
if $r \geq r' m_i$ for all $i$ or 
$r \geq (r-r')m_i$ for all $i$, then we also see that the claim of Theorem \ref{thm:Pic} holds.
In particular if $d \equiv \pm 1 \mod r$ and $r \geq m_i$ for all $i$, then
the claim of Theorem \ref{thm:Pic} holds.
\end{rem}

\begin{rem}
We can also see that
the birational map $M_{H_f}({\bf e}) \cdots \to \Hilb_Y^l \times \Pic^0(Y)$
in Theorem \ref{thm:birat2} extends to an open subscheme
whose complement is at least of codimension 2. 
Hence we have
$\Pic(M_{H_f}({\bf e})) \cong \Pic(\Hilb_Y^l \times \Pic^0(Y))$
as abstract groups although we do not know whether 
the statement of Theorem \ref{thm:Pic} holds.
\end{rem}

We shall study line bundles on $M_{H_f}({\bf e})$.
We define a non-negative integer $l$ by $\dim M_{H_f}({\bf e})=2l+g$ and assume that $l \geq 2$.

\begin{lem}\label{lem:fiber-indep}
For $c \in C$,
$\det {\cal E}_{|M_{H_f}({\bf e}) \times \{ x \}}$ is independent of $x \in \pi^{-1}(c)$.
\end{lem}

\begin{proof}
Let ${\bf L}$ be the universal line bundle on $\Pic^0(C) \times C$.
We set ${\bf L}':=(1_{\Pic^0(C)} \times \pi)^*({\bf L})$.
Since $\Pic^0(X)=\pi^*(\Pic^0(C))$,
there is a morphism $\varphi:M_{H_f}({\bf e}) \to \Pic^0(C)$
such that  
$\det {\cal E} \cong (\varphi \times 1_X)^*({\bf L}') \otimes 
(Q \boxtimes L)$,
where $Q \in \Pic(M_{H_f}({\bf e}))$ and $L \in \Pic(X)$.
Then 
\begin{equation}
\det {\cal E}_{|M_{H_f}({\bf e}) \times \{ x \}} \cong
 \varphi^*({\bf L}_{|\Pic^0(C) \times \{ c \}})
\otimes Q,
\end{equation}
which implies the claim.
\end{proof}

\begin{NB}
For $E \in M_{H_f}(r,\xi,a)$, 
there is a smooth fiber $f$ such that $E_{|f}$ is a stable vector bundle.
\end{NB}

\begin{defn}
For an effective divisor $\delta \in \Pic(C)$,
we set ${\cal L}_{\bf e}(\delta):=\theta_{\bf e}(F)$, 
where $\tau(F)=(0,r \pi^{-1}(\delta),(\pi^{-1}(\delta) \cdot \xi))$.
\end{defn}

By Lemma \ref{lem:fiber-indep},
${\cal L}_{\bf e}(\delta)$ is independent of the choice of $F$.

\begin{prop}[{cf. O'Grady \cite{O}}]
\begin{enumerate}
\item[(1)]
Let $\delta$ be an effective divisor on $C$. Then 
${\cal L}_{\bf e}(m \delta)$ $(m \gg 0)$
is base point free and defines a morphism
$M_{H_f}({\bf e}) \to S^l C$.
\item[(2)]
If the Kodaira dimension of $X$ is 1, then
$K_{M_{H_f}({\bf e})}^{\otimes m}$
defines a morphism $M_{H_f}({\bf e}) \to S^l C$.
\end{enumerate}
\end{prop}

\begin{NB}
If $\Hom(E,F) \ne 0$, then 
$E_{|f} \cong F$.
Hence we can take $F$ such that $\Hom(E,F)=0$.
\end{NB}

\begin{proof}
(1)
We may assume that $\delta$ is very ample.
For $E \in M_{H_f}({\bf e})$, 
we take a divisor $\sum_{i=1}^N p_i \in |\delta|$ such that $\pi^{-1}(p_i)$
are smooth and $E_{|\pi^{-1}(p_i)}$ are stable vector bundles.
We take stable vector bundles $F_i$ of rank $r$ on $\pi^{-1}(p_i)$
such that $\chi(F_i)=(\xi \cdot f)$
and $\Hom(E,F_i(K_X))=0$ for all $i$.
We set
\begin{equation}
D:=\{ y \in M_{H_f}({\bf e}) \mid \Ext^2(\oplus_i F_i, {\cal E}_{|\{y \} \times X}) \ne 0 \}.
\end{equation}
Then 
${\cal O}_{M_{H_f}({\bf e})}(D) \cong \theta_{\bf e}(\sum_i F_i)$.
Since $E \not \in D$, ${\cal L}_{\bf e}(\delta)$ is base point free. 
For ${\bf e}_1=(1,0,e-l)$,
we have a morphism
$\Hilb_X^l \times \Pic^0(X) \to S^l X \to S^l C$ and
${\cal L}_{{\bf e}_1}(\delta) \cong {\cal O}_{M_{H_f}({\bf e}_1)}(D_1)$,
where
$$
D_1=\{ (I_Z(L) \mid I_Z \in \Hilb_X^l, \pi(Z) \cap \{p_1,p_2,...,p_N \} \ne \emptyset, L \in \Pic^0(X) \}.
$$
Hence ${\cal L}_{{\bf e}_1}(\delta)$ is the pull-back of
an ample line bundle on $S^l C$.
By the birational map
$$
M_{H_f}({\bf e}) \cdots \to \Hilb_X^l \times \Pic^0(X),
$$
${\cal L}_{\bf e}(m\delta)$ corresponds to the
line bundle ${\cal L}_{{\bf e}_1}(m \delta)$.
Therefore we get (1).

(2)
We note that $mrK_X=\pi^{-1}(\delta)$ $(\delta \in \Pic(C))$ if $m_i \mid m$
for all $i$.
Then $K_{M_{H_f}({\bf e})}^{\otimes m}={\cal L}_{\bf e}(\delta)$ up to torsion.
Hence the second claim holds.
\end{proof}

\begin{NB}
 Then $\theta_{\bf e}(F)$ is a line bundle such that $E$ is not the base point.

For a line bundle $L \in \Pic^0(X)=\pi^*(\Pic^0(C))$,
$F \otimes L \cong F$. 
Hence $\{L \in \Pic^0(X) \mid \Hom(E \otimes L,F) \ne 0 \}$
is 
\end{NB}

The following proposition shows the relation between our chamber structure
and relative ample line bundles over $S^l C$.

\begin{prop}\label{prop:rel-ample}
Assume that $\alpha \in f^\perp$.
Let $(\lambda_1,\lambda_2)$ be a chamber for $\lambda$-stability.
Assume that $\eta=\eta_0+\lambda H \in \NS(X)_{\Bbb Q}$ satisfies
\begin{equation}\label{eq:rel-ample}
\lambda_1
<\frac{(\eta \cdot f)}{(H \cdot f)}
<\lambda_2
\end{equation}
where $\eta_0 \in f^\perp$ is sufficiently close to $\alpha$.
Then $\theta_{\bf e}(F)$ is relatively ample over $S^l C$, where
$F \in K(X)_{\bf e}$ with
$\tau(F)=-n(1,\eta, b)$ $(n \gg 0)$.
\end{prop}

\begin{proof}
We shall prove the claim by modifying the proof of Theorem \ref{thm:moduli}.
We set $\lambda=\frac{d_0}{r_0 (f \cdot H)}$ ($r_0 \in {\Bbb Z}_{>0}$, $d_0 \in {\Bbb Z}$).
We take an integer $k \in (f \cdot H){\Bbb Z}$ such that
$$
\lambda_1<\frac{d_0 k-1}{r_0 k (f \cdot H)}
<\frac{d_0 }{r_0 (f \cdot H)}<\frac{d_0 k+1}{r_0 k (f \cdot H)}<
\lambda_2.
$$
If the claim holds for $\lambda=\frac{d_0 k-1}{r_0 k (f \cdot H)}, \frac{d_0 k+1}{r_0 k (f \cdot H)}$,
then the claim also holds for $\lambda=\frac{d_0}{r_0 (f \cdot H)}$.
\begin{NB}
$s(1,\eta_1,b_1)+t(1,\eta_2,b_2) \sim
(1,\frac{s \eta_1+t \eta_2}{s+t},\frac{s b_1+t b_2}{s+t})$.
\end{NB}
Hence we shall prove the claim under the assumption $\gcd(r_0 (f \cdot H),d_0)=1$.
We set $Y=M_H^\eta(0,r_0 f,d_0)=M_H^{\eta_0}(0,r_0 f,d_0)$. We take a general $\eta_0$.
Then $Y$ is a fine moduli space and we have a Fourier-Mukai transform
$\Phi_{X \to Y}^{{\cal P}^{\vee}[1]}:{\bf D}(X) \to {\bf D}(Y)$ which induces an isomorphism
${\cal M}^\lambda(r,\xi,a)^{ss} \cong {\cal M}_{H_f}(r',\xi',a')^{ss}$.

We set $G:=(\Phi_{X \to Y}^{{\cal P}^{\vee}[1]}(F))$.
By \cite[Lem. 3.2.1]{PerverseII}, 
$\rk (G)=0$ and
$c_1(G)$
is $\pi$-ample. 
By Remark \ref{rem:moduli} and Lemma \ref{lem:ample}, 
$\theta_{{\bf e}'}(G)$ is relatively ample over $S^l C$, where
${\bf e}'=(r',\xi',a')$. Then the claim holds for a general $\eta_0$.
Hence the claim holds.
\end{proof}

\begin{NB}
For the moduli space $M_{H_f}({\bf e})$, 
$-(1,\eta,b)+k(0,H,*)$ is ample if $k \gg 0$.
If all fibers of $\pi$ are irreducible, then
$\theta_{\bf e}(F)$ is ample if $(\eta \cdot f) \ll 0$.
\end{NB}

\begin{NB}
If $E$ is a properly $\eta$-twisted semi-stable sheaf with
$\tau(E)=(0,r_0 f,d_0)$,
then there is a subsheaf $F$ with $\chi_\eta(F)=0$.
Since $\chi_\eta(F)=\chi(F)-(c_1(F) \cdot \eta)$,
$(c_1(F) \cdot \eta) \in {\Bbb Z}$.
\begin{rem}
Let $C_{ij}$ be effective divisors in the proof of Proposition \ref{prop:finite}. 
Then $D=D_1+n_i f_i$, $D_1 \in L$.
\end{rem}
\end{NB}

\begin{lem}\label{lem:ample}
Assume that $M_{H_f}({\bf e})$ consists of $\mu$-stable vector bundles.
Then $\theta_{\bf e}(G)$ is ample, where $G \in K(X)_{\bf e}$ satisfies
$\rk G=0$ and
$c_1(G) \in {\Bbb Z}_{>0} H_f$, where $H_f=H+nf$ $(n \gg 0)$. 
\end{lem}

\begin{proof}
For ${\bf e}=(r,\xi,a)$, 
we note that $E \in M_{H_f}({\bf e})$ if and only if 
$E^{\vee} \in M_{H_f}({\bf e}')$, where ${\bf e}'=\tau(E^{\vee})$.
Hence $D_X$ induces an isomorphism 
$M_{H_f}({\bf e}) \to M_{H_f}({\bf e}')$.
Under this isomorphism, we have
$c_1(\theta_{{\bf e}}(D^*(\alpha)))=-c_1(\theta_{{\bf e}'}(\alpha))$
for $\alpha \in K(X)_{{\bf e}'}$.
(see \eqref{eq:D}).
For $\alpha \in K(X)_{\bf e} \otimes {\Bbb Q}$,
we can define $c_1(\theta_{\bf e}(\alpha))$ as a rational cohomology class
in $H^2(M_{H_f}({\bf e}),{\Bbb Q})$.
For $F(\beta,n) \in K(X)_{\bf e}$ with
$$
\ch(F(\beta,n))=(-1,nH_f+\beta,b), b \in {\Bbb Q},
$$
$c_1(\theta_{\bf e}(F(\beta,n)))$ is ample for $n \gg 0$
(\cite[A.2]{MYY2}).
We note that
$$
\ch(D^*(F(-\beta+K_X,n)))=(-1,-(nH_f-\beta),b')=\ch(F(\beta,-n)), b' \in {\Bbb Q}.
$$
Since $c_1(\theta_{{\bf e}'}(F(-\beta+K_X,n)))=c_1(\theta_{\bf e}(-F(\beta,-n)))$
$(n \gg 0)$ is ample and
$$
\ch(F(\beta,n)-F(\beta,-n))=(0,2nH_f,b-b'),
$$
we get our claim.
\end{proof}

\subsection{A special case}
In this subsection, 
we assume that $\pi:X \to C$ is an elliptic surface such that 
$\NS(X)={\Bbb Z}H+{\Bbb Z}f$ and there is no multiple fiber.
We assume that $\alpha=0$.
As in the previous subsections, we define $l \in {\Bbb Z}$ 
by $\dim M^\lambda({\bf e})=2l+g$ and assume that $l \geq 2$.
For a chamber $I=(\lambda_1,\lambda_2)$,
we set $M^I({\bf e}):=M^\lambda({\bf e})$ ($\lambda \in I$).
\begin{prop}\label{prop:example1}
Let $I=(\lambda_1,\lambda_2)$ be a chamber and $F \in (K(X)_{\bf e}) \otimes {\Bbb Q}$
satisfy $\tau(F)=-(1,\lambda H,b)$. 
\begin{enumerate}
\item[(1)]
If $\lambda \in I$,
then $c_1(\theta_{\bf e}(F)) \in \NS(M^I({\bf e}))_{\Bbb Q}$ is relatively ample over $S^l C$.
\item[(2)]
If $\lambda=\lambda_1,\lambda_2$, then 
$c_1(\theta_{\bf e}(F)) \in \NS(M^I({\bf e}))_{\Bbb Q}$ gives a contraction over $S^l C$.
\end{enumerate}
\end{prop}

\begin{proof}
(1) In the notation of Proposition \ref{prop:rel-ample},
we assume that $\eta=\lambda H$.
Then \eqref{eq:rel-ample} means $\lambda_1<\lambda<\lambda_2$.
Hence the claim holds.

(2) 
We only treat the case where $\lambda=\lambda_1$.
Assume that the wall $\lambda=\lambda_1$ is 
defined by $\tau=(0,r_0 f,d_0)$, where $\gcd(r_0,d_0)=1$.
We first assume that $-\chi({\bf e}-\tau,\tau) \geq 3$.
By our assumption, $\chi(\tau(F),\tau)=r_0 \lambda(H \cdot f)-d_0=0$.
We take $\lambda_+$ satisfying $\lambda_1<\lambda_+<\lambda_2$.
For $E_1 \in M_H(\tau)$ and a $\lambda_1$-stable sheaf $E_2$ with $\tau(E_2)=
{\bf e}-\tau$, we set
$P:={\Bbb P}(\Ext^1(E_2,E_1)^{\vee})$.
Let ${\cal E}$ be a family of $\lambda_+$-stable sheaves on $P \times X$ fitting in an exact 
sequence
$$
0 \to E_1 \boxtimes {\cal O}_P(1) \to {\cal E} \to E_2 \boxtimes {\cal O}_P \to 0.
$$
\begin{NB}
Then $\theta_{\bf e}(F)_{|P} \cong {\cal O}_P$.
Since $\theta_{\bf e}(F)$ is relatively nef and big over $S^l C$ and
the canonical divisor is the pull-back of an ample divisor on $S^l C$,
the base point free theorem implies
$\theta_{\bf e}(F)$ gives a contraction over $S^l C$.
\end{NB}
Then $c_1(\theta_{\bf e}(F))_{|P}=\chi(\tau(F),\tau)c_1({\cal O}_P(1))=0$.
\begin{NB}
Consistency check:
$\chi(\tau(F),{\bf e})=\chi(\tau(F),\tau)=r_0 \lambda(H \cdot f)-d_0>0$ iff
\eqref{eq:rel-ample} holds.
\end{NB}
Since $c_1(\theta_{\bf e}(F))$ is relatively nef and big over $S^l C$ and
the canonical divisor is the pull-back of a divisor on $S^l C$,
the base point free theorem implies
$c_1(\theta_{\bf e}(F))$ gives a small contraction over $S^l C$
(see the proof of Proposition \ref{prop:birat}).
We next assume that $-\chi({\bf e}-\tau,\tau) =1,2$.
In these cases, we can also show that
$c_1(\theta_{\bf e}(F))$ gives the Hilbert-Chow contraction or the Gieseker-Uhlenbeck contraction.
Thus $c_1(\theta_{\bf e}(F))$ is in the boundary of the relative Movable cone $\Mov(M^I({\bf e})/S^l C)$
over $S^l C$.
\end{proof}

\begin{NB}
Assume that $-\chi({\bf e}-\tau,\tau)=2$.
Let $D$ be the effective divisor consisting of properly
$\lambda_1$-semi-stable sheaves. 
For $F \in K(X)_{\bf e}$ such that
$\tau(F) \equiv (r,\xi,b)$,
$c_1(\theta_{\bf e}(F)) \equiv D \mod \NS(S^l C \times \Pic^0(X))$.
Proof:
By a suitable Fourier-Mukai transform,
we can reduced to the rank 2 case.
We explain the twisted case.
For a universal family ${\cal F}$ on the moduli $M$ of twisted sheaves of rank 2
and a fixed determinant,
we take a locally free resolution $0 \to V_1 \to V_0 \to {\cal F} \to 0$. 
For a twisted locally free sheaf $G$ on $X$ such that $\det {\cal F}=\det G \otimes {\cal L}$,
${\cal L} \in \Pic(M)$. 
$\det p_{M!}({\cal E}xt^1({\cal F},G)) \cong {\cal O}((\rk G) D)$.
Since  $p_{M!}({\cal E}xt^1({\cal F},G)) \cong
p_{M!}(G,{\cal F}(K_X))+p_{M!}(G,{\cal F} \otimes {\cal L})=
p_{M!}(G,{\cal F}(K_X))+p_{M!}(G,{\cal F}) \otimes {\cal L}^{\otimes -\chi(G,E)}$,
$(\rk G) D \equiv \theta_{\bf e}(G(-K_X)+G-\chi(G,E){\Bbb C}_x)$.
Since $\rk G=2$ and $\tau(G(-K_X)+G) \equiv 2G \mod K(C)_{\Bbb Q}$,
we get the claim.

For $F' \in (K(X)_{\bf e})_{\Bbb Q}$ with $\tau(F')=-(1,\eta,b)$ such that
$\lambda_1 (H \cdot f)+(\alpha \cdot f)=(\eta \cdot f)$,
$\theta_{\bf e}(nF'-F)$ is ample for $n \gg 0$.
Indeed $r_1(\xi \cdot f)-r d_1=2$ implies 
$\frac{d_1}{r_1(H \cdot f)}<\frac{(\xi \cdot f)}{r(H \cdot f)}$.
Hence $\theta_{\bf e}(F')$ is big. Since it is nef, it defines a contraction over $S^l C$.
Since $\gcd(r,(\xi \cdot f))=1$, $D$ is a primitive divisor.

We note that $(K_D)_{|P} \cong K_P$.
Since $K_M$ is trivial over $S^l C$, ${\cal O}(D)_{|P} \cong K_P={\cal O}_P(-2)$.
\end{NB}

We next restrict to the case where ${\bf e}=(1,0,e-l)$ and
describe the chamber structure for $\lambda$-stability.
We set $t:=(H \cdot f) \lambda$ and describe the space of $t \in (-\infty,0)$.
We set 
$$
I_n:=
\begin{cases}
(-\frac{2}{n},-\frac{2}{n+1}) & n \geq 1\\
(-\infty,-2) & n=0.
\end{cases}
$$
Then we have a decomposition
\begin{equation}
(-\infty,0) \setminus \{\tfrac{-2}{n} \mid n \in {\Bbb Z}_{>0} \}=\cup_{n=0}^\infty I_n.
\end{equation}
We shall consider a relation of the intervals $I_n$.
%
%
We set $Y=M_H(0,f,-1)$. Then $Y \cong X$.
Let ${\bf P}$ be a universal family such that
$\Phi_{X \to Y}^{{\bf P}^{\vee}[1]}({\cal O}_X)={\cal O}_Y$.
Let $E_{p,q}$ be a stable 1-dimensional sheaf on a fiber with
$\tau(E_{p,q})=(0,pf,q)$. Then 
$\Phi_{X \to Y}^{{\bf P}^{\vee}[1]}(E_{p,q})=E_{p+q,q}$.
\begin{NB}
By $\Phi_{X \to Y}^{{\bf P}^{\vee}[1]}({\cal O}_X)={\cal O}_Y$,
$\Phi_{X \to Y}^{{\bf P}^{\vee}[1]}(E_{1,0})=E_{1,0}$.
By $\Phi_{X \to Y}^{{\bf P}^{\vee}[1]}(E_{1,-1})=E_{0,1}[-1]$,
$$
A 
\begin{pmatrix} 
1& 1\\
0 & -1
\end{pmatrix}=
\begin{pmatrix} 
1& 0\\
0 & -1
\end{pmatrix}.
$$
Hence we get $\Phi_{X \to Y}^{{\bf P}^{\vee}[1]}(E_{p,q})=E_{p+q,q}$.
\end{NB}
For $t$, we set
$\varphi(t):=\frac{t}{1+t}$.
Then we get
$\varphi(I_n)=I_{n-2}$ for $n \geq 2$ and
$$
(-\infty,0) \setminus \{ -\tfrac{2}{n} \mid n \in {\Bbb Z}_{>0} \}
=\cup_{k=0}^{\infty} \varphi^{-k} (I_0 \cup I_1).
$$

\begin{prop}
\begin{enumerate}
\item[(1)]
The isomorphism 
classes of $M^\lambda({\bf e})$ are parameterized by 
$I_0$.
\begin{NB}
For $t \in I_n$, 
\end{NB}
\item[(2)]
The set of walls in $I_0$ is given by
$$
\{\tfrac{q}{p} \mid p,q \in {\Bbb Z}, \gcd(p,q)=1, 0<2p <-q \leq l \}.
$$
\end{enumerate}
\end{prop}

\begin{proof}
(1) 
Assume that $t>-1$.
By Proposition \ref{prop:n-stability2}, 
$\Phi_{X \to Y}^{{\cal P}^{\vee}[1]}$ induces an isomorphism
$$
M^{\frac{t}{(H \cdot f)}}({\bf e}) \to 
M^{\frac{\varphi(t)}{(H \cdot f)}}({\bf e}).
$$
For a general $t \in I_1$, $\frac{-t}{1+t} \in I_0$
and 
$D_Y \circ \Phi_{X \to Y}^{{\cal P}^{\vee}[1]}$ induces an isomorphism
 $$
M^{\frac{t}{(H \cdot f)}}({\bf e}) \to 
M^{-\frac{\varphi(t)}{(H \cdot f)}}({\bf e})
$$
by Proposition \ref{prop:n-stability3}. 
By $\varphi(I_n)=I_{n-2}$ $(n \geq 2)$, we get (1).
For the proof of (2), let $p,q$ be integers satisfying
$p>0$ and  $\gcd(p,q)=1$.
Then $\tau=(0,pf,q)$ defines a wall for ${\bf e}$ in $I_0$
if $-l \leq \frac{q}{p}<-2$ and $l \geq -q$ by
the non-emptyness of $M^\lambda({\bf e}-\tau)$. 
Thus $2p<-q \leq lp$ and $-q \leq l$.
\end{proof}

\begin{rem}
By the definition of $\lambda$-stability,
we see that there is a torsion free $\lambda$-stable sheaf
if and only if $\lambda<-1$. 
\end{rem}

\begin{NB}
For the Hilbert-Chow contraction, the exceptional divisor is divisible by 2.
On the other hand for the Gieseker-Uhlenbeck contraction,
the exceptional divisor is primitive.
Hence all birational morphisms preserves the the relative movable cone.

$\lambda=\frac{t}{(H \cdot f)}$.
$\lambda_i=\frac{q_i}{p_i (H \cdot f)}$.
\end{NB}
We further assume that the Kodaira dimension of $X$ is $1$ and$g=0$.
Then the relative Picard rank
$\rho(M^{\tfrac{t}{(H \cdot f)}}({\bf e})/S^l {\Bbb P}^1)$ is $2$ for a general $t \in (-\infty,0)$.
\begin{NB}
For a chamber $I=(t_1,t_2)$, we set $\overline{I}:=[t_0,t_1]$.
\end{NB}
For $t <0$,
we set $\xi_t:=c_1(\theta_{\bf e}(F_t))$, where
$$
F_t=\tfrac{(H \cdot f)}{t}(1,0,l)+(0,H,(H \cdot K_X)) \in (K(X)_{\bf e})_{\Bbb R}.
$$
We note that $\xi_{-\infty}$ gives the Hilbert-Chow contraction
$\Hilb_X^l \to S^l X$.

Let $\Mov(M^I({\bf e})/S^l {\Bbb P}^1)$ be the relative Movable cone of
$M^I({\bf e}) \to S^l {\Bbb P}^1$ and
$\Nef(M^I({\bf e})/S^l {\Bbb P}^1)$ the relative Nef cone.
By Proposition \ref{prop:example1}, we get the following. 
\begin{prop} 
Let $I=(t_1,t_2)$ be a chamber in $I_0$.
We set $M^I({\bf e}):=M^{\frac{t}{(H \cdot f)}}({\bf e})$ $(t \in I)$.
\begin{enumerate}
\item[(1)]
$\Mov(M^I({\bf e})/S^l {\Bbb P}^1)$ is generated by
$\xi_{-\infty}$ and $\xi_{-2}$.
\item[(2)]
$\Nef(M^I({\bf e})/S^l {\Bbb P}^1)$ is generated by
$\xi_{t_1}$ and $\xi_{t_2}$. 
\begin{NB}
$\xi_t$ $(t \in \overline{I})$.
\end{NB}
\end{enumerate}
\end{prop}

Thus each chamber describe the isomorphism classes of $M^\lambda({\bf e})$. 

\begin{NB}
If $\chi(F)=(1,\lambda H,a)$, then 
$\chi(F,I_Z)=a+e-l+\lambda(H \cdot K_X)/2$.
If $a=l-e-\lambda(H \cdot K_X)/2$, then
$\chi(F)=a-\lambda (H \cdot K_X)/2+e=l-\lambda (H \cdot K_X)$.
Hence $\tau(F)=(1,\lambda H,l-\lambda (H \cdot K_X))=(1,0,l)+\lambda (0,H,-(H \cdot K_X))$.
For $F_0$ with $\tau(F_0)=(1,0,l)$,
$\ch_2(F_0)+e=l$. Hence $\ch_2(F_0)=l-e$ and
$\chi(F_0,I_Z)=l-e+e-l=0$.
\end{NB}

\subsection{Hodge numbers}

The following result is proved for an elliptic surface
with a section such that all fibers are irreducible in \cite{Y:twist2}.

\begin{thm}\label{thm:hodge}
Assume that $\gcd(r,(\xi \cdot f))=1$. Then
\begin{equation}\label{eq:hodge}
e(M_{H_f}(r,\xi,a))=e(\Hilb_X^l \times \Pic^0(X)),
\end{equation}
where 
$l=(\xi^2)-2ra+(r^2+1)\chi({\cal O}_X)-r(\xi \cdot K_X)$.
\end{thm}

\begin{proof}
We note that $\dim M_{H_f}(r,\xi,a)=2l+q$ (cf. Proposition \ref{prop:lambda-smooth}). 
For $Y:=M_H^\alpha(0,r_1 f,d_1)$ with $r_1(\xi \cdot f)-rd_1=1$,
let us consider $\Phi_{X \to Y}^{{\cal P}^{\vee}}:{\bf D}(X) \to {\bf D}(Y)$
in Theorem \ref{thm:birat2}. Replacing the universal family ${\cal P}$
by ${\cal P} \otimes p_Y^*(L)$ ($L \in \Pic(Y)$),
we may assume that $c_1(\Phi_{X \to Y}^{{\cal P}^{\vee}[1]}(E))=0$
for $E \in M_{H_f}(r,\xi,a)$. 
By Proposition \ref{prop:n-stability}, there is $\lambda$ and we have an isomorphism
$$
M^\lambda(r,\xi,a) \cong \Hilb_Y^l \times \Pic^0(Y).
$$
We can easily prove that
$e({\cal M}^\lambda(r,\xi,a)^{ss})$ is independent of the choice of 
a general $\lambda$.
Indeed by the classification of walls, the proof is the same as in \cite[Prop. 3.15]{Y:twist2}.
Since $e(X)=e(Y)$, we see that
$$
e(\Hilb_Y^l \times \Pic^0(Y))=e(\Hilb_X^l \times \Pic^0(X))
$$
by \cite{GS}.
Therefore 
$$
e(M_{H_f}(r,\xi,a))=e(M^\lambda(r,\xi,a))=e(\Hilb_Y^l \times \Pic^0(Y))=
e(\Hilb_X^l \times \Pic^0(X)).
$$
\end{proof}

\begin{rem}
Assume that $K_X$ is nef, that is, $K_X =c f$, $c \in {\Bbb Q}_{\geq 0}$.
Then $K_{M_{H_f}(r,\xi,a)}$ is nef.
By the invariance of Hodge numbers of minimal models
\cite{W}, \eqref{eq:hodge} also follows.
\end{rem}

\section{Appendix}

\subsection{A category related to $\lambda$-stability.}

We give a remark on the $\lambda$-stability. 
We set $\lambda:=\frac{d_1-(\alpha \cdot r_1 f)}{(r_1 f \cdot H)}$.
\begin{defn}
\begin{enumerate}
\item[(1)]
Let ${\cal S}_\lambda$ be the full subcategory of $\Coh(X)$ 
generated by $\alpha$-twisted stable sheaves $S$
on a fiber 
such that 
$\frac{\chi_\alpha(S)}{(c_1(S) \cdot H)} \leq \lambda$.
\item[(2)]
Let ${\cal T}_\lambda$ be the full subcategory of $\Coh(X)$ whose object $E$ satisfies
$\Hom(E, S)=0$ for all $S \in {\cal S}_\lambda$.
\item[(3)]
Let ${\cal F}_\lambda$ be the full subcategory of $\Coh(X)$ whose objects $E$ 
fits in an exact sequence
$$
0 \to S \to E \to E_2 \to 0
$$
where $S \in {\cal S}_\lambda$ and $E_2$ 
is a torsion free sheaf such that
$(E_2)_{\eta}$ is generated by $\mu$-semi-stable vector bundles
$F_\eta$ with $\deg F_\eta/\rk F_\eta \leq d_1/r_1$.
\end{enumerate}
\end{defn}

\begin{NB}
For $E$, we have an exact sequence
$$
0 \to E_1 \to E \to E_2 \to 0
$$
such that $(E_1)_{\eta}$ is generated by $\mu$-semi-stable vector bundles
$F_\eta$ with $\deg F_\eta/\rk F_\eta>d/r$ and
$E_2$ is a torsion free sheaf such that
$E_2 \in {\cal F}$.

For $E_1$, we have an exact sequence
$$
0 \to E_1' \to E_1 \to S \to 0 
$$
such that $S \in {\cal S}$ and $\Hom(E_1',S)=0$ for all $S \in {\cal S}$.
Then we have an exact sequence
$$
0 \to E_1' \to E \to E/E_1' \to 0
$$
where $E/E_1'$ is an extension of $E_2$ by $S$.

If $(E_1)_1$......
\end{NB}

Let $Y:=M_H^\alpha(0,r_1 f,d_1)$ be a fine moduli space
and ${\cal P}$ a universal family on $X \times Y$,
where $(\alpha \cdot f)=0$ and 
$(H,\alpha)$ is general.
Let $\Phi_{X \to Y}^{{\cal P}^{\vee}}:{\bf D}(X) \to {\bf D}(Y)$ be 
the Fourier-Mukai transform by ${\cal P}$.

\begin{prop}
We set $\lambda:=\frac{\chi_\alpha({\cal P}_{|X \times \{ y\}})}{(r_1 f \cdot H)}=\frac{d_1}{r_1(f \cdot H)}$.
Then $({\cal T}_\lambda,{\cal F}_\lambda)$ is a torsion pair.
Let ${\cal A}_\lambda$ be the tilting.
Then
$\Phi_{Y \to X}^{{\cal P}[1]}(\Coh(Y))={\cal A}_\lambda$.
\end{prop}

\begin{proof}
For $F \in \Coh(Y)$ and an $\alpha$-twisted stable sheaf $A$ on a fiber 
such that $A \in {\cal S}_\lambda$,
Lemma \ref{lem:A} implies that
$$
\Hom(\Phi_{Y \to X}^{{\cal P}[1]}(F),A)=\Hom(F,\Phi_{X \to Y}^{{\cal P}^{\vee}}(A(K_X))[1])=0.
$$
For an $\alpha$-twisted stable sheaf $A$ on a fiber such that
$\frac{\chi_\alpha(A)}{(c_1(A) \cdot H)} > \frac{\chi_\alpha({\cal P}_{|X \times \{ y\}})}{(r_1 f \cdot H)}$,
$$
\Hom(A,\Phi_{Y \to X}^{{\cal P}}(F))=\Hom(\Phi_{X \to Y}^{{\cal P}^{\vee}}(A)[1]),F[-1])=0.
$$
We also note that 
$H^i(\Phi_{Y \to X}^{{\cal P}[1]}(F))=0$ for $i \ne -1,0$.
By using these facts and \cite[sect. 6.2]{Br:1}, we shall prove that $({\cal T},{\cal F})$ is a
torsion pair and ${\cal A}=\Phi_{Y \to X}^{{\cal P}[1]}(\Coh(Y))$.

For $E \in \Coh(X)$, \cite[sect. 6.2]{Br:1} implies
\begin{equation}
\begin{split}
& H^i(\Phi_{X \to Y}^{{\cal P}^{\vee}[1]}(E(K_X))) =0,\; i \ne 0,1,\\ 
& E_1:=\Phi_{Y \to X}^{{\cal P}[1]}(H^0(\Phi_{X \to Y}^{{\cal P}^{\vee}[1]}(E(K_X)))) \in \Coh(X),\\
& E_2:=\Phi_{Y \to X}^{{\cal P}}(H^1(\Phi_{X \to Y}^{{\cal P}^{\vee}[1]}(E))) \in \Coh(X).
\end{split}
\end{equation}
We also have an exact sequence
$$
0 \to E_1 \to E \to E_2 \to 0.
$$
Then by using \cite[Lem. 6.2]{Br:1}, we see that 
$$
E_1 \in {\cal T}_\lambda,\;
E_2 \in {\cal F}_\lambda.
$$
Thus we have a desired decomposition of $E$. We also get 
$\Phi_{Y \to X}^{{\cal P}[1]}(\Coh(Y))={\cal A}_\lambda$.
\end{proof}

For an object $E \in {\cal A}_\lambda \cap \Coh(X)$,
the condition Definition \ref{defn:n-stable} (3)
implies $E$ is an torsion free object of ${\cal A}_\lambda$.

\begin{NB}
A general case:
For $E \in \Coh(X)$,
we take $\lambda'>\lambda>\lambda''$ which are sufficiently close to $\lambda$.
Then ${\cal S}_{\lambda''} \subset {\cal S}_\lambda \subset {\cal S}_{\lambda'}$ and
${\cal T}_{\lambda''} \supset {\cal T}_\lambda \supset {\cal T}_{\lambda'}$.
We have a decompositions
$$
0 \to F' \to E \to G' \to 0
$$
and 
$$
0 \to F'' \to E \to G'' \to 0
$$
such that $F' \in {\cal T}_{\lambda'}$, $G' \in {\cal F}_{\lambda'}$,
$F'' \in {\cal T}_{\lambda''}$ and $G'' \in {\cal F}_{\lambda''}$.

Since $F' \subset F''$ and $F'' \in {\cal T}_{\lambda''}$,
$F''/F' \in {\cal F}_{\lambda''}$.
\end{NB}

\begin{NB}
For a torsion free sheaf $F$ on $Y$ such that
$F_\eta$ is semi-stable, assume that
$\gcd((c_1(F) \cdot f)/m,\rk F)=1$, where
$m{\Bbb Z}=\{(D \cdot f) \mid D \in \NS(X) \}$.
Then $F$ does not have a subsheaf $F_1$ such that
$(F_1)_\eta$ is a semi-stable vector bundle with
$\deg F_\eta/\rk F_\eta=\deg (F_1)_\eta/\rk (F_1)_\eta$.
Then by the proposition,
we have a similar description of $\Phi_{Y \to X}^{{\cal P}[1]}(F) \in {\cal A}_\lambda$. 
\end{NB}

\begin{NB}
Definition \ref{defn:n-stable} (3) means that
$E \in {\cal A}_\lambda$ is torsion free.
\end{NB}

\begin{NB}
In general, the stability is not preserved.

Assume that there is a section $\sigma$.
$\ch(E)=(r,d \sigma+(c+ed)f,s) \mapsto \ch(\Phi(E))/d=(1,-r/d \sigma-s/d f+...,...)$.

\end{NB}

\subsection{Fourier-Mukai transforms and the $\lambda$-stability.}\label{subsect:n-stable2}

In this subsection, we shall refine Proposition \ref{prop:n-stability}.
Assume that $(\alpha \cdot f)=0$.
Let $\Phi_{X \to Y}^{{\cal P}^{\vee}}:{\bf D}(X) \to {\bf D}(Y)$ be the
Fourier-Mukai transform in subsection \ref{subsect:relFM}. 
\begin{NB}
As in \ref{subsect:relFM},
we set $Y:=M_H^\alpha(0,r_1 f,d_1)$ and assume that
$Y$ is a smooth projective surface and
there is a universal family ${\cal P}$ on $X \times Y$.
\end{NB}
As we explained in subsection \ref{subsect:relFM},
we take a locally free sheaf $G$ such that
$\frac{c_1(G)}{\rk G}=\alpha+\frac{d_1 H}{r_1 (f \cdot H)}$.
Then $\chi(G,{\cal P}_{|X \times \{ y\}})=0$ ($y \in Y$).
We set $\tau({\cal P}^{\vee}_{|\{ x \} \times Y}[1])=(0,r_1' f, d_1')$.
We also define a locally free sheaf 
$G':=\Phi_{X \to Y}^{{\cal P}^{\vee}}({\cal O}_H)[1]$ and
set $H':=c_1(\Phi_{X \to Y}^{{\cal P}^{\vee}}(G)[2])$.
We set $\alpha':=\frac{c_1(G')}{\rk G'}-\frac{d_1' H'}{r_1' (f \cdot H')}$.
Since $\frac{(c_1(G') \cdot f)}{\rk G'}=\frac{d_1'}{r_1'}$, $(\alpha' \cdot f)=0$.
\begin{NB}
We assume that there is an integral curve $C \in |H|$, and we set $L:={\cal O}_C$.
Then $G':=\Phi_{X \to Y}^{{\cal P}^{\vee}}(L)[1]$ is a locally free sheaf on $Y$.
We set $L':=\Phi_{X \to Y}^{{\cal P}^{\vee}}(G)[2]$.
Then $L'$ is a purely 1-dimensional sheaf on $Y$.
By \cite[Lem. 3.2.1]{PerverseII}, $\widehat{H}:=c_1(L')$ is $\pi'$-ample,
where $\pi':Y \to C$ be the elliptic fibration. 
By \cite[Thm. 3.2.8]{PerverseII},
${\cal P}^{\vee}_{|\{ x \} \times Y}[1]$ $(x \in X)$ 
is $G'$-twisted stable with respect to $\widehat{H}+nf$
$(n \gg 0)$.
For a 1-dimensional sheaf $A$ on $X$,
\begin{equation}
\begin{split}
\chi(G',\Phi_{X \to Y}^{{\cal P}^{\vee}}(A))=&(c_1(L) \cdot c_1(A))\\
(c_1(L') \cdot \Phi_{X \to Y}^{{\cal P}^\vee}(A))=& -\chi(G,A)
\end{split}
\end{equation}
\end{NB}


For $\lambda_0=\frac{d_0-(\alpha \cdot r_0 f)}{(r_0 f \cdot H)}=\frac{d_0}{r_0 (f \cdot H)}$
such that $r_0 \in {\Bbb Z}_{>0}$, $d_0 \in {\Bbb Z}$ and $\gcd(r_0,d_0)=1$, we take  
an $\alpha$-twisted stable sheaf $P$ with $\tau(P)=(0,r_0 f,d_0)$.
We set $\varphi(\lambda_0):=\frac{\chi_{\alpha'}(P')}{(c_1(P') \cdot H')}=
\frac{\chi(G',P')}{\rk G' (c_1(P') \cdot H')}-\frac{d_1'}{r_1'(f \cdot H')}$,
where $P'=\Phi_{X \to Y}^{{\cal P}^{\vee}[1]}(P)$ and
$\tau(P')=(0,r'_0 f,d_0')$.
We regard $\lambda \in (-\infty,\infty)$ as an element of ${\Bbb P}_{\Bbb R}^1 
\cong S^1$.
Then $\varphi$ is extened to an isomorphism of ${\Bbb P}_{\Bbb R}^1$, where
$\varphi(\infty)=
\frac{\chi_{\alpha'}({\cal P}^{\vee}_{|\{x \} \times Y}[1])}{(c_1({\cal P}^{\vee}_{|\{x \} \times Y}[1])) \cdot H')}=\frac{d_1'}{r_1' (f \cdot H')}$.

\begin{prop}\label{prop:n-stability2}
Assume that $\gcd((\xi \cdot f),r)=1$ and $r_1(\xi \cdot f)-rd_1>0$.
We assume that
$\lambda_0=\frac{d_0-(\alpha \cdot r_0 f)}{(r_0 f \cdot H)}
>\lambda_1=\frac{d_1-(\alpha \cdot r_1 f)}{(r_1 f \cdot H)}$.
Then for a $\lambda_0$-stable sheaf $E$, 
$E':=\Phi_{X \to Y}^{{\cal P}^{\vee}[1]}(E)$ is $\varphi(\lambda_0)$-stable. 
Thus we have an isomorphism
$$
\Phi_{X \to Y}^{{\cal P}^{\vee}[1]}:
{\cal M}^{\lambda_0}({\bf e})^s \to {\cal M}^{\varphi(\lambda_0)}({\bf e}')^s,
$$
where ${\bf e}'=\tau(E')$.
\end{prop}

\begin{proof}
Let $A'$ be an $\alpha'$-twisted stable sheaf on a fiber.
We shall prove the following.
\begin{enumerate}
\item[(i)]
If $\frac{\chi_{\alpha'}(A')}{(c_1(A') \cdot H')}>\varphi(\lambda_0)$, then
$\Hom(A',E')=0$.
\item[(ii)]
If $\frac{\chi_{\alpha'}(A')}{(c_1(A') \cdot H')} \leq \varphi(\lambda_0)$, then
$\Hom(E',A')=0$.
\end{enumerate}

(i)
If 
$\frac{\chi_{\alpha'}(A')}{(c_1(A') \cdot H')}>\varphi(\infty)=\frac{d_1'}{r_1' (f \cdot H')}$, 
then there is an $\alpha$-twisted stable 1-dimensional sheaf $A$ such that
$A'=\Phi_{X \to Y}^{{\cal P}^{\vee}[1]}(A[1])$ and
$\frac{\chi_{\alpha}(A)}{(c_1(A) \cdot H)}<\lambda_1$ (see Lemma \ref{lem:G-stability}
and Remark \ref{rem:G-stability}).
Hence $\Hom(A',E')=\Hom(A[1],E)=0$.
If 
$\varphi(\infty) > \frac{\chi_{\alpha'}(A')}{(c_1(A') \cdot H')}>\varphi(\lambda_0)$, 
then there is an $\alpha$-twisted stable 1-dimensional sheaf $A$ such that
$A'=\Phi_{X \to Y}^{{\cal P}^{\vee}[1]}(A)$ and
$\frac{\chi_{\alpha}(A)}{(c_1(A) \cdot H)}>\lambda_0$.
Hence $\Hom(A',E')=\Hom(A,E)=0$.
If $\varphi(\infty)=\frac{\chi_{\alpha'}(A')}{(c_1(A') \cdot H')}$, then
$\Hom(A',E')=\Hom({\Bbb C}_x,E)=0$.

(ii)
If 
$\frac{\chi_{\alpha'}(A')}{(c_1(A') \cdot H')} \leq \varphi(\lambda_0)$, then
 there is an $\alpha$-twisted stable 1-dimensional sheaf $A$ such that
$A'=\Phi_{X \to Y}^{{\cal P}^{\vee}[1]}(A)$ and
$\frac{\chi_{\alpha}(A)}{(c_1(A) \cdot H)} \leq \lambda_0$.
Hence $\Hom(E',A')=\Hom(E,A)=0$.
\end{proof}

We also have the following result.
\begin{prop}\label{prop:n-stability3}
We set
$\psi(\lambda_0):=\frac{\chi_{-\alpha'}({P'}^{\vee}[1])}{(c_1({P'}^{\vee}[1]) \cdot H')}
=-\frac{\chi_{\alpha'}(P')}{(c_1(P') \cdot H')}$.
Assume that $\lambda_0<\lambda_1$. Then 
for a $\lambda_0$-stable sheaf $E$, 
$E'':=\Phi_{X \to Y}^{{\cal P}^{\vee}[1]}(E)^{\vee}$ 
satisfies the following.
\begin{enumerate}
\item
For a $(-\alpha')$-stable fiber sheaf $A''$ 
with
$\psi(\lambda_0) \leq \frac{\chi_{-\alpha'}(A'')}{(c_1(A'') \cdot H)}$,
$\Hom(A'',E'')=0$.
\item
For a $(-\alpha')$-stable fiber sheaf $A''$ with
$\psi(\lambda_0) > \frac{\chi_{-\alpha}(A'')}{(c_1(A'') \cdot H)}$,
$\Hom(E'',A'')=0$.
\end{enumerate}
In particular if $\lambda_0$ is general, then
$E''$ is $\psi(\lambda_0)$-stable.
\end{prop}

\begin{proof}
(i) Assume that $\psi(\lambda_0) \leq \frac{\chi_{-\alpha'}(A'')}{(c_1(A'') \cdot H)}<
\psi(\infty)=\frac{-d_1'+(\alpha' \cdot r_1' f)}{r_1' (f \cdot H)}$. Then 
$A''=\Phi_{X \to Y}^{{\cal P}^{\vee}[2]}(A)^{\vee}[1]$ and 
$\frac{\chi_\alpha(A)}{(c_1(A) \cdot H)} \leq \lambda_0$. Hence
$\Hom(A'',E'')=\Hom(E,A)=0$.
Assume that $\psi(\infty) \leq \frac{\chi_{-\alpha'}(A'')}{(c_1(A'') \cdot H)}$. Then 
$A''=\Phi_{X \to Y}^{{\cal P}^{\vee}[1]}(A)^{\vee}[1]$. Hence
$\Hom(A'',E'')=\Hom(E[1],A)=0$.

(ii) 
Assume that $\psi(\lambda_0) > \frac{\chi_{-\alpha}(A'')}{(c_1(A'') \cdot H)}$.
Then $A''=\Phi_{X \to Y}^{{\cal P}^{\vee}[2]}(A)^{\vee}[1]$ and
$\frac{\chi_\alpha(A)}{(c_1(A) \cdot H)} > \lambda_0$.
Hence
$\Hom(E'',A'')=\Hom(A,E)=0$.
\end{proof}

\end{document}